\documentclass[10pt]{amsart}
\usepackage{amsmath,amssymb,amsfonts,amsthm,amsopn}
\usepackage{tgtermes}
\usepackage{latexsym,graphicx}
\usepackage{xcolor}
\usepackage{color, colortbl}
\usepackage{pb-diagram}
\usepackage{comment}
\usepackage[title]{appendix}
\usepackage{tikz-cd}
\usepackage{tikz}
\usepackage{mathtools}
\usepackage{enumerate}
\usepackage{hyperref}
\usepackage{mathtools}
\mathtoolsset{showonlyrefs}

\usepackage{multirow}

\newtheorem{theorem}{Theorem}[section]
\newtheorem{lemma}[theorem]{Lemma}
\newtheorem{corollary}[theorem]{Corollary}
\newtheorem{proposition}[theorem]{Proposition}

\newtheorem{definition}[theorem]{Definition}

\newtheorem{remark}[theorem]{Remark}
\theoremstyle{definition}
\newtheorem{example}[theorem]{Example}

\newcommand{\beqa}{\begin{eqnarray*}}
\newcommand{\eeqa}{\end{eqnarray*}}

\DeclareMathOperator*{\id}{id}
\DeclareMathOperator*{\Sp}{Sp}
\DeclareMathOperator*{\w}{w}
\DeclareMathOperator*{\GL}{GL}
\DeclareMathOperator*{\Mp}{Mp}
\DeclareMathOperator*{\Sym}{Sym}
\DeclareMathOperator*{\diag}{diag}

\newcommand{\field}[1]{\mathbb{#1}}

\newcommand{\bR}{\field{R}}        
\newcommand{\bZ}{\field{Z}}        
\newcommand{\bC}{\field{C}}        
\def\la{\lambda}

\def\cF{\mathcal{F}}              
\def\cS{\mathcal{S}}
\def\cD{\mathcal{D}}

\def\cB{\mathcal{B}}
\def\cE{\mathcal{E}}

\def\cM{\mathcal{M}}
\def\cK{\mathcal{K}}
\def\cU{\mathcal{U}}
\def\cA{\mathcal{A}}

\def\cI{\mathcal{I}}
\def\cC{\mathcal{C}}
\def\cQ{\mathcal{Q}}
\def\cW{\mathcal{W}}

\def\cR{\mathcal{R}}

\def\cZ{\mathcal{Z}}

\def\frT{\mathfrak{T}}
\def\frp{\mathfrak{p}}
\def\frR{\mathfrak{R}}

\def\rd{\bR^d}

\def\rdd{{\bR^{2d}}}

\def\R{\right)}

\def\<{\left<}
\def\>{\right>}

\def\mv1{M_v^1}

\def\mn{(m,n)}
\def\mn'{(m',n')}

\newcommand{\abs}[1]{\lvert#1\rvert}
\newcommand{\norm}[1]{\lVert#1\rVert}

\def\w{\mathrm{w}}

\def\R{\mathbb{R}}
\def\Ren{\mathbb{R}^d}

\def\f{\varphi}

\def\Sn2{S_{2}(L^{2}(\Ren))}
\def\S1{S_{1}(L^{2}(\Ren))}
\def\sig00{\sigma_{0,0}}

\def\la{\langle}
\def\ra{\rangle}

\newcommand{\Op}{\mathrm{Op}}

\begin{document}

\begin{abstract}
	We develop a systematic analysis of the metaplectic semigroup $\Mp_+(d,\bC)$ associated with positive complex symplectic matrices, a notion introduced almost simultaneously and independently by H\"ormander, Brunet, Kramer, and Howe, thereby extending the classical metaplectic theory beyond the unitary setting.
    While the existing literature has largely focused on propagators of quadratic evolution equations, where integral representations are available in specific cases thanks to Mehler formulas, our approach is operator-theoretic and symplectic in spirit and adapts techniques from the standard metaplectic group $\Mp(d,\bR)$ to a substantially broader framework that is not driven by differential problems or particular propagators. 
    This point of view provides deeper insight into the structure of the metaplectic semigroup, and allows us to investigate its generators, polar decomposition, and intertwining relations with complex conjugation and with the Wigner distribution. We then exploit these structural results to characterize, from a metaplectic perspective, classes of time-frequency representations satisfying prescribed structural properties. Moreover, we discuss further implications for evolution equations with complex quadratic Hamiltonians, we study the boundedness of their propagators on modulation spaces, we obtain estimates in time of their operator norms. 
    Finally, we apply our theory to the study of the propagation of concentration of Wigner distributions.
\end{abstract}

\title[The metaplectic semigroup and its applications to time-frequency analysis and evolution operators
]{The metaplectic semigroup and its applications to time-frequency analysis and evolution operators
}

\author{Gianluca Giacchi}
\address{Università della Svizzera Italiana, IDSIA, Dipartimento di informatica, Via La Santa 1, 6962 Lugano, Switzerland}
\email{gianluca.giacchi@usi.ch}
\author{Luigi Rodino}
\address{Università di Torino, Dipartimento di matematica, Via Carlo Alberto 10, 10123 Torino, Italia}
\email{luigi.rodino@unito.it}
\author{Davide Tramontana}
\address{Università di Bologna, Dipartimento di matematica, Piazza di Porta San Donato 5, 40126 Bologna, Italy}
\email{davide.tramontana4@unibo.it}

\maketitle

\tableofcontents

\section{Introduction}
Motivated by applications to time-frequency analysis and to evolution equations, the aim of this work is to present the metaplectic semigroup $\Mp_+(d,\bC)$ in the language of modern harmonic analysis, and to develop a theory analogous to that of the classical metaplectic group $\Mp(d,\bR)$. In contrast to most of the existing literature, where it is presented as a semigroup of bounded operators on the Fock space, we study the metaplectic semigroup as a semigroup of operators bounded on $L^2(\rd)$, in the spirit of H\"ormander.
In the first three subsections of this Introduction, we review known results about the basics of $\Mp(d,\bR)$. Applications to time-frequency analysis and equations of Schr\"odinger type are recalled. We emphasize that the extension to the ``complex" case appears as a natural project for both the issues. Finally, Subsection \ref{subsec:14} of this Introduction presents our setting of $\Mp_+(d,\bC)$, with the main results providing a unitary treatment of time-frequency and PDEs problems.

\subsection{The metaplectic representation}
Metaplectic operators are a cornerstone in modern harmonic analysis and quantum mechanics, for they describe the evolution of Schrödinger initial value problems with quadratic Hamiltonians:
\begin{equation}\label{intro.eq1}
	\begin{cases}
		i\frac{1}{2\pi}\partial_tu(t,x)=\Op^\w(a)u(t,x) & \text{for $t\in\bR$ and $x\in\rd$,}\\
		u(0,x)=u_0(x),
	\end{cases}
\end{equation}
where $u_0\in L^2(\rd)$, and $a$ is the Weyl quantization of a real quadratic form $a(z)=\frac 1 2 \cQ z\cdot z$, $z\in\rdd$ and $\cQ=\cQ^\top$, that is
\begin{equation}
	\Op^\w(a)f(x)=\int_{\rdd}f(y)a\left(\frac{x+y}{2},\xi\right)e^{2\pi i(x-y)\cdot\xi}dyd\xi, \qquad f\in\cS(\rd).
\end{equation}
The family $\{e^{-2\pi it\Op^\w(a)}\}_{t\in\bR}$ is indeed a one-parameter subgroup of the {\em metaplectic group} $\Mp(d,\bR)$, the double cover of the symplectic group $\Sp(d,\bR)$. In his seminal paper \cite{wigner1932quantum}, Wigner defines a quadratic time-frequency representation whereby reducing problems \eqref{intro.eq1} to transport equations. With a more modern notation, the {\em (cross-)Wigner distribution} of $f,g\in L^2(\rd)$ is
\begin{equation}\label{intro.defW}
	W(f,g)(x,\xi)=\int_{\rd}f\left(x+\frac y 2 \right)\overline{g\left(x-\frac y 2 \right)}e^{-2\pi iy\cdot\xi}dy, \qquad x,\xi\in\rd.
\end{equation}
If $f=g$, we write $Wf=W(f,f)$, the {\em Wigner distribution} of $f$. Then, if $U_t\in\Sp(d,\bR)$ describes the canononical transformation associated with {\em the Hamiltonian} $a$ in \eqref{intro.eq1}, the aforementioned cover of the symplectic group maps $e^{-2\pi it\Op^\w(a)}\in\Mp(d,\bR)\mapsto U_t\in\Sp(d,\bR)$ for every $t\in\bR$, and these evolution operators read, for every $t\in\bR$, as linear symplectic transformations of the phase-space:
\begin{equation}
	W(e^{-2\pi it\Op^\w(a)}u_0)(z)=Wu_0(U_t^{-1}z), \qquad z\in\rdd, \quad f,g\in L^2(\rd).
\end{equation}
More generally, if $\widehat U\in\Mp(d,\bR)$ and its {\em projection} onto $\Sp(d,\bR)$ is $U$, the {\em symplectic covariance} of the Wigner distribution
\begin{equation}\label{intro.WigCovarianceProp}
	W(\widehat Uf,\widehat Ug)(z)=W(f,g)(U^{-1}z), \qquad z\in\rdd, 
\end{equation}
is not only satisfied for every $f,g\in L^2(\rd)$, but it also characterizes metaplectic operators and of their projection onto $\Sp(d,\bR)$ \cite{MR1906251}. In a more algebraic fashion, the metaplectic operators associated to a symplectic matrix $U\in\Sp(d,\bR)$ are classically defined, up to a phase factor, through the intertwining relation with the Schrödinger representation of the Heisenberg group
\begin{equation}\label{intro.Schrodinger}
	\rho(x,\xi;\tau)f(y)=e^{2\pi i\tau}e^{-i\pi \xi\cdot x}e^{2\pi i\xi\cdot y}f(y-x), \qquad f\in L^2(\rd), \, x,\xi\in\rd, \, \tau\in\bR,
\end{equation}
that is
\begin{equation}\label{intro.intertSShar}
	\widehat U\rho(z;\tau)=\rho(Uz;\tau)\widehat U, \qquad z\in\rdd, \, \tau\in\bR,
\end{equation}
we refer to Section \ref{sec.pre} below for the details. For the purpose of this discussion, the phase factor $e^{2\pi i\tau}$ plays a minor role and we will therefore write $\rho(z)$ for the rest of the Introduction, when it is convenient. This dualism between operators and matrices is crucial: the aforementioned cover map is indeed a homomorphism, and this considerably simplifies the analysis of metaplectic operators by reducing operator calculus to linear algebra. This allows properties of metaplectic operators to be inferred and interpreted directly from the structure of their projection. Conversely, the mapping associating to each $U\in\Sp(d,\bR)$ the corresponding metaplectic operator, defined up to a phase factor, is a projective representation known as the {\em metaplectic representation}.

Alternatively, a more explicit description of $\Mp(d,\bR)$ can be given in terms of its generators: they are obtained by composing any number of the following unitary operators on $L^2(\rd)$, in any order: the Fourier transform $\cF$, the rescalings
\begin{equation}\label{intro.COV}
	\big\{\mathfrak{T}_Ef(x)=|\det(E)|^{1/2}f(Ex) : E\in \GL(d,\bR)\big\},
\end{equation}
and the product operators
\begin{equation}\label{intro.POps}
	\big\{\mathfrak{p}_Qf(x)=e^{i\pi Qx\cdot x}f(x) : Q\in\bR^{d\times d}, \, Q=Q^\top\big\}.
\end{equation}
In \eqref{intro.COV} the correct phase factors are omitted for simplicity of the discussion. We leave the details to Section \ref{sec.pre}. The interest in the metaplectic representation has increased in the last decade, as it is witnessed by the recent contributions of several authors, see e.g. \cite{Nittis,MR4748757,FS2024,MR4846323,Lerner}.

\subsection{Metaplectic Wigner distributions associated to \texorpdfstring{$\widehat\cU\in\Mp(2d,\bR)$}{Lg}} A more careful examination of the cross-Wigner distribution \eqref{intro.defW}, shows that 
\begin{equation}\label{intro.Wig}
	W(f,g)=\cF_2\mathfrak{T}_{E_w}(f\otimes\overline g), \qquad f,g\in L^2(\rd),
\end{equation}
where
\begin{equation}
	\cF_2F(x,\xi)=\int_{\rd}F(x,y)e^{-2\pi iy\cdot \xi}dy, \qquad F\in\cS(\rdd), \, x,\xi\in\rd,
\end{equation}
is the {\em partial Fourier transform with respect to the frequency variables}, and $\mathfrak{T}_{E_w}F(x,y)=F(x+y/2,x-y/2)$. 
Both the operators $\cF_2$ and $\mathfrak{T}_{E_w}$ are in $\Mp(2d,\bR)$, and so is their composition. The Wigner distribution shares this feature with many other popular time-frequency representations, such as the short-time Fourier transform (STFT)
\begin{equation}\label{intro.defSTFT}
	V_gf(x,\xi)=\int_{\rd}f(y)\overline{g(y-x)}e^{-2\pi i\xi\cdot y}dy, \qquad f,g\in L^2(\rd), \, x,\xi\in\rd,
\end{equation}
with the result that time-frequency analysis is strictly woven with metaplectic operators’ theory. Note that, in this fashion, the definitions of the Wigner distribution and of the STFT can also be generalized to $f,g\in\cS'(\rd)$.
\begin{definition}\cite{Cordero_Part_I}\label{intro.def11}
	Let $\widehat\cU$ be a metaplectic operator on $L^2(\rdd)$. The {\em (cross-)metaplectic Wigner distribution} associated to $\widehat\cU$, or (cross-)$\cU$-Wigner distribution, is the time-frequency representation
	\begin{equation}\label{intro.defWA}
		W_\cU(f,g)=\widehat\cU(f\otimes\overline g), \qquad f,g\in L^2(\rd).
	\end{equation}
	If $f=g$, we write $W_\cU f=W_\cU(f,f)$.
\end{definition}
By working in this setting, we can exploit the deep relation of the metaplectic group with symplectic matrices to express and explain the properties of an $\cU$-Wigner distribution in terms of the structure of the projection $\cU$. Nowadays, metaplectic Wigner distributions are a well-established subject of study in the time-frequency analysis community, their applications spanning from pseudodifferential calculus \cite{MR4182445,MR2843220} to frame theory \cite{CG02,MR3303672}, uncertainty principles \cite{GroSha}, quantum physics \cite{MR4880223} and signal processing \cite{MR4846323,MR4685284,MR4843048}. In a recent unpublished work \footnote{See K. Gr\"ochenig and I. Shafkulovska. More uncertainty principles for metaplectic time-Frequency representations. {\em arXiv:2503.13324}, 2025.}, it has been proven that, under minimal continuity assumptions, a time-frequency representation $Q:\cS(\rd)\times\cS(\rd)\to\cS'(\rdd)$ satisfies
\begin{equation}\label{intro.Gro}
	Q(\rho(z)f,\rho(z)g)=\rho(\phi(z,w))Q(f,g), \qquad z,w\in\rdd
\end{equation}
for some measurable function $\phi$ if and only if $Q=W_\cA$ for some $\widehat \cA\in\Mp(2d,\bR)$. On the one hand, \eqref{intro.Gro} shows that metaplectic Wigner distributions are the only time-frequency representation satisfying the property of preserving phase-space shifts. On the other hand, it also says that $\cA$-Wigner distributions, as they are defined so far, are too restrictive, excluding other representations from the pantheon of time-frequency distributions whose properties can be explained in terms symplectic geometry. A prototypical example is the Husimi distribution \cite{Husimi}
\begin{equation}
	H(f,g)=|V_gf|^2, \qquad f,g\in L^2(\rd),
\end{equation}
which is not an $\cU$-Wigner distribution. Property \eqref{intro.Gro} is very strong, but even within the metaplectic framework, one can construct a rich class of time–frequency representations by allowing $\cA\in\Sp(2d,\bC)$, the group of $4d\times4d$ complex symplectic matrices. These representations do not satisfy \eqref{intro.Gro}, yet they still enjoy useful structural features, such as covariance.
\begin{definition}\label{intro.defCov}
A quadratic time-frequency representation $Q:L^2(\rd)\times L^2(\rd)\to L^2(\rdd)$ is \emph{covariant} if
	\begin{equation}\label{intro.covariance}
		Q(\pi(z)f,\pi(z)g)=Q(f,g)(\cdot-z),
	\end{equation}
    where $\pi(x,\xi)f(y)=e^{2\pi i\xi\cdot y}f(y-x)$.
\end{definition}
Covariant time-frequency representations interpret phase-space shifts of signals as translations in phase space. Classical examples include the Wigner and Husimi distributions. Including the latter into a metaplectic framework amounts to extending the class of metaplectic operators to the so-called {\em metaplectic semigroup}, the main subject of this work.

\subsection{The metaplectic semigroup \texorpdfstring{$\Mp_+(d,\bC)$}{Lg}: former contributions}
The four leading figures who investigated this complex extension of the metaplectic group are Brunet and Kramer \cite{Brunet,BrunetKramer}, Howe \cite{Howe1}, and H\"ormander \cite{Hormander1,Hormander3}. Although they worked independently, they all arrived at the same mathematical object. Brunet and Kramer coined the term {\em metaplectic semigroup}, which was also used later by H\"ormander, while Howe defined the {\em oscillator semigroup}. Howe, Brunet and Kramer argued at the level of the Fock space, where the existence of a reproducing kernel facilitated the definition of operators in terms of their Schwartz kernels (as operators on the Fock space). These two approaches were shown to be equivalent by Hilgert in 1989 \cite{Hilgert}. On the contrary, H\"ormander's semigroup of {\em positive symplectic linear maps} consists essentially of operators on $L^2(\rd)$ with Gaussian kernel, which form a double cover of a semigroup of positive complex symplectic matrices, see Definition \ref{defPosSymp} in Section \ref{sec:HPC}. The equivalence between H\"ormander's theory and Brunet-Kramer-Howe's theory is discussed for the benefit of the reader in Remark \ref{rem-equivalences} below. It can be either inferred by comparing the positivity conditions \cite[Equation (5.10)]{Hormander3} and \cite[Page 214]{BrunetKramer}, or by comparing \cite[Proposition 5.10]{Hormander3} with \cite[Theorem 2.1]{Brunet} and \cite[Section 4]{Hormander1} with \cite[Section V]{Brunet}. 
We devote Section \ref{sec.thecomplexmeta} to reviewing the definition and structure of the metaplectic semigroup, denoted here by $\Mp_+(d,\bC)$, and of its projection onto the subsemigroup of positive symplectic matrices $\Sp_+(d,\bC)$.

H\"ormander's interest in the metaplectic semigroup arises from the fact that it contains one-parameter semigroups of operators of the form $\{e^{-2\pi it \Op^\w(a)}\}_{t\geq0}$, propagators of evolution problems \eqref{intro.eq1}
    with Hamiltonian $a$ satisfying
    \begin{equation}\label{intro.LR1}
        a(z)=\frac 1 2\cQ z\cdot z, \qquad z\in\rdd,
    \end{equation}
    with $\cQ^\top=\cQ$ a complex quadratic form with positive semi-definite imaginary part. 
    Outside of the caustics, the solution $u(t,x)=e^{-2\pi it \Op^\w(a)}u_0(x)$ can be expressed by Mehler formulas, proven by H\"ormander \cite[Theorem 4.1]{Hormander3}.
	
	The analysis of these problems has initiated a long series of recent and well-celebrated works, extending far beyond H\"ormander's original contributions. Building on H\"ormander’s analysis of quadratic operators, Hitrik and Pravda-Starov extended his framework to the non-elliptic setting in \cite{Hitrik} and further derive smoothing properties of heat semigroups, spectral discreteness under suitable ellipticity on this space, and the large-time behavior of the generated contraction semigroups. Importantly, they discover the singular space associated to a complex quadratic form, which describes geometrically the (symplectic orthogonal complement of the) directions where the corresponding heat semigroup is smoothing.
	Microlocal analysis of problems of the form \eqref{intro.eq1} with complex quadratic Hamiltonians was subsequently developed in a series of works \cite{PravdaStarov,PRW} and, more recently, in \cite{MPT2025,MR4442595}, within the framework of propagation of singularities, aimed at encoding the microlocal singularities of solutions of quadratic evolution equations 
    and relate them with those of the initial datum, inside the singular space. 
     
    In \cite{MR4598725} the authors derive a polar decomposition for one-parameter semigroups describing the propagators of the corresponding problem \eqref{intro.eq1}.
    Particularly noteworthy are the recent contributions of Hitrik, Pravda-Starov, and Viola \cite{HPSV}, Viola \cite{Viola2,Viola1}, and Alphonse and Bernier \cite{AlphonseBernier,Bernier}, which study the smoothing and subelliptic properties of quadratic non-selfadjoint differential operators, once again relying on the geometric structure of the singular space.
    A more time-frequency analysis approach is instead developed by Trapasso in \cite{Trapasso}, where he provides off-diagonal decay estimates for the Gabor matrices of these one-parameter semigroups.
    It seems natural to develop an operator-theoretic treatment of the general propagator $e^{-2\pi it\Op^\w(a)}$ of \eqref{intro.eq1}, with $a(z)$ satisfying \eqref{intro.LR1}, in the lines of a metaplectic calculus. This allows applications to the so-called Wigner wave front set and to the boundedness in time-frequency spaces.
	
	\subsection{Present contributions}\label{subsec:14}
	Since they are mostly concerned with the propagators of problems \eqref{intro.eq1}, which form subsemigroups of $\Mp_+(d,\bC)$, the aforementioned contributions do not aim at a systematic investigation of the metaplectic semigroup $\Mp_+(d,\bC)$. Instead, in the present work, we analyze deeper several, yet exhaustive, aspects of $\Mp_+(d,\bC)$ and $\Sp_+(d,\bC)$. We then apply our results to the analysis of time-frequency representations within a metaplectic framework, showing that properties which, in the standard setting, are satisfied only by a few trivial examples, here hold for a much broader class of {\em nonstandard} metaplectic Wigner distributions. This demonstrates that our approach provides a natural and necessary generalization of the classical theory. For the rest of this work, the term {\em metaplectic operator} will be used to indicate an operator in $\Mp_+(d,\bC)$.	
	\subsubsection{Generators}
	In parallel with the generators of $\Mp(d,\bR)$ discussed above, the symplectic group $\Sp(d,\bR)$ is generated by the matrix 
	\begin{equation}\label{intro.defJ}
		J=\begin{pmatrix}
			O_d & I_d\\
			-I_d & O_d
		\end{pmatrix}
	\end{equation}
	of the standard symplectic form of $\bC^d$, and by the subgroups of matrices in the form
	\begin{equation}\label{intro.defDEVQ}
		\cD_E=\begin{pmatrix}
			E^{-1} & O_d\\
			O_d & E^\top
		\end{pmatrix}
		\qquad \text{and} \qquad 
		V_Q=\begin{pmatrix}
			I_d & O_d\\
			Q & I_d
		\end{pmatrix}, 
	\end{equation}
	varying $E\in\GL(d,\bR)$ and $Q$ real and symmetric. By allowing $E$ and $Q$ to be complex, we obtain the generators of $\Sp(d,\bC)$. However, constructing a projective representation associating every $S\in\Sp(d,\bC)$ to an operator bounded on $L^2(\rd)$, such that it agrees with the metaplectic representation on $\Sp(d,\bR)$ and its image forms a group, is impossible. This also reflects the fact that, unlike the real case, problems \eqref{intro.eq1} with complex quadratic Hamiltonians are generally defined only for $t \geq 0$. Indeed, $\Sp_+(d,\bC)$ forms a semigroup rather than a group, and the question of its generators arises naturally.
	\begin{proposition}
		The semigroup $\Sp_+(d,\bC)$ is generated by $J$ and by the matrices $\cD_E$ and $V_Q$ in \eqref{intro.defDEVQ}, varying $E\in GL(d,\bR)$ and $Q\in\bC^{d\times d}$ symmetric and with positive semi-definite imaginary part.
	\end{proposition}
	See also Proposition \ref{prop51} below. This result reflects on the generators of $\Mp_+(d,\bC)$, as discussed in Proposition \ref{prop52}.

	\subsubsection{Polar decomposition and its applications}
		The polar decomposition derived in \cite{MR4598725} for the propagators of  problems \eqref{intro.eq1} with complex quadratic Hamiltonians is actually a very general result holding for every metaplectic operator, see Theorem \ref{polar}. This yields the algebraic characterization of $\Mp_+(d,\bC)$ as isomorphic to $\Mp(d,\bR)\times \Mp_0(d,\bC)$, where $\Mp_0(d,\bC)$ is the subset of metaplectic operators that are positive and self-adjoint, and to a further characterization of the operators in $\Mp_0(d,\bC)$ as the $\Mp(d,\bR)$-conjugates of an explicit subsemigroup of metaplectic operators, here called {\em atomic metaplectic contractions}, see Section \ref{subsec:32} and Corollary \ref{CorPolar}. Let us briefly summarize these results.
		
		\begin{theorem}\label{intro.thm2}
			The polar decomposition $\widehat S=\widehat U\widehat Z$ of $\widehat S\in\Mp_+(d,\bC)$ satisfies 
			\begin{enumerate}[(i)]
			\item $\widehat U\in\Mp(d,\bR)$.
			\item $\widehat Z\in\Mp_0(d,\bC)$.
			\end{enumerate}
		\end{theorem}
		
	In 1932, Wigner introduced a form of phase-space analysis of operators via the Wigner distribution \eqref{intro.Wig}. The core of his {\em Wigner analysis} is the observation that for a linear and continuous operator $T:\cS(\rd)\to\cS'(\rd)$ there exists an operator $K:\cS(\rdd)\to\cS'(\rdd)$ such that
			\begin{equation}\label{intro.defKmale}
				W(Tf)=KWf, \qquad f\in\cS(\rd).
			\end{equation}
			Nowadays, we know that $K$ is not uniquely determined by $T$ under the sole condition \eqref{intro.defKmale}. To retrieve uniqueness, we must consider the cross-Wigner distribution. In fact, there exists a unique operator $K:\cS(\rdd)\to\cS'(\rdd)$ such that
			\begin{equation}
				W(Tf,Tg)=KW(f,g), \qquad f,g\in\cS(\rd).
			\end{equation}
			We refer to $K$ as the {\em Wigner operator} associated to $T$.
			
			By \eqref{intro.WigCovarianceProp}, the Wigner operator of $\widehat U\in\Mp(d,\bR)$ is a symplectic rescaling, specifically $K=\mathfrak{T}_{U^{-1}}$. It is worth to observe that 
            \begin{equation}
                W(Tf,Tg)=W(f,g)(M^{-1}z), \quad z\in\rdd,
            \end{equation}
            can hold for $M\in\GL(2d,\bR)$ if and only if the corresponding operator (possibly anti-linear) $T$ is either a metaplectic operator or the composition of a metaplectic operator with complex conjugation, as it was proven by Dias and Prata \cite{DiasPrata}. 
			
			Thanks to the polar decomposition of metaplectic operators, we are now able to generalize \eqref{intro.WigCovarianceProp} to $\Mp_+(d,\bC)$. Specifically, if $\widehat S=\widehat U\widehat Z$ as in Theorem \ref{intro.thm2}, then
			\begin{equation}\label{intro.WignerPolar}
				W(\widehat Sf,\widehat Sg)(z)=K_ZW(f,g)(\widehat U^{-1}z), \qquad f,g\in L^2(\rd), \, z\in\rdd,
			\end{equation}
			where $K_Z\in\Mp_0(2d,\bC)$, see Theorem \ref{thmS1S2S3W} for its explicit expression and further details.
			
			Since the polar decomposition of $\widehat S$ is unique, formula \eqref{intro.WignerPolar} provides an unambiguous separation between the contribution of what can be regarded as a sort of {\em real part} of $\widehat S$, namely $\widehat U\in\Mp(d,\bR)$, and that of its {\em imaginary part}, represented by $\widehat Z\in\Mp_0(d,\bC)$. However inappropriate, these phraseologies provide a clear picture of how the positive self-adjoint operator $\widehat Z$, absent for metaplectic operators with real projections, affects the phase-space concentration of $\widehat S$. Stated differently, when $\widehat U\in\Mp(d,\bR)$ we recover the classical symplectic covariance property of the Wigner distribution \eqref{intro.WigCovarianceProp}, whereas in the general case \eqref{intro.WignerPolar} yields an extension in which the action of $\widehat U$ is disentangled from the additional contribution induced by $\widehat Z$.
	
	\subsubsection{Schr\"odinger representation}
	For $(z,w)\in\bC^{2d}$ and $\tau\in\bR$, we might consider
	\begin{equation}\label{intro.complexTFshift}
		\rho(z,w;\tau)f(s)=e^{2\pi i\tau}e^{-i\pi z\cdot w}e^{2\pi iw\cdot s}\int_{\rd}\widehat f(\eta)e^{2\pi i(s-z)\cdot\eta}d\eta,
	\end{equation}
	for suitably regular $f$ whose Fourier transform $\widehat{f}$ has suitable decay properties (e.g. $\widehat{f}$ with compact support). However, this complexification of the Schr\"odinger representation maps $L^2(\rd)$ to itself if and only if $z,w\in\rd$. This follows by the fact that its Schwartz kernel is not in $\cS'(\rdd)$ if $S\notin\Sp(d,\bR)$. Consequently, if $z\in\rdd$ and $S\in\Sp_+(d,\bC)\setminus\Sp(d,\bR)$, the corresponding term $\rho(Sz;\tau)$ in \eqref{intro.intertSShar} is only densely defined on $L^2(\rd)$. Nevertheless, we retrieve \eqref{intro.intertSShar} on a dense subspace.
	\begin{theorem}\label{intro.thm34}
	Let $\widehat S\in\Mp_+(d,\bC)$ and $S\in\Sp_+(d,\bC)$ be its projection. Then, identity \eqref{intro.intertSShar} holds for a dense subset of $L^2(\rd)$. Specifically, it holds for 
	\begin{equation}
	f\in \mathrm{span}\big\{\text{$\rho(z;\tau)e^{-\pi Mx\cdot x}$ : $z\in\rdd$, $\tau\in\bR$, and $M=M^\top$ real and positive-definite}\big\}.
	\end{equation}
	\end{theorem}
	We address Section \ref{subsec:43} for the details. A similar result appears in \cite[Section 5]{Viola2}, although the link between the metaplectic operator and its symplectic projection is not discussed there. Here, we present an explicit version of this relation.
    Theorem \ref{intro.thm34} is used in Section \ref{sec.apptimefreq} to characterize covariant metaplectic Wigner distributions.
	
	\subsubsection{Intertwining with complex conjugation and its applications}
	For operators in $\Mp(d,\bR)$, the following result was proven in \cite[Appendix A]{CG02}.
			\begin{proposition}\label{PropCG02}
			If $\widehat U\in\Mp(d,\bR)$, then the operator $\widehat{\widetilde U}:L^2(\rd)\to L^2(\rd)$ given by
			\begin{equation}\label{defRealConjMetap}
				\widehat{\widetilde U}f=\overline{\widehat U\bar f}
			\end{equation}
			is metaplectic, and if $U$ has blocks \eqref{blockS}, then
			\begin{equation}\label{projRealConj}
				\pi^{Mp}(\widehat{\widetilde U})=\widetilde U=\begin{pmatrix}
					A & -B\\
					-C & D
				\end{pmatrix}.
			\end{equation}
			\end{proposition}
			Aside of its applications in characterizing time-frequency spaces, the connection between complex conjugation and metaplectic operators is much deeper. Following the same notation in \cite{DiasPrata}, we can write 
            \begin{equation}\label{intro.RSR}
                \widetilde U=RUR,
            \end{equation} 
            with
			\begin{equation}\label{intro.defRmatrix}
				R=\begin{pmatrix}
					I_d & O_d\\
					O_d & -I_d
				\end{pmatrix}.
			\end{equation}
			Recall that the {\em anti-symplectic group} $R\Sp(d,\bR)$ is the group of real $2d\times2d$ matrices $M$ such that $M^\top JM=-J$. This choice of notation is consistent with the fact that every matrix in $M\in R\Sp(d,\bR)$ can be written as $M=RU$ for some $U\in\Sp(d,\bR)$. If we set $\widehat Rf=\overline f$, equation \eqref{defRealConjMetap} reads as $\widehat{\widetilde{U}}=\widehat R\widehat U\widehat R$, consistently with \eqref{intro.RSR}. By \cite[Theorem 25]{DiasPrata}, the intertwining relation \eqref{intro.WigCovarianceProp}, holds for $M\in\GL(2d,\bR)$ if and only $M\in \Sp(d,\bR)\cup R\Sp(d,\bR)$. In conclusion, one may be tempted to extend the class of metaplectic operators up to include complex conjugation (as an anti-linear operator), setting $\pi^{Mp}(\widehat R)=R$. We mention \cite{MR4846323} as a recent contribution where this perspective has been adopted in the context of (Gabor) frames theory.
			
			Having said so, the question arises of whether Proposition \ref{PropCG02} generalizes to complex symplectic matrices. However, in the complex case, the matter is more delicate.
			
			\begin{theorem}\label{thmComplexConj}
				For $S\in\Sp_+(d,\bC)$ with blocks \eqref{blockS}, let
				\begin{equation}\label{defStildecomplex}
					\widetilde S=\begin{pmatrix}
						\overline A & -\overline B\\
						-\overline C & \overline D
					\end{pmatrix},
				\end{equation}
				then the operator
				\begin{equation}
					Tf(x)=\overline{\widehat S\bar f}, \qquad f\in L^2(\rd),
				\end{equation}
				is in $\Mp_+(d,\bC)$ and $\pi^{Mp}_+(T)=\widetilde S$.
			\end{theorem}
			For the rest of this work, the operator $T$ of Theorem \ref{thmComplexConj} will be denoted by $\widehat{\widetilde S}$, consistently with Proposition \ref{PropCG02}. The reader may observe that in the complex case it is not true anymore that $\widetilde S=RSR$, so complex conjugation alone cannot be included consistently among metaplectic operators in the complex framework. The subsemigroup of metaplectic operators so that $\widehat{\widetilde S}=\widehat S$, that are precisely those commuting with complex conjugation, are then characterized in Theorem \ref{StructuralTheorem}. Remarkably, its proof requires a preliminary result on the structure of $\Sp(d,\bR)$, see Lemma \ref{structuralLemma}.
			
			The Wigner distribution satisfies the following intertwining relation with complex conjugation
			\begin{equation}\label{intro.WigCompConj}
				W(g,f)=\overline{W(f,g)}, \qquad f,g\in L^2(\rd),
			\end{equation}
			which in turns implies that $Wf$ is always real-valued for $f\in L^2(\rd)$. Theorem \ref{thmComplexConj} will be applied in Section \ref{sec.apptimefreq} to characterize metaplectic Wigner distributions satisfying a properly formulated generalization of \eqref{intro.WigCompConj}, that, in the case of metaplectic Wigner distributions defined from operators in $\Mp(2d,\bR)$, reduce to rescaled cross-Wigner distributions.
	
	\subsubsection{A new metaplectic setting for time-frequency representations}
		By allowing $\widehat\cA\in\Mp(2d,\bC)$ in Definition \ref{intro.def11}, we obtain a generalization of metaplectic Wigner distributions of \cite{Cordero_Part_I}, that we present in Section \ref{sec.apptimefreq}. We thereby transport the techniques resulted advantageous in the framework of the metaplectic group $\Mp(2d,\bR)$ to the metaplectic semigroup $\Mp_+(2d,\bC)$. Importantly, we characterize the following classes of metaplectic Wigner distributions in terms of the symplectic projection $\cA$ of the metaplectic operator $\widehat\cA$ defining $W_\cA$, and provide the explicit expressions of the time-frequency representations satisfying these properties. 
		\begin{enumerate}
			\item The Cohen's class, which incidentally coincides with the class of covariant metaplectic Wigner distributions. Here, we exploit the intertwining relation \eqref{intro.intertSShar} for $\Mp_+(2d,\bC)$.
			\item {\em Generalized spectrograms} \cite{BDO1,BDO}, time-frequency representations $Q$ in the form
			\begin{equation}
				Q(f,g)=V_\phi f\overline{V_\psi g}, \qquad f,g\in\cS'(\rd),
			\end{equation}
			for suitable distributions $\phi,\psi\in\cS'(\rd)$, appearing in the time-frequency analysis of the so-called {\em localization operators}, see \cite{localization}. As a by-product we also characterize {\em spectrograms}, i.e., the case $\phi=\psi$. In our proofs, we consider a further invertibility assumption on the $d\times d$ block $A_{13}$ of $\cA$. Without this hypothesis, the characterization of generalized spectrograms becomes unnecessarily more involved.
			\item Metaplectic Wigner distributions satisfying
			\begin{equation}
				W_\cA(g,f)=\overline{W_\cA(f,g)},
			\end{equation}
			which also entails that $W_\cA f$ is real-valued. This property shows up also in the study of Wigner wave front sets.
		\end{enumerate} 

    \subsubsection{Evolution equations with complex quadratic Hamiltonian.} Finally, the (complex) metaplectic theory that we study applies to evolution equations with quadratic non-selfadjoint Hamiltonian. 
    In this direction, we consider Cauchy problems of the form
    \begin{equation}\label{evolutionintro}
		\begin{cases}
		\frac{1}{2\pi}\partial_t u+\mathrm{Op}^\mathrm{w}(a)u=0 \quad \text{in} \ \ \R^+ \times \R^d \\
		u(0,\cdot)=u_0 \in L^2(\R^d),
		\end{cases}
		\end{equation}
		with 
        \[
        a(z)=  \frac 1 2 \cQ z\cdot z, \quad z \in \R^{2d},
        \]
        \textit{complex} quadratic form associated with a matrix $\cQ \in \mathbb{C}^{2d\times 2d}$ satisfying $\Re(\cQ) \geq 0$.
        As aforementioned, the propagator
        \[
        e^{-2\pi t\mathrm{Op}^{\mathrm{w}}(a)}:L^2(\R^d) \rightarrow L^2(\R^d), \qquad u_0\mapsto u(t,\cdot)= e^{-2\pi t\mathrm{Op}^{\mathrm{w}}(a)}u_0,
        \]
        is a one-parameter semigroup of metaplectic operators $e^{-2\pi t\mathrm{Op}^{\mathrm{w}}(a)} \in \Mp_+(d,\mathbb{C})$ (cf. \cite{Hormander3}).
        Hence, in Section \ref{sec.evo} we study the evolution of the microlocal concentration of $\cA$-Wigner distributions, with $\cA \in \Sp(2d,\R)$ covariant, induced by \eqref{evolutionintro}. In such a section, relying on the evolution of a mixed heat–Schr\"odinger type induced by harmonic oscillators, we prove boundedness results on modulation spaces and propagation of micro-singularities, we address Section \ref{sec.evo}.
	
	\subsection*{Outline} In Section \ref{sec.pre}, we introduce the notation and the preliminaries necessary for this work, while in Section \ref{sec.thecomplexmeta} we introduce the metaplectic semigroup $\Mp_+(d,\bC)$ and its main known properties. Section \ref{subsec:anp} is devoted to extend the main properties of the metaplectic group $\Mp(d,\bR)$ in the context of $\Mp_+(d,\bC)$. In Section \ref{sec:Subsemigroups}, we characterize the semigroups metaplectic operators in $\Mp_+(d,\bC)$ with upper and lower block triangular projections, we obtain the intertwining relations with the Wigner distribution and complex conjugation. Applications to time-frequency analysis are described in Section \ref{sec.apptimefreq}, and applications to evolution equations are presented in Section \ref{sec.evo}.

\section{Preliminaries and notation}\label{sec.pre}
We denote by $xy=x\cdot y$ the standard inner product of $x,y\in\rd$. For a $d\times d$ complex matrix $S\in\bC^{d\times d}$, we write $\Re(S)$ and $\Im(S)$ to denote its real and imaginary parts, respectively, while $\overline S$ denotes its complex conjugate. We denote by $\GL(d,\bR)$, respectively $\GL(d,\bC)$, the group of $d\times d$ real, respectively complex, invertible matrices. Analogously, $\Sym(d,\bR)$ and $\Sym(d,\bC)$, denote the class of symmetric $d\times d$ real and complex matrices, i.e., satisfying $S^\top=S$. If $S\in\Sym(d,\bR)$, we write $S\geq0$, respectively $S>0$, if $S$ is positive semi-definite, respectively positive-definite. We also write $S\leq0$ and $S<0$ if $-S\geq0$ and $-S>0$, respectively. We denote by $\Sigma_{\geq0}(d)$ the (closure of the) Siegel upper-half space, i.e., the set of complex symmetric matrices $S$ with $\Im(S)\geq0$. We denote by $I_d$ and $O_d$ the $d\times d$ identity and null matrices, respectively. We also write $O_{m\times n}$ for the $m\times n$ null matrix. 
If $S_1,S_2$ are $2d\times 2d$ matrices, with
	\begin{equation}
		S_j=\begin{pmatrix}
			A_j & C_j\\
			B_j & D_j
		\end{pmatrix}, \qquad A_j,B_j,C_j,D_j\in\bC^{d\times d}, \quad j=1,2,
	\end{equation}
	we shall denote by 
	\begin{equation}\label{tensmat}
	S_1\otimes S_2=\left(\begin{array}{cc|cc}
		A_1 & O & B_1 & O\\
		O & A_2 & O & B_2\\
		\hline
		C_1 & O & D_1 & O\\
		O & C_2 & O & D_2
	\end{array}\right).
	\end{equation}
		
	We denote by $\cS(\rd)$ the Schwartz class of smooth rapidly decreasing functions, its topological dual is the space of tempered distributions. The sesquilinear inner product of $f,g\in L^2(\rd)$ is $\la f,g\ra =\int_{\rd}f(x)\overline{g(x)}dx$, which extends uniquely to a duality pairing $\cS'(\rd)\times\cS(\rd)$, antilinear in the second component. If $X$ and $Y$ are Banach spaces and $T:X\to Y$ is a bounded linear operator, $\norm{T}_{X\to Y}$ denotes its operator norm. Whenever it does not cause confusion, we write $\norm{T}_{op}$.
    
    The {\em Fourier transform} of $f\in L^1(\rd)$ is
	\begin{equation}\label{hatf}
		\widehat f(\xi)=\int_{\rd}f(x)e^{-2\pi i\xi \cdot x}dx, \qquad \xi\in\rd.
	\end{equation}
	The translation and the modulation of a function $f\in L^2(\rd)$ are $T_xf(y)=f(y-x)$ and $M_\xi f(y)=e^{2\pi i\xi \cdot y}f(y)$, respectively. We also write $\pi(x,\xi)f=M_\xi T_x$. For $M>0$, we write $\f_M(t)=e^{-\pi Mt\cdot t}$ for the Gaussian with covariance matrix $M$.
	 
	In what follows, we present briefly the background needed for our discussion. We will mainly follow the notation and the terminology of \cite{CorderoBook} for time-frequency analysis and of \cite{hormander3book,lernermetrics} for microlocal analysis. Equally valid sources are \cite{Pre-Iwasawa,MauriceDeGosson,Folland,Grochenig}.

    \subsection{Pseudo-inverses and Schur complements}
    Schur complements are used in this work to study semi-definiteness of symmetric matrices. Let us recall the definition of Moore-Penrose inverse of a real-valued matrix $M$. There exist $U,V$ orthogonal and $\Sigma$ diagonal such that $M=U\Sigma V^\top$ (singular value decomposition) and the diagonal elements of $\Sigma$ are the singular values of $M$, that are the eigenvalues of $(M^\top M)^{1/2}$. Without loss of generality, we may assume that $\Sigma=\diag(\sigma_1,\ldots,\sigma_d)$, with $\sigma_1\geq\ldots\geq \sigma_r>0$, $r=\mathrm{rank}(M)$. The Moore-Penrose inverse of $M$ is $M^+=V\Sigma^+U^\top$, where $\Sigma^+=\diag(\lambda_1,\ldots,\lambda_d)$, $\lambda_j=\sigma_j^{-1}$ for $1\leq j\leq r$ and $\sigma_j=0$ if $r<j\leq d$. We recall that $M^+M$ is the orthogonal projection onto $\ker(M)^\perp$ and $MM^+$ is the orthogonal projection onto $\mathrm{range}(M)$. Both $M^+M$ and $MM^+$ are symmetric and positive semi-definite. Finally, we recall that $(M^+)^\top=(M^\top)^+$. 
    
    We will use the following result, see \cite[Theorem 16.1]{Schur}. 
    \begin{theorem}\label{theoremSchur}
        Let 
        \begin{equation}
            M=\begin{pmatrix}
                A & B\\
                B^\top & C
            \end{pmatrix}\in\Sym(2d,\bR).
        \end{equation}
        The following statements are equivalent.
        \begin{enumerate}[(i)]
            \item $M\geq0$.
            \item $A\geq0$, $(I-AA^+)B=O_d$ and $C-B^\top A^+ B\geq0$.
            \item $C\geq0$, $(I-CC^+)B^\top=O_d$ and $A-BC^+B^\top\geq0$.
        \end{enumerate}
    \end{theorem}
    Since $AA^+$ is the orthogonal projection onto $\mathrm{range}(A)$, the reader may observe that
    $(I-AA^+)B=O_d$ is then equivalent to $\mathrm{range}(B)\subseteq \mathrm{range}(A)$. Similarly, $(I-CC^+)B^\top=O_d$ is equivalent to $\mathrm{range}(B^\top)\subseteq \mathrm{range}(C)$.
        
        \subsection{Gaussian integrals}
		We will use the following well-established formula for Gaussian integrals. The reader may refer to \cite[Appendix A]{Folland} for a more detailed discussion.
		\begin{proposition}\label{GaussianIntegrals}
			Let $M\in \Sym(d,\bC)\cap\GL(d,\bC)$ be such that $\Re(M)>0$. Then, for every $z\in\bC^d$,
			\begin{equation}\label{GaussInt}
				\int_{\rd}e^{-\pi Mx\cdot x}e^{-2\pi iz\cdot x}dx=\det(M)^{-1/2}e^{-\pi M^{-1}z\cdot z},
			\end{equation}
			where $\det(M)^{-1/2}$ is chosen so that $\det(M)^{-1/2}>0$ if $M$ is real. Moreover, if $\Re(M) \geq 0$, \eqref{GaussInt} still holds in the sense of distribution.
        \end{proposition}

	\subsection{The real and complex symplectic groups}\label{subsec:23}
    We denote by
		\begin{equation}\label{defJ}
			J_d=\begin{pmatrix}
				O_d & I_d\\
				-I_d & O_d
			\end{pmatrix}
		\end{equation}
		the matrix of the canonical symplectic form of $\rdd$, 
		\begin{equation}\label{defsigmaR}
			\sigma(z,w)=J_dz\cdot w, \qquad z,w\in\rdd.
		\end{equation}
		We shall omit the subscript $d$ whenever it shall not cause confusion. A matrix $U\in\bR^{2d\times2d}$ is {\em symplectic} if $U^\top JU=J$. We denote by $\Sp(d,\bR)$ the group of $2d\times 2d$ real symplectic matrices. If
		\begin{equation}\label{blockS}
			\begin{pmatrix}
				A & B\\
				C & D
			\end{pmatrix}, \qquad A,B,C,D\in\bR^{d\times d},
		\end{equation}
		is the expression of $U$ in terms of its $d\times d$ blocks, then $U\in \Sp(d,\bR)$ if and only if
		\begin{align}
			\label{symp-rel1}
			& A^\top C=C^\top A,\\
			\label{symp-rel2}
			& B^\top D = D^\top B,\\
			\label{symp-rel3}
			& A^\top D-C^\top B=I.
		\end{align} 
		Similarly, the complex symplectic group $\Sp(d,\bC)$ is the group of matrices $S\in\bC^{2d\times 2d}$ satisfying $S^\top JS=J$. Again, $S\in\Sp(d,\bC)$ if and only if $S$ has block decomposition \eqref{blockS} with $A,B,C,D\in\bC^{d\times d}$ satisfying \eqref{symp-rel1}, \eqref{symp-rel2} and \eqref{symp-rel3}. If $S\in\Sp(d,\bC)$, $\det(S)=1$, cf. \cite{mackey2003determinant}. 
		We stress that there is another definition of complex symplectic group in the literature, where $S^\top$ is replaced by the Hermitian $S^\ast=\overline{S^\top}$, see e.g. \cite{Folland}. However, when discussing spaces of functions defined on $\rd$, as we do, the definition with the transpose is more natural.

        We recall that every symplectic matrix $U\in\Sp(d,\bR)$ admits a singular values decomposition with symplectic factors. Specifically, there exist $V,W\in\Sp(d,\bR)$ orthogonal and $\delta\in\Sp(d,\bR)$ diagonal such that $U=W\delta V^\top$. The diagonal elements of $\delta$ are the singular values of $U$. Recall that if $\lambda$ is a singular value of $U$, then also $1/\lambda$ is a singular value of $U$. Observe that this also implies that $U$ and $U^{-1}$ have the same singular values, and that the largest singular value of $U$ is $\geq1$. We refer to \cite{serafini} for the details. Moreover, by imposing $V^{-1}=V^\top$, it is easy to see that $V\in\Sp(d,\bR)$ is orthogonal if and only if
        \begin{equation}\label{blockS-orth}
            V=\begin{pmatrix}
                A & B\\
                -B & A
            \end{pmatrix}, \qquad A,B\in\bR^{d\times d},\, A^\top B\in\Sym(d,\bR), \, A^\top A+B^\top B=I_d.
        \end{equation}

		With an abuse of notation, we shall denote by $\sigma$ the extension of the canonical symplectic form \eqref{defsigmaR} to $\bC^{2d}$, i.e., 
		\begin{equation}
			\sigma(z,w)=Jz\cdot w, \qquad z,w\in\bC^{2d}.
		\end{equation}
		If $S\in\Sp(d,\bC)$ has blocks \eqref{blockS}, then
		\begin{equation}\label{blockS-1}
			S^{-1}=\begin{pmatrix}
				D^\top & -B^\top\\
				-C^\top & A^\top
			\end{pmatrix}.
		\end{equation}
		For the purpose of this work we also introduce the symplectic matrix
		\begin{equation}\label{blockShash}
			S^\#=\begin{pmatrix}
				D^\ast & -B^\ast\\
				-C^\ast & A^\ast
			\end{pmatrix},
		\end{equation}
		which is also in $\Sp_+(d,\bC)$.
        
		Now, we discuss the generators of the real and the complex symplectic groups. For $E\in\GL(d,\bC)$ and $Q\in\Sym(d,\bC)$, we set
		\begin{equation}\label{defVQ}
			V_Q=\begin{pmatrix}
				I & O\\
				Q & I
			\end{pmatrix}
		\end{equation} 
		and
		\begin{equation}\label{defDE}
			\cD_E=\begin{pmatrix}
				E^{-1} & O\\
				O & E^\top
			\end{pmatrix}.
		\end{equation} 
		
		The symplectic group $\Sp(d,\bC)$ is path-connected \cite[Theorem 4.7]{mackey2003determinant}. Hence, the very same argument proving \cite[Proposition 4.9]{Folland} shows that $\Sp(d,\bC)$ is generated by the complex analogues of the real generators. Precisely,
		
		\begin{proposition}\label{prop-gen-Sp}
			Let $\mathbb{F}$ be either $\bR$ or $\bC$. The symplectic group $\Sp(d,\mathbb{F})$ is generated by $J$, and by the matrices in \eqref{defVQ} and \eqref{defDE} with $E\in\GL(d,\mathbb{F})$ and $Q\in\Sym(d,\mathbb{F})$.
		\end{proposition}

        The blocks relations \eqref{symp-rel1}--\eqref{symp-rel3} defining $\Sp(d,\bR)$ conceal deeper structural properties of real symplectic matrices, that were studied in \cite{MR1906251}. We gather the ones we need in the following result.
        \begin{proposition}\label{propMO}
            Let $U\in\Sp(d,\bR)$ have blocks as in \eqref{blockS}. Then the maps $A:\ker(C)\to \mathrm{range}(C)^\perp$ and $B^\top:\mathrm{range}(A)^\perp\to\ker(A)$ are isomorphisms.
        \end{proposition}
        
		\subsection{The metaplectic group}\label{sub.meta}
		For $x,\xi\in\rd$ and $\tau\in\bR$, we consider the operator
		\begin{equation}\label{intertMetapDef}
			\rho(x,\xi;\tau)f(t)=e^{2\pi i\tau}e^{-i\pi x\cdot \xi}e^{2\pi i\xi\cdot t}f(t-x), \qquad f\in L^2(\rd).
		\end{equation}
		Then, $\rho$ is a unitary irreducible representation of the Heisenberg group on $L^2(\rd)$, and it is called the {\em Schrödinger representation}. For every $U\in\Sp(d,\bR)$, the representation $\rho_U(x,\xi;\tau)=\rho(U(x,\xi);\tau)$ is another irreducible representation of the Heisenberg group, with $\rho_U(0;\tau)=e^{2\pi i\tau}\id_{L^2}$. By Stone-von Neumann theorem, $\rho$ and $\rho_U$ are unitarily equivalent, i.e., there exists $\widehat U:L^2(\rd)\to L^2(\rd)$ unitary so that
		\begin{equation}\label{defMetap}
			\widehat U\circ\rho\circ \widehat U^{-1}=\rho_U.
		\end{equation} 
		The operator $\widehat U$ is unitary on $L^2(\rd)$ and it is called {\em metaplectic operator}. For any fixed $U\in\Sp(d,\bR)$, the operator $\widehat U$ satisfying \eqref{defMetap} is never unique, but if $\widehat U'$ is another such operator, then there exists $c\in\bC$, $|c|=1$ ({\em phase factor}), such that $\widehat U'=c\widehat U$. The group $\{\widehat U:U\in\Sp(d,\bR)\}$ has a subgroup containing exactly two operators for each symplectic matrix. This subgroup is called {\em metaplectic group}, and it is denoted by $\Mp(d,\bR)$. The cover $\pi^{Mp}:\widehat U\in\Mp(d,\bR)\mapsto U\in\Sp(d,\bR)$ is a group homomorphism with kernel $\{\pm\id_{L^2}\}$. For this reason, Proposition \ref{prop-gen-Sp} reflects on $\Mp(d,\bR)$ as follows.
		
		\begin{proposition}\label{propMPgen}
			The metaplectic group $\Mp(d,\bR)$ is generated by
			\begin{equation}\label{not.FT}
				\cF f(\xi)=i^{-d/2}\int_{\rd}f(x)e^{-2\pi i\xi \cdot x}dx=i^{-d/2}\widehat f(\xi), \qquad f\in L^1(\rd), \quad \xi\in\rd,
			\end{equation}
			and by the families of operators
			\begin{align}
				\label{defpQ}
				& \frp_Q f(x)=\Phi_Q(x)f(x), \qquad Q\in\Sym(d,\bR),\; f\in L^2(\rd),\\
				\label{defTE}
				& \frT_E f(x)= i^m |\det(E)|^{1/2}f(Ex), \qquad E\in\GL(d,\bR),\; f\in L^2(\rd),
			\end{align}
			where 
			\begin{equation}\label{chirp}
			\Phi_Q(x)=e^{i\pi Qx\cdot x}, \qquad x\in\rd,
		\end{equation}
			and $m\in\bZ$ is the argument of $\det(E)$ ({\em Maslov index}). Moreover, $\pi^{Mp}(\cF)=J$, $\pi^{Mp}(\frp_Q)=V_Q$ and $\pi^{Mp}(\frT_E)=\cD_E$, being these matrices defined as in \eqref{defJ}, \eqref{defVQ} and \eqref{defDE}, respectively.
		\end{proposition} 
        When it does not cause confusion, we shall omit the phase factors in \eqref{not.FT} and \eqref{defTE}. 
        
		The operator $\cF:f\in L^2(\rd)\mapsto\widehat f\in L^2(\rd)$ is, up to the phase factor $i^{-d/2}$, the {\em Fourier transform} operator \eqref{hatf}. Real metaplectic operators restrict to homeomorphisms of $\cS(\rd)$ and extend to homeomorphisms of $\cS'(\rd)$ as follows:
		\begin{equation}
			\la \widehat Uf,g\ra = \la f,\widehat U^{-1}g\ra, \qquad f\in\cS'(\rd), \; g\in\cS(\rd).
		\end{equation}

        We will use the following result, see \cite[Proposition 105]{MauriceDeGosson} with $\hbar=1/2\pi$.
        \begin{lemma}\label{lemmaUGauss}
            Let $\widehat U\in\Mp(d,\bR)$ with orthogonal projection as in \eqref{blockS-orth}. Let $\varphi(x)=e^{-\pi |x|^2}$. Then,
            \begin{equation}
                \widehat U\f(x)=e^{2\pi i\tau}\f(x)
            \end{equation}
            for some $\tau\in\bR$.
        \end{lemma}

		\subsection{(Classical) metaplectic Wigner distributions}
		Metaplectic Wigner distributions were defined in \cite{Cordero_Part_I} as in Definition \ref{intro.def11}. For the general theory, we address \cite{giacchi2025metaplecticTF} and the literature therein.
		For the $\cU$-Wigner distribution $W_\cU$, the following continuity properties follow from those of metaplectic operators.
		\begin{proposition}
			\begin{enumerate}[(i)]
				\item $W_\cU:\cS(\rd)\times\cS(\rd)\to\cS(\rdd)$ is continuous.
				\item For every $f_1,f_2,g_1,g_2\in L^2(\rd)$, Moyal's identity holds
				\begin{equation}\label{Moyal}
					\la W_\cU(f_1,g_1),W_\cU(f_2,g_2)\ra = \la f_1,f_2\ra\overline{\la g_1,g_2\ra}.
				\end{equation}
				In particular, $\norm{W_\cU(f,g)}_2=\norm{f}_2\norm{g}_2$ for every $f,g\in L^2(\rd)$.
				\item $W_\cU:\cS'(\rd)\times\cS'(\rd)\to\cS'(\rdd)$ is continuous.
			\end{enumerate}
		\end{proposition} 
		
		The main properties of a metaplectic Wigner distribution can be related to the structure of the projection of $\cU$, written in terms of its $d\times d$ blocks as
		\begin{equation}\label{blockA2}
			\begin{pmatrix}
				U_{11} & U_{12} & U_{13} & U_{14}\\
				U_{21} & U_{22} & U_{23} & U_{24}\\
				U_{31} & U_{32} & U_{33} & U_{34}\\
				U_{41} & U_{42} & U_{43} & U_{44}
			\end{pmatrix}, \qquad U_{i,j}\in\bR^{d\times d}, \quad i,j=1,\ldots,4.
		\end{equation}
        A systematic treatment of metaplectic Wigner distributions from this perspective was initiated in \cite{CR02}. We will discuss the other properties and relative references in Section \ref{sec.apptimefreq}.

        \subsection{Pseudodifferential operators} If $f,g\in\cS(\rd)$, then $W(f,g)\in\cS(\rdd)$. If $a\in\cS'(\rdd)$, we may consider the {\em Weyl-quantized pseudodifferential operator} (or, {\em Weyl operator}) with {\em symbol} $a$, that is the linear operator continuous from $\mathcal{S}(\R^d)$ to $\mathcal{S}'(\R^d)$ defined by
\begin{equation}\label{defOpwSp}
\langle \mathrm{Op}^\w(a)f,g \rangle=\langle a, W(g,f) \rangle, \quad f,g \in \mathcal{S}(\R^d).
\end{equation}

Here, we focus on symbols without decay, belonging to so-called $S_{0,0}^0$-H\"ormander class, for which we recall the definition here. 
\begin{definition}
Let $a\in C^\infty(\mathbb{R}^{2d})$. We say that {\em $a$ is a symbol of order $0$ without decay}, and write $a \in S_{0,0}^0(\mathbb{R}^{2d})$, if for each $\alpha \in \mathbb{N}_0^{2d}$ there exists a constant $C_\alpha>0$ such that 
\[
\abs{\partial_z^\alpha a(z)} \leq C_\alpha, \quad \forall z \in \mathbb{R}^{2d}. 
\]
\end{definition}

If $a\in S^0_{0,0}(\rdd)$, formula \eqref{defOpwSp} can be restated as the oscillatory integral

\begin{equation}
	\Op^{\mathrm{w}}(a)f(x)=\int_{\rdd}e^{2\pi i(x-y)\xi}a\Big(\frac{x+y}{2},\xi\Big)f(y)dyd\xi, \qquad f\in\cS(\rd).
\end{equation}
    The Weyl symbol of $\Op^\w(a)$ is related to its Schwartz kernel as follows. Recall that the Schwartz kernel of $\Op^\w(a)$ is the unique tempered distribution $k\in\cS'(\rdd)$ such that
    \begin{equation}\label{defkappa}
        \la \Op^\w(a)f,g\ra=\la k,g\otimes\bar f\ra, \qquad f,g\in\cS(\rd).
    \end{equation}
    Then,
    \begin{equation}\label{relak}
        a(x,\xi)=\int_{\rd} k(x+y/2,x-y/2)e^{-2\pi iy\cdot\xi}dy,
    \end{equation}
    where the integral is interpreted in a distributional sense, see \cite[Chapter 4]{CorderoBook} for the details.

We conclude this subsection by recalling the metaplectic invariance of Weyl H\"ormander calculus, stated as follows (see \cite[Theorem 2.1.1]{lernermetrics}).
\begin{theorem}\label{theo.metaweyl}
Let $\widehat{U} \in \Mp(d,\mathbb{R})$ with projection $U \in \Sp(d,\mathbb{R})$ and let $a \in \mathcal{S}'(\R^{2d})$. Then
\begin{equation}\label{intertOpwMetap}
	\widehat U^{-1}\Op^{\mathrm{w}}(a)\widehat U=\Op^{\mathrm{w}}(a\circ U).
\end{equation}
\end{theorem}
        \subsection{Modulation spaces}
        Modulation spaces were defined by Feichtinger in his celebrated work \cite{Feichtinger}. For $F:\rdd\to\bC$ measurable and $1\leq p,q\leq\infty$, we consider the mixed-Lebesgue norm
        \begin{equation}
            \norm{F}_{L^{p,q}}=\norm{y\mapsto\norm{F(\cdot,y)}_p}_q.
        \end{equation}
        Fix a {\em window} $g\in\cS(\rd)\setminus\{0\}$. Then, for every $f\in\cS'(\rd)$, the STFT $V_gf$ in \eqref{intro.defSTFT} defines a continuous function on $\rdd$. For $s\geq0$, we set $v_s(x,\xi)=(1+|(x,\xi)|^2)^{s/2}$, $x,\xi\in\rd$. The modulation space $M^{p,q}_{v_s}(\rd)$ is the space of tempered distribution $f\in\cS'(\rd)$ such that
        \begin{equation}\label{ModSpaceNorm}
            \norm{f}_{M^{p,q}_{v_s}}=\norm{V_gf}_{L^{p,q}_{v_s}}:=\norm{v_sV_gf}_{L^{p,q}}<\infty.
        \end{equation}
        Modulation spaces are Banach spaces and different choices of the window $g$ yield equivalent norms. For every $1\leq p,q\leq1$ and $s\geq0$,
        \begin{equation}
            \cS(\rd)\hookrightarrow M^{p,q}_{v_s}(\rd)\hookrightarrow\cS'(\rd).
        \end{equation}
        We abbreviate $M^{p,p}_{v_s}(\rd)=M^p_{v_s}(\rd)$ and $M^{p,q}_{v_0}(\rd)=M^{p,q}(\rd)$.

        \begin{remark}\label{Shubinrem}
        $M^2(\rd)=L^2(\rd)$ with equivalence of norms. Moreover, for $s\in\bR$, the modulation space $M^2_{v_s}(\R^d)$ coincides with the Shubin-Sobolev space which is defined as 
        \begin{equation}
            Q_s(\rd)=\Big\{f\in \cS'(\rd) ; \ \Lambda^sf \in L^2(\R^d) \Big\}, \qquad s\in\bR,
        \end{equation}
        with norm, 
        \[
        \|u\|_{Q_s}:=\|\Lambda^su \|_{2}, \quad u \in Q_s(\R^d),
        \]
        for $\Lambda^s=(1+|x|^2-\Delta)^{s/2}$ defined through the functional calculus, see \cite[Section 4.4.3.1] {CorderoBook} and \cite{Helffer}.
        More simplicity, for $k \in \mathbb{N}$, the Shubin-Sobolev space of order $k$ can be equivalently defined as  
        \[
        Q_k(\mathbb{R}^d)=\lbrace u \in L^2(\mathbb{R}^d); \quad x^\alpha \partial_x^\beta u \in L^2(\mathbb{R}^d), \ \forall \abs{\alpha}+\abs{\beta}\leq k \rbrace
        \]
        and
        \[
        \|u\|_{Q_k}=\Bigl(\sum_{\abs{\alpha}+\abs{\beta} \leq k}\|x^\alpha \partial_x^\beta u\|^2_{2}\Bigr)^2.
        \]
        \end{remark}
        
        For the general theory of modulation spaces, we refer to \cite{CorderoBook}. Among modulation spaces, we mention weighted Sj\"ostrand classes $M^{\infty,1}_{1\otimes v_s}(\rd)$. Their intersection is the H\"ormander class of symbols of order 0, i.e.,
        \begin{equation}
            S^0_{0,0}(\rdd)=\bigcap_{s\geq0}M^{\infty,1}_{1\otimes v_s}(\rd).
        \end{equation}
        Importantly, pseudodifferential operators with symbols in $S_{0,0}^0(\rdd)$, preserve the time-frequency concentration of signals, measured by the modulation spaces' norms.
        \begin{theorem}\label{thm20}
            Let $a\in S_{0,0}^0(\rdd)$. Then, $\Op^\w(a):M^{p,q}_{v_s}(\rd)\to M^{p,q}_{v_s}(\rd)$ for every $1\leq p,q\leq\infty$ and $s\geq0$, with 
            \begin{equation}
                \norm{\Op^\w(a)f}_{M^{p,q}_{v_s}}\lesssim \norm{a}_{M^{\infty,1}_{1\otimes v_s}}\norm{f}_{\cM^{p,q}_{v_s}}.
            \end{equation}
        \end{theorem}
        \begin{proof}
            It follows directly from Theorem 4.4.15 in \cite{CorderoBook}.
            
        \end{proof}

        The boundedness of metaplectic operators in $\Mp(d,\bR)$ on modulation spaces has been studied in \cite{FS2024}. A more quantitative approach, more similar to ours, was used in in \cite[Corollary 3.5]{CNTdispersion} for the unweighted $M^p$-spaces. Here, we obtain estimates for $\norm{\widehat U}_{op}$ in terms of the singular values of the projection $U$, for $M^{p,q}_{v_s}$-spaces. 

        \begin{theorem}\label{thmFS}
           Let $\widehat U\in \Mp(d,\bR)$ have projection $U$ with blocks as in \eqref{blockS}. Let $1\leq p,q\leq\infty$ and $s\geq0$.
            \begin{enumerate}[(i)]
                \item $\widehat U:M^p_{v_s}(\rd)\to M^p_{v_s}(\rd)$ is an isomorphism.
                \item If $C=O_d$ in \eqref{blockS}, then $\widehat U:M^{p,q}_{v_s}(\rd)\to M^{p,q}_{v_s}(\rd)$ is an isomorphism.
            \end{enumerate}
            In any case, if $\sigma_1,\ldots,\sigma_d$ are the singular values of $U$ that are $\geq1$, and $\sigma_{\mathrm{max}}(U)$ is the largest of them, then 
            \begin{equation}\label{normUB}
                \norm{\widehat U f}_{M^{p,q}_{v_s}}\leq c\cdot |\det(A)|^{1/p-1/q}{\sigma_{\mathrm{max}}(U)^{2s}}\prod_{j=1}^d\Big(\frac{\sigma_j}{1+\sigma_j^2}\Big)^{-1/2}\norm{f}_{M^{p,q}_{v_s}}
            \end{equation}
            where $c=c(g,s,d,p,q)>0$, being $g$ the window chosen to compute the $M^{p,q}_{v_s}$-norm, and we understand $1/\infty=1$ and $|\det(A)|^{1/p-1/q}=1$ if $p=q$ (also when $A=O_d$). 
        \end{theorem}
        \begin{proof}
            The statements concerning the isomorphisms were proven in \cite{FS2024}. We need to check the upper bound \eqref{normUB}. 
            Let $f\in M^{p,q}_{v_s}(\rd)$ and $g\in \cS(\rd)\setminus\{0\}$ be a window function. Similarly to \eqref{intro.WigCovarianceProp}, 
            \begin{equation}\label{step0}
                |V_g\widehat Uf(z)|=|V_{\widehat U^{-1}g}f(U^{-1}z)|, \qquad z\in\rdd,
            \end{equation}
            see e.g. \cite[Proposition 1.2.14]{CorderoBook}. Therefore,
            \begin{align}\label{step05}
                \norm{\widehat Uf}_{M^{p,q}_{v_s}}&=\norm{V_g\widehat Uf}_{L^{p,q}_{v_s}}=\norm{(V_{\widehat U^{-1}g}f)(U^{-1}\cdot)}_{L^{p,q}_{v_s}}.
            \end{align}
            If $U$ has blocks \eqref{blockS} with $C=O_d$ and $F=V_{\widehat U^{-1}g}f$, then    
            \begin{align}\label{step1}
                \norm{v_s(\cdot) F(U^{-1}\cdot)}_{L^{p,q}}&=\norm{((v_s(U\cdot) F)(U^{-1}\cdot)}_{L^{p,q}}.
            \end{align}
            In the following, we argue for $p,q\neq\infty$, the discussion for the case $\max\{p,q\}=\infty$ is similar. For $G=v_s(U\cdot)F$, if $p=q$ we obviously have $\norm{G\circ U^{-1}}_p=\norm{G}_p$, since $\det(U)=1$. Otherwise, using the expression \eqref{blockS-1} for the inverse of $U$ (which is now assumed to be upper-block triangular), we find
            \begin{align}
                \norm{G\circ U^{-1}}_{L^{p,q}}&=\Big(\int_{\rd}\Big(\int_{\rd}|G(U^{-1}(x,\xi))|^pdx\Big)^{q/p}d\xi\Big)^{1/q}\\
                &=\Big(\int_{\rd}\Big(\int_{\rd}|G(D^\top x-B^\top \xi,A^\top \xi)|^pdx\Big)^{q/p}d\xi\Big)^{1/q}\\
                &=|\det(D)|^{-1/p}\Big(\int_{\rd}\Big(\int_{\rd}|G(y,A^\top \xi)|^pdy\Big)^{q/p}d\xi\Big)^{1/q}\\
                &=|\det(D)|^{-1/p}|\det(A)|^{-1/q}\norm{G}_{L^{p,q}}.
            \end{align}
            If $C=O_d$, then $A=D^{-\top}$, from which we infer:
            \begin{align}
                \norm{G\circ U^{-1}}_{L^{p,q}}&=|\det(A)|^{1/p-1/q}\norm{G}_{L^{p,q}}.
            \end{align}
            By plugging this information in \eqref{step1}, we obtain
            \begin{align}
                \norm{v_s(\cdot)F(U^{-1}\cdot)}_{L^{p,q}}&=|\det(A)|^{1/p-1/q}\norm{v_s(U\cdot) F}_{L^{p,q}},
            \end{align}
            where we interpret $|\det(A)|^{1/p-1/q}=1$ whenever $p=q$. Since $v_s(Uz)\leq\sigma_{\mathrm{max}}(U)^sv_s(z)$, where $\sigma_{\mathrm{max}}(U)$ is the largest singular value of $U$, we obtain
            \begin{align}
                \norm{ F(U^{-1}\cdot)}_{L^{p,q}_{v_s}}&\leq |\det(A)|^{1/p-1/q}\sigma_{\mathrm{max}}(U)^s\norm{ F}_{L^{p,q}_{v_s}}.
            \end{align}
            Let us return to \eqref{step05}. We have
            \begin{align}
                \norm{\widehat Uf}_{M^{p,q}_{v_s}}&\leq |\det(A)|^{1/p-1/q}\sigma_{\mathrm{max}}(U)^s \norm{V_{\widehat U^{-1}g}f}_{L^{p,q}_{v_s}}.
            \end{align}
            Next, we replace the window $\widehat U^{-1}g$ with $g$ by resorting to the standard convolution inequality in \cite[Lemma 1.2.29]{CorderoBook} 
            \begin{align}
                 \norm{\widehat Uf}_{M^{p,q}_{v_s}}&\leq |\det(A)|^{1/p-1/q} \frac{\sigma_{\mathrm{max}}(U)^s}{|\la \widehat Ug,g\ra|}\norm{|V_gf|\ast |V_{\widehat U^{-1}g}\widehat U^{-1}g|}_{L^{p,q}_{v_s}}.
            \end{align}
            By Young's inequality for weighted mixed norm spaces, see e.g. \cite[Theorem 2.2.3]{CorderoBook},
            \begin{align}
                \norm{|V_gf|\ast |V_{\widehat U^{-1}g}\widehat U^{-1}g|}_{L^{p,q}_{v_s}}&\leq c' \norm{V_gf}_{L^{p,q}_{v_s}}\norm{V_{\widehat U^{-1}g}\widehat U^{-1}g}_{L^1_{w_s}}\\
                &=c'\norm{V_gf}_{L^{p,q}_{v_s}}\norm{V_{ g} g(U\cdot)}_{L^1_{w_s}}\\
                &\leq c'\sigma_{\mathrm{max}}(U^{-1})^s \norm{V_gf}_{L^{p,q}_{v_s}}\norm{V_gg}_{L^1_{w_s}}\\
                &=c'\sigma_{\mathrm{max}}(U)^s\norm{V_gg}_{L^1_{w_s}}\norm{V_gf}_{L^{p,q}_{v_s}},
            \end{align}
            where $w_s(z)=(1+|z|)^s$, $c'>0$ is the constant appearing in Young's inequality, depending only on $s$, $p$ and $q$.
            Here, we are using that if $U$ is symplectic and if $\lambda$ is a singular value of $U$, then also $\lambda^{-1}$ is a singular value of $U$, and as a general fact the singular values of $U^{-1}$ are the inverses of the singular values of $U$.  
            Now, we choose the Gaussian window $g(y)=e^{-\pi |y|^2}$ and write the singular values decomposition of $U$ as $U=W^\top \delta V$, where $V$ and $W$ are orthogonal and symplectic, and 
            \begin{equation}
                \delta=\cD_E=\begin{pmatrix}
                    E^{-1} & O_d\\
                    O_d & E^\top
                \end{pmatrix}, \qquad E=\diag(\sigma_1^{-1},\ldots,\sigma_d^{-1}),
            \end{equation}
            where $\sigma_1,\ldots,\sigma_d$ are the singular values of $U$ that are $\geq1$. Then,
            \begin{equation}
                \la \widehat U g,g \ra=\la \widehat\delta \widehat{V}g,\widehat{W}g\ra.
            \end{equation}
            Since $V$ and $W$ are orthogonal, by Lemma \ref{lemmaUGauss}, $\widehat{V}g=e^{2\pi i\tau}g$ and $\widehat{W}g=e^{2\pi i\tau'}g$ for some $\tau,\tau'\in\bR$. So, 
            \begin{equation}\label{Ugg2}
                |\la \widehat U g,g\ra |=|\la \widehat \delta g,g\ra|=\la \widehat \delta g,g\ra.
            \end{equation}
            Let $\sigma_1,\ldots,\sigma_d$ be the eigenvalues of $\delta$ that are $\geq1$, i.e., the singular values of $U$ that are $\geq1$. Recall that the other eigenvalues are $1/\sigma_j$ for $j=1,\ldots,d$. Then, up to a phase factor,
            \begin{align}\label{Ugg}
                \la \widehat U g,g\ra&=\int_{\rd}\widehat \delta g(x)g(x)dx=\prod_{j=1}^d\sigma_j^{-1/2}\int_{-\infty}^\infty e^{-\pi(1+\sigma_j^{-2})x_j^2}dx_j=\prod_{j=1}^d\Big(\frac{\sigma_j}{1+\sigma_j^2}\Big)^{1/2}.
            \end{align}
            We thereby conclude that
            \begin{equation}
                \norm{\widehat U f}_{M^{p,q}_{v_s}}\leq c' |\det(A)|^{1/p-1/q}{\sigma_{\mathrm{max}}(U)^{2s}}\prod_{j=1}^d\Big(\frac{\sigma_j}{1+\sigma_j^2}\Big)^{-1/2}\norm{V_gg}_{L^1_{w_s}}\norm{f}_{M^{p,q}_{v_s}}.
            \end{equation}
            If $g$ is not a Gaussian, there exists a constant $b>0$ depending only on the new window, such that
             \begin{equation}
                \norm{\widehat U f}_{M^{p,q}_{v_s}}\leq c' b |\det(A)|^{1/p-1/q}\sigma_{\mathrm{max}}(U)^{2s}\prod_{j=1}^d\Big(\frac{\sigma_j}{1+\sigma_j^2}\Big)^{-1/2}\norm{V_gg}_{L^1_{w_s}}\norm{f}_{M^{p,q}_{v_s}}.
            \end{equation}
            This concludes the proof.
            
        \end{proof}

        \begin{remark} 
            The estimates in the previous result is not optimal. For $p=q=2$ and $s=0$, we indeed obtain
            \begin{equation}
               \norm{\widehat Uf}_2\leq c\prod_{j=1}^d\Big(\frac{\sigma_j}{1+\sigma_j^2}\Big)^{-1/2}\norm{f}_2,
            \end{equation}
            while we know that $\norm{\widehat Sf}_2=\norm{f}_2$.
            A more precise estimate for $p=q=2$ and $s=0$ was obtained in \cite[Corollary 3.5]{CNTdispersion} using interpolation results. There, the authors prove 
            \begin{equation}
                \norm{\widehat Sf}_{M^p}\lesssim \Big(\prod_{j=1}^d\sigma_j\Big)^{|1/2-1/p|}\norm{f}_{M^p}.
            \end{equation}
            The price we pay for obtaining more general estimates for $p \neq q$ and $s>0$ via Young's inequality is the appearance of the factor
            \begin{equation}
                \prod_{j=1}^d \Bigl(\frac{\sigma_j}{1+\sigma_j^2}\Bigr)^{-1/2},
            \end{equation}
            with exponent $-1/2$, which arises from the term $\langle \widehat{U} g, g \rangle$, and does not depend on $p$ and $q$. We expect sharper estimates can be obtained through Wigner analysis, a finer investigation which is beyond the scope of this manuscript.
        \end{remark}

        \begin{remark}
            Another formulation of \eqref{normUB} can be obtained by using \cite[Proposition 105]{MauriceDeGosson} in \eqref{Ugg2}. This gives
            \begin{equation}
                |\la \widehat Ug,g\ra| = |\det(I_d-i\cC(U)^{-1})|^{-1/2},
            \end{equation}
            where $\cC(U)=(C+iD)(A+iB)^{-1}$ is the Cayley transform of $U$, see also \cite{Folland}. Consequently,
            \begin{equation}
                \norm{\widehat Uf}_{M^{p,q}_{v_s}}\lesssim |\det(A)|^{1/p-1/q}\sigma_{\mathrm{max}}(U)^{2s}|\det(I_d-i\cC(U)^{-1})|^{1/2}\norm{f}_{M^{p,q}_{v_s}}.
            \end{equation}
        \end{remark}

		\section{The metaplectic semigroup}\label{sec.thecomplexmeta}
		In this section, we collect the main facts about the metaplectic semigroup from the works of Howe, Brunet, Kramer and H\"ormander. 
		
        \subsection{The positivity condition}\label{sec:HPC}
		It is impossible to extend $\Mp(d,\bR)$ to a group of operators that are bounded on $L^2(\rd)$, so that $\pi^{Mp}$ extends consistently to a double cover of $\Sp(d,\bC)$. A selection of {\em admissible} complex symplectic matrices is needed a priori.
		\begin{definition}\label{defPosSymp}
			A matrix $S\in\Sp(d,\bC)$ is {\em positive} if 
		\begin{equation}\label{cond.pos}
       			i^{-1}\big(\sigma(Sz,\overline{Sz})-\sigma(z,\overline z)\big)\geq0, \qquad z\in\bC^{2d},
  		\end{equation}
  		where $\sigma$ is the complexified standard symplectic form. 
		We denote by $\Sp_+(d,\bC)$ the semigroup of positive symplectic matrices. Instead, we say that $S$ is {\em strictly positive} if \eqref{cond.pos} holds with $>0$. 
		\end{definition}
		The semigroup $\Sp_+(d,\bC)$ is denoted as $\mathcal{C}_+$ in \cite{Hormander3}. 
		We refer to \cite[Lemma 3.2]{Hormander3} for a proof of the semigroup property, in the more general framework of {\em positive linear canonical relations}. 
		
		\noindent
		It will be useful to express condition \eqref{cond.pos} more explicitly, as we shall do promptly.
		
		\begin{proposition}\label{SmathcalM}
			Let $S\in\Sp(d,\bC)$. Then, by writing $S=S_R+iS_I$, $S_R,S_I\in\bR^{2d\times2d}$, one has that $S$ is positive (respectively strictly positive) if and only if the matrix 
			\begin{equation}\label{defMS}
				\mathcal{M}_S=\begin{pmatrix}
					S_R^\top JS_I & S_I^\top JS_I \\ 
					-S_I^\top JS_I & S_R^\top JS_I
				\end{pmatrix} \in \bR^{4d\times4d}
			\end{equation}
			is positive semi-definite (respectively positive-definite).
		\end{proposition}
		\begin{remark}\label{S_ReS_I}
			Notice that the condition $S^\top JS=J$ defining $\Sp(d,\bC)$ can be rephrased in terms of $\cM_S$ as
			\[
				S_R^\top JS_R-S_I^\top JS_I=J, \quad S_I^\top JS_R+S_R^\top JS_I=0.
			\]
		\end{remark}
		\begin{proof}[Proof of Proposition \ref{SmathcalM}]
			During the proof we denote by $z=x+i\xi \in \mathbb{C}^{2d}$ and 
			\begin{align}
				Sz&=S_Rx+iS_I\xi+iS_Ix-S_I\xi=(S_Rx-S_I\xi)+i(S_I\xi+S_Ix)=:\widetilde{x}+i\widetilde{\xi}.
			\end{align}
			Under this notation, we have that $S$ is positive if and only if  
			\[
				i^{-1}\big(\sigma(\widetilde{x}+i\widetilde{\xi}, \widetilde{x}-i \widetilde{\xi})-\sigma(x+i\xi,x-i\xi)\big) \geq 0, \qquad  x,\xi \in \mathbb{R}^{2d},
			\]
			which is equivalent to (recall that $\sigma$ is antisymmetric)
			\[
				i^{-1}\big(2\sigma(i\widetilde{\xi}, \widetilde{x})-2\sigma(i\xi,x)\big)\geq 0, \qquad x,\xi \in \mathbb{R}^{2d}.
			\]
			Therefore, $S$ is positive if and only if 
			\[
				\sigma(\widetilde{\xi},\widetilde{x})-\sigma(\xi,x) \geq 0, \qquad x,\xi \in \mathbb{R}^{2d},
			\]
			that is 
			\[
				\sigma(S_R\xi+S_Ix,S_Rx-S_I\xi)-\sigma(\xi,x) \geq 0, \qquad x,\xi \in \mathbb{R}^{2d}.
			\]
			Equivalently,
			\[
				J(S_R\xi+S_Ix)\cdot(S_Rx-S_I\xi) - J \xi\cdot x \geq 0, \qquad x,\xi \in \mathbb{R}^{2d},
			\]
			and so
			\[
				 S_R^\top JS_R \xi\cdot x +  S_R^\top JS_I x\cdot x  -  S_I^\top JS_R \xi\cdot \xi - S_I^\top JS_I x\cdot \xi - J \xi\cdot x \geq 0,
			\]
			for every $x,\xi \in \mathbb{R}^{2d}$. By Remark \ref{S_ReS_I}, $S\in\Sp_+(d,\bC)$ if and only if 
			\begin{equation}\label{defposMS}
				\mathcal{M}_S=\begin{pmatrix}
					S_R^\top JS_I & S_I^\top JS_I \\ 
					-S_I^\top JS_I & S_R^\top JS_I
				\end{pmatrix} 
				\begin{pmatrix}
					x \\
					\xi
				\end{pmatrix}\cdot \begin{pmatrix}
					x \\
					\xi
				\end{pmatrix} \geq0, \qquad (x,\xi) \in \mathbb{R}^{4d},
			\end{equation}
			where $\mathcal{M}_S$ is symmetric since, by Remark \ref{S_ReS_I}, 
			\[
				\mathcal{M}_S^\top =\begin{pmatrix}
					-S_I^\top JS_R & S_I^\top JS_I\\
					-S_I^\top JS_I & -S_I^\top J S_R
				\end{pmatrix}=\begin{pmatrix}
					S_R^\top JS_I & S_I^\top JS_I\\
					-S_I^\top JS_I & S_R^\top JS_I
				\end{pmatrix}=\mathcal{M}_S.
			\] 
			The proof for strict positivity is identical. 
			
		\end{proof}
		
		Our next goal is to obtain a description of $\Sp_+(d,\bC)$, in terms of its generators, as explicit as possible. To this end, for $\vartheta\geq0$, let
		\begin{equation}\label{Rt}
			\cR_\vartheta=\begin{pmatrix}
				\cosh \vartheta & -i\sinh \vartheta \\
				i\sinh \vartheta & \cosh \vartheta
				\end{pmatrix}\in\Sp(1,\bC)
		\end{equation}
		and for $\alpha>0$ let  
		\begin{equation}\label{Va}
			V_{i\alpha}=\begin{pmatrix}
				1 & 0 \\
				i\alpha & 1
			\end{pmatrix}\in\Sp(1,\bC).
		\end{equation}
		
		Positive symplectic matrices admit a very beneficial polar-type decomposition, see \cite[Proposition 5.10]{Hormander3} and \cite[Theorem 2.1]{Brunet}. To simplify the next statement and the discussion that follows, let us define the following class of positive symplectic matrices (recall \eqref{tensmat} for the definition of $\otimes$).
		\begin{definition}\label{defPurePos}
			A {\em purely positive} symplectic matrix is any matrix $\Xi\in\Sp_+(d,\bC)$ in the form 
			\begin{equation}
				\Xi=\Xi_1\otimes\ldots\otimes \Xi_d,
			\end{equation} 
			with $\Xi_j=I_{2\times 2}$, or $\Xi_j=\cR_{\vartheta_j}$ ($\vartheta_j>0$), or $\Xi_j=V_{i\alpha_j}$ ($\alpha_j>0$), $j=1,\ldots,d$.
		\end{definition}
				
		\begin{theorem}\label{theo.class}
			The semigroup $\Sp_+(d,\bC)$ is generated by $\Sp(d,\bR)$ and by purely positive symplectic matrices. Moreover, every matrix $S\in\Sp_+(d,\bC)$ can be factorized as 
			\begin{equation}\label{factorization}
				S=U_1\Xi U_2,
			\end{equation}
			where $U_1,U_2\in\Sp(d,\bR)$ and $\Xi$ is purely positive.
		\end{theorem} 
		Let us compare the statements in the works of H\"ormander and Brunet that, together, give Theorem \ref{theo.class}.
		\begin{remark}\label{rem-equivalences}
			Although apparently different, \cite[Proposition 5.10]{Hormander3} and \cite[Theorem 2.1]{Brunet} are equivalent. To see it, three points must be discussed.
						
			First, Brunet's approach uses the embedding $\Sp(d,\bR)\hookrightarrow\Sp(d,\bC)$ given by the conjugation with 
			\begin{equation}
				\cW=\frac{1}{\sqrt{2}}\begin{pmatrix}
					I_d & iI_d\\
					I_d & -iI_d
				\end{pmatrix},
			\end{equation}
			see also \cite{Folland}, where the very same notation is used, and the Examples in Brunet's \cite[Section V]{Brunet}, where the matrices $\cR_\vartheta$ and $V_{i\alpha}$ are derived explicitly. Precisely, $U\in\Sp(d,\bR)$ is identified with $\cW U\cW^{-1}\in\Sp(d,\bC)$. In contrast, H\"ormander does not resort to this identification. Moreover, in his construction, Brunet uses
			\begin{equation}\label{matBrunet}
				\cW^{-1}\begin{pmatrix}
					1-\alpha & -\alpha\\
					\alpha & 1+\alpha
				\end{pmatrix}\cW=V_{2i\alpha},
			\end{equation}
			rather than $V_{i\alpha}$. This difference is irrelevant, since in turns $V_{i\alpha}$ and $V_{2i\alpha}$ can be both replaced by $V_{i}$, as we hereby explain.
			
			In H\"ormander's \cite[Proposition 5.10]{Hormander3}, only $\alpha=1$ in \eqref{Va} is allowed, but this is not restrictive: if $\alpha>0$,
			\begin{equation}
				V_{i\alpha}=\cD_{\alpha^{1/2}}V_{i}\cD_{\alpha^{-1/2}},
			\end{equation}
			and the matrices $\cD_{\alpha^{1/2}},\cD_{\alpha^{-1/2}}$ are in $\Sp(1,\bR)$, so the contributions of every $\alpha_j\neq1$ can be embodied in $U_1$ and $U_2$ in \eqref{factorization}, respectively.
			
			Finally, in \cite[Proposition 5.10]{Hormander3}, the matrices $\cR_{\tau_j}$ and $V_{i\alpha_j}$ constructing $\Xi$ can be arranged in any order. However, in \cite[Theorem 2.1]{Brunet}, the matrix $\Xi$ is constructed by assembling the (conjugate via $\cW$ of the blocks) $V_{i\alpha}$ firstly, and the blocks $\cR_{\vartheta_j}$ secondly. Again, the two constructions are equivalent, since the matrices $\cR_\vartheta\otimes V_i$ and $V_i\otimes\cR_\vartheta$ are conjugate through the real symplectic matrix $\cD_L$, with
			\begin{equation}
				L=\begin{pmatrix}
				0 & 1\\
				1 & 0
				\end{pmatrix},
			\end{equation}
			i.e., (observe that $\cD_L^{-1}=\cD_L$)
			\begin{equation}
				\begin{pmatrix}
					1 & 0 & 0 & 0\\
					0 & \cosh\vartheta & 0 & -i\sinh\vartheta\\
					i & 0 & 1 & 0\\
					0 & i\sinh\vartheta & 0 & \cosh\vartheta
				\end{pmatrix}=\cD_L 
				\begin{pmatrix}
					\cosh\vartheta & 0 & -i\sinh\vartheta & 0\\
					0 & 1 & 0 & 0\\
					i\sinh\vartheta & 0 & \cosh\vartheta & 0\\
					0 & i & 0 & 1
				\end{pmatrix}
				\cD_L.
			\end{equation}
			The equivalence for $d>2$ is just a mere generalization of the preceding argument.
		\end{remark}
        
		We conclude this section by stating the role of $\Sp(d,\bR)$ as a subgroup of $\Sp_+(d,\bC)$ and the topological properties of $\Sp_+(d,\bC)$.
		
		\begin{theorem}\cite[Theorem 5.12]{Hormander3} 
			A matrix $S\in\Sp_+(d,\bC)$ is invertible if and only if it is a real symplectic matrix.
		\end{theorem}
		
		For the following result we refer to \cite[Theorem 2.5]{Brunet} and \cite[Section 5]{Hormander3}.
		
		\begin{theorem}
			The semigroup $\Sp_+(d,\bC)$ is connected and closed in $\Sp(d,\bC)$. Moreover, the interior of $\Sp_+(d,\bC)$ is the set (semigroup) of {strictly positive symplectic matrices}, and its boundary contains $\Sp(d,\bR)$.
		\end{theorem}
		
		We now turn to the definition of the \emph{metaplectic semigroup} $\Mp_+(d,\bC)$.
		
		\subsection{Definition and main properties of the metaplectic semigroup}\label{subsec:32}
		
		Following \cite{Hormander3}, for $\vartheta\geq0$, we consider
		\begin{equation}\label{defRtau}
			\mathfrak{R}_\vartheta f(x) =(\cosh \vartheta)^{-1/2} \int_{-\infty}^\infty e^{i\pi (2x\eta +i(x^2+\eta^2)\sinh\vartheta)/\cosh\vartheta}\widehat{f}(\eta) d\eta, \quad f \in \cS(\mathbb{R}),
		\end{equation}
 		whereas for $\alpha>0$,
		\begin{equation}\label{defpialpha}
			\frp_{i\alpha} f(x)=e^{-\pi \alpha x^2}f(x), \quad f \in \cS(\mathbb{R}).
		\end{equation}
		These operators are contractions of $L^2(\bR)$. This is trivial for $\frp_{i\alpha}$, whereas a simple computation shows that 
		\begin{equation}
			\mathfrak{R}_\vartheta=\frp_{i\tanh\vartheta}\mathfrak{T}_{1/\cosh\vartheta}\cF^{-1}\frp_{i\tanh\vartheta}\cF, 
		\end{equation}
		which is a composition of contractions. 
		
		For every $j=1,\ldots,d$, let $\widehat \Xi_j=\mathfrak{R}_{\vartheta_j}$ ($\vartheta_j\geq0$) or $\widehat \Xi_j=\frp_{i\alpha_j}$ ($\alpha_j>0$). Then, there exists a unique operator
		\begin{equation}\label{tensorOp}
			\widehat \Xi=\widehat \Xi_1\otimes\ldots\otimes \widehat \Xi_d,
		\end{equation}
		whose action on $L^2(\rd)$ is characterized by the action on tensor products
		\begin{equation}
			\widehat \Xi(f_1\otimes\ldots\otimes f_d)=\bigotimes_{j=1}^d\widehat \Xi_j f_j, \qquad f_1,\ldots,f_d\in L^2(\bR),
		\end{equation}
		see e.g. \cite[Theorem 2.6.12]{Kadison}. This operator is bounded on $L^2(\rd)$. For every $j=1,\ldots,d$, let 
		\begin{equation}
			\Xi_j=\begin{cases}
				\cR_{\vartheta_j} & \text{if $\widehat \Xi_j=\mathfrak{R}_{\vartheta_j}$},\\
				V_{i\alpha_j} & \text{if $\widehat \Xi_j=\frp_{i\alpha_j}$}.
			\end{cases}
		\end{equation}
		Let $\Xi=\Xi_1\otimes\ldots\otimes \Xi_d$. By construction, $\Xi$ is a purely positive symplectic matrix.
        \begin{definition}
            The operators $\widehat \Xi$ in \eqref{tensorOp} are called {\em atomic metaplectic contractions}.
        \end{definition}
		\begin{definition}\label{defOscillatorSemigroup}
			The {\em metaplectic semigroup} $\Mp_+(d,\bC)$ is the semigroup generated by $\Mp(d,\bR)$ and atomic metaplectic contractions. Precisely, $\widehat S\in \Mp_+(d,\bC)$ if $\widehat S=\widehat S_1\circ\ldots\circ\widehat S_N$ for some $N\geq 1$ and each $\widehat S_j$ ($j=1,\ldots,N$) is either an atomic metaplectic contraction or in $\Mp(d,\bR)$. We call any operator $\widehat S\in \Mp_+(d,\bC)$ {\em complex metaplectic operator}. 
		\end{definition}
		Under the notation above, we define a homomorphism $\pi^{Mp}_+:\Mp_+(d,\bC)\to\Sp_+(d,\bC)$ by its action on the generators of $\Mp_+(d,\bR)$ as follows. We set:
		\begin{equation}
			\pi^{Mp}_+(\widehat U)=\pi^{Mp}(\widehat U)
		\end{equation}
		if $\widehat U\in \Mp(d,\bR)$, while if $\widehat \Xi=\widehat \Xi_1\otimes\ldots\otimes \widehat \Xi_d$ is an atomic metaplectic contraction, we set
		\begin{equation}
		\pi^{Mp}_+(\widehat \Xi)=\Xi_1\otimes\ldots\otimes \Xi_d. 
		\end{equation}
		
		The contents of \cite[Theorem 5.12]{Hormander3} and \cite[Proposition 4.2]{Brunet} are summarized in the following Theorem.
		
		\begin{theorem}\label{thmOscillator}
			\begin{enumerate}[(i)]
			\item The homomorphism $\pi^{Mp}_+:\Mp_+(d,\bC)\to \Sp_+(d,\bC)$ is a two-fold cover of  $\Sp_+(d,\bC)$. The mapping $S\in\Sp_+(d,\bC)\mapsto \widehat S\in\Mp_+(d,\bC)$ is a projective representation. 
			\item For every $\widehat S\in\Mp_+(d,\bC)$ there exist $\widehat U_1,\widehat U_2\in\Mp(d,\bR)$ and $\widehat \Xi$ atomic metaplectic contraction such that 
            \begin{equation}\label{HFactoriz}
                \widehat S=\widehat U_1\widehat \Xi\widehat U_2.
            \end{equation}
			\item Every complex metaplectic operator is an injective contraction.
			\item $\widehat S\in \Mp_+(d,\bC)$ is invertible if and only if $\widehat S\in\Mp(d,\bR)$.
			\item The image of $\widehat S\in\Mp_+(d,\bC)$ is dense in $L^2(\rd)$, and it coincides with $L^2(\rd)$ if and only if $\widehat S\in \Mp(d,\bR)$.
			\item If $\widehat S\in\Mp_+(d,\bC)$, then $\widehat S^\ast\in\Mp_+(d,\bC)$, and $\widehat S^\ast = \widehat{S^\#}$, where $\pi^{Mp}_+(\widehat S^\ast)=S^\#$ has blocks as in \eqref{blockShash}.
			\item The inverses of the elements of $\Mp_+(d,\bC)$ are closed densely defined operators on $L^2(\rd)$.
			\item Every operator in $\Mp_+(d,\bC)$ restricts to a continuous operator from $\cS(\rd)$ to itself. Moreover, if $\widehat S\in\Mp_+(d,\bC)$, for $f\in\cS'(\rd)$
			\begin{equation}
				\langle \widehat Sf,g\rangle=\langle f,\widehat S^\ast g\rangle, \qquad g\in\cS(\rd),
			\end{equation}
			extends $\widehat S$ to a continuous operator from $\cS'(\rd)$ to itself.
			\end{enumerate}
		\end{theorem}

        We stress that, in view of item $(i)$ of the preceding Theorem \ref{thmOscillator}, $S\in\Sp_+(d,\bC)$ identifies $\widehat S$ up to a constant $c\in\bC\setminus\{0\}$, with $|c|\leq1$, in order to obtain contractions. In absence of ambiguity, we will omit the choice of $c$ in the following discussion.

        \begin{remark}\label{remAMC}
            Let us summarize the properties of atomic metaplectic contractions in this remark. Every such operator is the tensor product of $d$ ``atoms": \begin{equation}
                \widehat\Xi=\bigotimes_{j=1}^d\widehat\Xi_j,
            \end{equation}
            where for each $j=1,\ldots,d$, 
            \begin{equation}
                \widehat\Xi_j=\begin{cases}
                    \mathfrak{p}_{i\alpha_j} & \text{for some $\alpha_j>0$},\\
                    \id_{L^2} \\
                    \frR_{\vartheta_j} & \text{for some $\vartheta_j>0$}.
                \end{cases}
            \end{equation}
            The corresponding projection is
            \begin{equation}
                \Xi=\bigotimes_{j=1}^d\Xi_j
            \end{equation}
            where
            \begin{equation}
                \Xi_j=\begin{cases}
                    V_{i\alpha_j} & \text{if $\widehat\Xi_j=\frp_{i\alpha_j}$,}\\
                    I & \text{if $\widehat\Xi_j=\id_{L^2}$,}\\
                    \mathcal{R}_{\vartheta_j} & \text{if $\widehat\Xi_j=\frR_{\vartheta_j}$}.
                \end{cases}
            \end{equation}
            Observe that, by the equivalence between H\"ormander's and Brunet's approaches, we may restrict to consider $\alpha_j=1$, when it is convenient.
            We can divide the contributions of the submatrices in the form $\cR_{\vartheta_j}$ and those in the form $V_{i\alpha_j}$ as follows
		\begin{equation}\label{subdivision}
			\Xi=\underbrace{\begin{pmatrix}
				\diag(a_1,\ldots,a_d) & -\diag(b_1,\ldots,b_d)\\
				\diag(b_1,\ldots,b_d) & \diag(a_1,\ldots,a_d)
			\end{pmatrix}}_{=:\Xi_1}
			\underbrace{\begin{pmatrix}
				I & O\\
				i\diag(\lambda_1,\ldots,\lambda_d) & I
			\end{pmatrix}}_{=:\Xi_2},
		\end{equation}
		where
		\begin{equation}
			a_j=\begin{cases}
				\cosh\vartheta_j & \text{if $\Xi_j=\cR_{\vartheta_j}$},\\
				1 & \text{otherwise},
			\end{cases}\qquad 
			b_j=\begin{cases}
				i\sinh\vartheta_j & \text{if $\Xi_j=\cR_{\vartheta_j}$},\\
				0 & \text{otherwise},
			\end{cases}\qquad 
			\lambda_j=\begin{cases}
				\alpha_j & \text{if $\Xi_j=V_{i\alpha_j}$},\\
				0 & \text{otherwise}.
			\end{cases}
		\end{equation}
        Observe that by construction $\Xi_1$ and $\Xi_2$ commute, and $\alpha_j\vartheta_j=0$ for every $j=1,\ldots,d$. 
        We write $\Xi_1=\mathcal{R}_\Theta$ and $\Xi_2=V_{i\Delta}$, with $\Theta=(\vartheta_1,\ldots,\vartheta_d)$, where we imply $\vartheta_j=0$ if $\Xi_j$ is the identity or in the form $V_{i\alpha_j}$ ($\alpha_j>0$), and $\Delta=\diag(\alpha_1,\ldots,\alpha_d)$ where, again, we assume $\alpha_j=0$ if $\Xi_j=I$ or in the form $\Xi_j=\mathcal{R}_{\vartheta_j}$ ($\vartheta_j>0$). Consequently, if $\widehat\Xi$ is an atomic metaplectic contraction, we can write
        \begin{equation}\label{S2Vp}
            \widehat\Xi=\frp_{i\Delta}\frR_{\Theta}=\frR_{\Theta}\frp_{i\Delta},
        \end{equation}
        here $\frp_{i\Delta}f=\bigotimes_{j=1}^d\frp_{i\alpha_j}$ and $\frR_{\Theta}f=\bigotimes_{j=1}^d\frR_{\vartheta_j}$, where again we imply $\alpha_j=0$ and $\vartheta_j=0$ if $\Xi_j=I$. 
        \end{remark}
        
	\section{Analysis of the metaplectic semigroup \texorpdfstring{$\Mp_+(d,\bC)$}{Lg}}\label{subsec:anp}
	In this section, we establish the generators of $\Mp_+(d,\bC)$, we study their polar decomposition and the validity of the intertwining relation with the Schr\"odinger representation in this context.
	\subsection{Generators} 
	Now that the metaplectic semigroup has been extensively introduced, we turn to a description of $\Sp_+(d,\bC)$ and $\Mp_+(d,\bC)$ that does not resort to atomic metaplectic contractions. At the level of positive symplectic matrices, we shall see that $\Sp_+(d,\bC)$ is simply obtained by adding matrices $V_Q$ with $Q\in\Sigma_{\geq0}(d)$ to the generators of $\Sp(d,\bR)$.
	\begin{proposition}\label{prop51}
		The semigroup $\Sp_+(d,\bC)$ is generated by $J$ and by the matrices in the form $\cD_E$, $E\in\GL(d,\bR)$, and $V_Q$, $Q\in\Sigma_{\geq0}(d)$.
	\end{proposition}
	\begin{proof}
		Let
		\begin{align}
			&\cE(d)=\{J\}\cup \{\cD_E:E\in\GL(d,\bR)\}\cup\{V_Q:Q\in\Sym(d,\bR)\}\cup \{S:\text{$S$ is purely positive}\}, \\
			&\cE'(d)=\{J\}\cup\{\cD_E:E\in\GL(d,\bR)\}\cup\{V_Q:Q\in\Sigma_{\geq0}(d)\}.
		\end{align}
		Observe that $\cE(d)$ generates $\Sp_+(d,\bC)$ according to Theorem \ref{theo.class}. Let us denote by $\Sp_+'(d,\bC)$ the semigroup generated by $\cE'(d)$.
		
		First, we show that $\cE'(d)\subseteq \Sp_+(d,\bC)$ using H\"ormander's positivity condition. It is clear that we only need to verify the inclusion of the matrices $V_Q$, with $Q\in\Sigma_{\geq0}(d)$, in $\Sp_+(d,\bC)$. For any such a matrix, let us write $Q=Q_1+iQ_2$, $Q_2\geq0$. In this case, the matrix $\cM_{V_Q}$ in \eqref{defMS}, is
		\begin{equation}
			\cM_{V_Q}=\left(\begin{array}{cc|cc}
				Q_2 & O & O & O\\
				O & O & O & O\\
				\hline
				O & O & Q_2 & O\\
				O & O & O & O
			\end{array}\right).
		\end{equation}
		Clearly, $\cM_{V_Q}\geq0$ if and only if $Q_2\geq0$, which proves the first inclusion. 
		
		Conversely, to prove that $\cE(d)\subseteq\Sp_+'(d,\bC)$, we need to show that every purely positive symplectic matrix can be decomposed into a product of factors in $\cE'(d)$. Observe that for every $\alpha>0$, the matrix $V_{i\alpha}$ defined in \eqref{Va} is trivially in $\Sp_+'(1,\bC)$. Moreover, for $\vartheta\geq0$, the factorization
		\begin{equation}\label{factorRtau}\begin{split}
			\cR_\vartheta&=\begin{pmatrix}
				1 & 0\\
				i\tanh\vartheta & 1
			\end{pmatrix}
			\begin{pmatrix}
				\cosh\vartheta & 0\\
				0 & 1/\cosh\vartheta
			\end{pmatrix}
			J^{-1}
			\begin{pmatrix}
				1 & 0\\
				i\tanh\vartheta & 1
			\end{pmatrix}J\\
			&=V_{i\tanh\vartheta}\cD_{1/\cosh\vartheta}J^{-1}V_{i\tanh\vartheta}J,
		\end{split}
		\end{equation}
		implies that the $\cR_\vartheta$'s are all in $\Sp_+'(1,\bC)$. Let $\Xi=\Xi_1\otimes\ldots\otimes \Xi_d$ be a purely positive matrix, with $\Xi=V_{i\Delta}\mathcal{R}_\Theta$ as in Remark \ref{remAMC}. 
        
		The matrix $V_{i\Delta}\in\mathcal{E}'(d)$, so we just need to show that $\cR_\Theta\in\Sp_+'(d,\bC)$. By \eqref{factorRtau},
		\begin{equation}
			\cR_\Theta=\begin{pmatrix}
				I & O\\
				iQ & I
			\end{pmatrix}
			\begin{pmatrix}
				\diag(a_1,\ldots,a_d) & O\\
				O & \diag(a_1,\ldots,a_d)^{-1}
			\end{pmatrix}
			J^{-1}
			\begin{pmatrix}
				I & O\\
				iQ & I
			\end{pmatrix}J,
		\end{equation}
		with $Q=\diag(\mu_1\ldots,\mu_d)$, $\mu_j=\tanh\vartheta_j$ if $\Xi_j=\cR_{\vartheta_j}$, $\mu_j=1$ otherwise. This concludes the proof. 

	\end{proof}
	
	From the perspective of $\Mp_+(d,\bC)$, recall the operator
	\begin{equation}
		\frp_Qf(x)=e^{i\pi Qx\cdot x}f(x), \qquad f\in L^2(\rd).
	\end{equation}
	If $\Im(Q)\geq0$, then $\frp_Q$ is a contraction of $L^2(\rd)$, and we may restate Proposition \ref{propMPgen} as follows.
	
	\begin{proposition}\label{prop52}
			The metaplectic semigroup $\Mp_+(d,\bC)$ is generated by
			\begin{equation}
				\cF f(\xi)=i^{-d/2}\int_{\rd}f(x)e^{-2\pi i\xi \cdot x}dx, \qquad f\in \cS(\rd), \quad \xi\in\rd,
			\end{equation}
			and by the family of $L^2$-contractions
			\begin{align}
				& \frp_Q f(x)=e^{iQx\cdot x}f(x), \qquad Q\in\Sigma_{\geq0}(d),\; f\in L^2(\rd),\\
				& \frT_E f(x)= i^m |\det(E)|^{1/2}f(Ex), \qquad E\in\GL(d,\bR),\; f\in L^2(\rd),
			\end{align}
			where $m\in\bZ$ is the argument of $\det(E)$. Moreover, $\pi_+^{Mp}(\cF)=J$, $\pi_+^{Mp}(\frp_Q)=V_Q$ and $\pi_+^{Mp}(\frT_E)=\cD_E$, the matrices being defined as in \eqref{defJ}, \eqref{defVQ} and \eqref{defDE}, respectively.
		\end{proposition} 
		
		\begin{remark}\label{remConvolution}
			Remarkably, the $2d\times2d$ complex symplectic matrices in the form
			\begin{equation}
				V_P^\top=\begin{pmatrix}
					I & P\\
					O & I
				\end{pmatrix}, \qquad P\in\Sym(d,\bC), \quad \Im(P)\leq0
			\end{equation}
			are in $\Sp_+(d,\bC)$, since $V_P^\top = J^{-1}V_{-P}J$. In analogy with the real case, the corresponding complex metaplectic operators are the Fourier multipliers
			\begin{equation}\label{defmP}
				\mathfrak{m}_{-P} f=\cF^{-1}(\Phi_{-P} \widehat f), \qquad f\in L^2(\rd),
			\end{equation} 
			where $\Phi_{-P}(x)=e^{-i\pi Px\cdot x}$. We will return to these operators in Section \ref{sec:Subsemigroups}.
		\end{remark}

        \subsection{The polar decomposition}
        For every linear and bounded operator $T$ on $L^2(\rd)$ there exist unique operators $T_1$ and $T_2$, such that $T=T_1T_2$, with the following properties \cite[Theorem VI.10]{ReedAndSimon}:
        \begin{enumerate}[(1)]
            \item $T_1$ is a partial isometry, i.e., $T_1|_{\ker( T_1)^\perp}$ is an isometry.
            \item $T_2$ is positive and self-adjoint.
            \item $\ker(T_1)=\ker(T)$.
        \end{enumerate}
        In our case, $T=\widehat S$ is injective, so that conditions (1)--(3) simplify as
        \begin{enumerate}[(1)]
            \item $T_1$ is an isometry.
            \item $T_2$ is positive and self-adjoint.
        \end{enumerate}
        Now, we prove that in the case of $\widehat S\in\Mp_+(d,\bC)$, the operators $T_1,T_2$ are themselves metaplectic. 

        \begin{theorem}\label{polar}
            Let $\widehat S\in \Mp_+(d,\bC)$. There exist unique $\widehat U\in\Mp(d,\bR)$ and $\widehat Z\in\Mp_+(d,\bC)$ such that \begin{equation}\label{PolarDecomposition}
                \widehat S=\widehat U\widehat Z
            \end{equation}
            and $\widehat Z$ is positive and self-adjoint.
        \end{theorem}
        \begin{proof}
            There exist (non-unique) $\widehat U_1,\widehat U_2\in\Mp(d,\bR)$ and $\widehat \Xi$ atomic metaplectic contraction such that $\widehat S=\widehat U_1\widehat \Xi \widehat U_2$. Let $\widehat U=\widehat U_1\widehat U_2$ and let $\widehat Z=\widehat U_2^{-1}\widehat \Xi\widehat U_2$. By Theorem \ref{thmOscillator} $(vi)$, $\widehat \Xi$ is self-adjoint. We may factorize $\widehat \Xi=\mathfrak{R}_\Theta \mathfrak{p}_{i\Delta}$ as in Remark \ref{remAMC}.
            Correspondingly, $\Xi=\mathcal{R}_\Theta V_{i\Delta}$ and the two factors also commute. 
            Since $\mathcal{R}_{\vartheta_j}^{1/2}=\mathcal{R}_{\vartheta_j/2}$ and $V_{i\alpha_j}^{1/2}=V_{i\alpha_j/2}$, by tensorization we obtain that $\widehat \Xi^{1/2}=\mathfrak{R}_{\Theta/2}\mathfrak{p}_{i\Delta/2}$, where we use here that $\mathfrak{R}_{\Theta/2}$ and $\mathfrak{p}_{i\Delta/2}$ commute as well by construction. 
            Since they are also self-adjoint, we infer that for every $f\in L^2(\rd)$,
            \begin{align}
                \la \widehat \Xi f,f \ra=\la \widehat \Xi^{1/2} f,\widehat \Xi^{1/2}f \ra=\norm{\widehat \Xi^{1/2}f}_2^2,
            \end{align}
            so that $\widehat \Xi$ is positive, other than self-adjoint. Since $\widehat U_2$ is unitary, it follows that $\widehat Z=\widehat U_2^{-1}\widehat \Xi\widehat U_2$ is self-adjoint and positive. 
            By its definition, $\widehat U=\widehat U_1\widehat U_2\in \Mp(d,\bR)$ is unitary, and in particular a partial isometry. 
            We thereby conclude that $\widehat S=\widehat U\widehat Z$ is a polar decomposition of $\widehat S$ with $\ker(\widehat S)=\ker(\widehat U)$. Since this polar decomposition is unique, we have done.
            
        \end{proof}
        
        Generally speaking, the polar decomposition of $\widehat S$ is rather implicit and not always easy to compute. 
        Specifically, 
        \begin{equation}
            \widehat Z=(\widehat S^\ast \widehat S)^{1/2}, \qquad \widehat U=\widehat S( \widehat S^\ast \widehat S)^{-1/2}.
        \end{equation}
        In contrast, atomic metaplectic contractions are explicit. For this reason, resorting to the non-unique factorization in \eqref{HFactoriz} remains useful when proving general statements about metaplectic operators. However, this is no longer sufficient for their phase-space analysis. Indeed, if \eqref{PolarDecomposition} is the polar decomposition of $\widehat S$, then $\widehat U$ acts as a sort of {\em real part} of $\widehat S$, and its action is reflected in a symplectic linear transformation of phase space, as in the classical case. This property fails for $\widehat Z$, whose action on phase space must be interpreted differently, see Section \ref{subsec:WignerAnalysis} below.

        For the rest of this work, we denote
        \begin{equation}
            {\Mp}_0(d,\bC)=\{\text{$\widehat Z\in{\Mp}_+(d,\bC)$ : $\widehat Z$ is positive and self-adjoint}\}.
        \end{equation}

        Thanks to Theorem \ref{polar} we obtain a characterization of $\Mp_0(d,\bC)$ and, with it, a full description of $\Mp_+(d,\bC)$.

        \begin{corollary}\label{CorPolar}
            Under the notation of Theorem \ref{polar}, let us consider the mapping $\iota:\widehat S\in \Mp_+(d,\bC)\mapsto (\widehat U,\widehat Z)\in \Mp(d,\bR)\times\Mp_0(d,\bC)$.
            \begin{enumerate}[(i)]
                \item $\iota$ is a bijection.
                \item $\widehat Z\in \Mp_0(d,\bC)$ if and only if $\widehat Z=\widehat V^{-1}\widehat \Xi\widehat V$ for some (non-unique) atomic metaplectic contraction $\widehat \Xi$ and $\widehat V\in\Mp(d,\bR)$.
                \item It holds $\widehat S=\widehat Z'\widehat U$ for a suitable $\widehat Z'\in\Mp_0(d,\bC)$.
            \end{enumerate}
        \end{corollary}
        \begin{proof}
            Observe that $\widehat S\in\Mp_0(d,\bC)\cap\Mp(d,\bR)$ if and only if $\widehat S=c\cdot\id_{L^2}$ for some phase factor $c\in\bC$, $|c|=1$.
            Item $(i)$ is just a trivial restatement of the uniqueness of the polar decomposition, using Theorem \ref{polar}. 
            To prove item $(ii)$, just observe that if $\widehat Z\in\Mp_0(d,\bC)$, then $\widehat Z=\widehat U_1\widehat\Xi\widehat U_2$ as in \eqref{HFactoriz}. By basically repeating the argument of Theorem \ref{polar}, we obtain the polar decomposition of $\widehat Z$ 
            \begin{equation}
                \widehat Z=(\widehat U_1\widehat U_2)(\widehat U_2^{-1}\widehat \Xi \widehat U_2).
            \end{equation}
            On the other hand, $\widehat Z$ is the polar decomposition of itself. By uniqueness, we infer $\widehat U_1\widehat U_2=\id_{L^2}$. Therefore,
            \begin{equation}
                {\Mp}_0(d,\bR)\subseteq\{\text{$\widehat V^{-1}\widehat \Xi\widehat V$ : $\widehat V\in\Mp(d,\bR)$ and $\widehat \Xi$ atomic metaplectic contraction}\}.
            \end{equation}
            The other side of the equivalence was already noted in the proof of Theorem \ref{polar}. Item $(iii)$ follows by associativity: set $\widehat Z'=\widehat U_1\widehat \Xi\widehat U_1^{-1}$, then
            \begin{align}
                \widehat S=\widehat U\widehat Z&=\widehat U_1\widehat \Xi\widehat U_2=(\widehat U_1\widehat \Xi\widehat U_1^{-1})(\widehat U_1\widehat U_2)=\widehat Z'\widehat U,
            \end{align}
            and we are done.
        \end{proof}
		
		\subsection{Intertwining with the Schr\"odinger representation}\label{subsec:43}
		The purpose of this section is to prove the intertwining property, characterizing metaplectic operators in $\Mp(d,\bR)$, between the metaplectic semigroup $\Mp_+(d,\bC)$ and Gaussian states. In other words, the goal is to prove Theorem \ref{thmGG1} below. 

		
		\subsubsection{Complex translations and modulations}
		Modulations are complexified by setting
		\begin{equation}
			M_wf(y)=e^{2\pi i w\cdot y}f(y), \qquad y\in\rd, \quad w \in \mathbb{C}^d.
		\end{equation} 
		On the other hand, translations are complexified as pseudodifferential operators as follows. If $z \in \mathbb{C}^d$ and $\widehat f$ decays as a Gaussian, the following integral converges (and it coincides with translations when $z\in\rd$):
		\begin{equation}
			T_zf(y)=\int_{\rd}\widehat f(\eta)e^{2\pi i\eta\cdot (y-z)}d\eta, \qquad y\in\rd.
		\end{equation}
		Let $z,w\in\bC^d$ and $\tau\in\bR$. We formally define
			\begin{equation}
				\rho(z,w;\tau)f(y)=e^{2\pi i\tau}e^{-i\pi z\cdot w}e^{2\pi iw\cdot y}\int_{\rd}\widehat f(\eta)e^{2\pi i\eta\cdot (y-z)}d\eta, \qquad y\in\rd.
			\end{equation}

		\begin{remark}
			Using the Gaussian integral formula, we obtain
			\begin{equation}
				\int_{\rd}\widehat{\f_M}(\eta)e^{2\pi i\eta\cdot(y-z)}d\eta=e^{-\pi M(y-z)\cdot(y-z)},
			\end{equation}
			and, therefore,
			\begin{equation}
				\rho(z,w;\tau)\f_M(y)=e^{2\pi i\tau}e^{-i\pi z\cdot w}e^{2\pi iw\cdot y}e^{-\pi M(y-z)(y-z)}, \qquad y\in\rd.
			\end{equation}
			In particular, $\rho(z,w;\tau)\f_M\in L^2(\rd)$ for every $z,w\in\bC^d$.
		\end{remark}
		
		\subsubsection{The interwining relation} Under this notation, our aim is to prove the following theorem. 
        \begin{theorem}\label{thmGG1}
			Let $x,\xi\in\rd$, $\tau\in\bR$. Let $M>0$ and $\f_M(t)=e^{-\pi Mt\cdot t}$ be the corresponding Gaussian state. For every $\widehat S\in \Mp_+(d,\bC)$,
			\begin{equation}\label{intertOsc}
				\widehat S\rho(x,\xi;\tau)\f_M=\rho(S(x,\xi);\tau)\widehat S\f_M,
			\end{equation}
			with $S=\pi^{Mp}_+(\widehat S)$.
\end{theorem}

        For the sake of simplicity, we divide the proof of Theorem \ref{thmGG1} into two parts. First, we check that metaplectic operators in $\Mp(d,\bR)$ satisfy the intertwining relation \eqref{intertOsc} even when $x,\xi\in\bC^d$. Precisely:
		
		\begin{lemma}\label{lemma1}
			Let $z,w\in\bC^d$, $\tau\in\bR$, and $\widehat S\in\Mp(d,\bR)$. Then, 
			\begin{equation}
				\widehat S\rho(z,w;\tau)\f_M=\rho(S(z,w);\tau)\widehat S\f_M.
			\end{equation}
		\end{lemma}
		\begin{proof}
			Clearly, it is enough to prove the assertion for the generators of $\Mp(d,\bR)$. First, we consider the Fourier transform. By Proposition \ref{GaussianIntegrals},
			\begin{align}
				\cF\rho(z,w;\tau)\f_M(x)&=i^{-d/2}\int_{\rd}\rho(z,w;\tau)\f_M(y)e^{-2\pi iy\cdot x}dy\\
				&=i^{-d/2}e^{2\pi i\tau}e^{-i\pi z\cdot w}\int_{\rd}e^{2\pi i w \cdot y}e^{-\pi M(y-z)\cdot(y-z)}e^{-2\pi ix\cdot y}dy\\
				& =i^{-d/2}e^{2\pi i\tau}e^{-i\pi z\cdot w}e^{-\pi Mz\cdot z}\int_{\rd} e^{-\pi My\cdot y}e^{-2\pi iy\cdot (x-w+iMz)}dy \\
				& = i^{-d/2}\det(M)^{-1/2}e^{2\pi i\tau}e^{-i\pi z\cdot w}e^{-\pi Mz\cdot z}e^{-\pi M^{-1}(x-w+iMz)\cdot (x-w+iMz)}\\
				&= i^{-d/2}\det(M)^{-1/2}e^{2\pi i\tau}e^{-i\pi z\cdot w}e^{-\pi M^{-1}(x-w)\cdot(x-w)}e^{-2\pi i(x-w)\cdot z}\\
				&=i^{-d/2}e^{2\pi i\tau}e^{i\pi z\cdot w}e^{-2\pi ix\cdot z}\int_{\R^d}e^{-\pi Mu\cdot u}e^{-2\pi i (x-w)\cdot u}du\\
				&= \rho(w,-z;\tau)\cF \f_M(x)=\rho(J(z,w);\tau)\cF \f_M(x).
			\end{align}
			This proves the assertion for the Fourier transform. Let $E\in \GL(d,\bR)$ and $\frT_E$ be the rescaling defined as in \eqref{defTE}, 
			\begin{align}
				\frT_E\rho(z,w;\tau)\f_M(x)&=i^m|\det(E)|^{1/2}[\rho(z,w;\tau)\f_M](Ex)\\
				&=i^m|\det(E)|^{1/2}e^{2\pi i\tau}e^{-i\pi z\cdot w}e^{2\pi iEx\cdot w}e^{-\pi M(Ex-z)\cdot(Ex-z)}\\
				&=i^m|\det(E)|^{1/2}e^{2\pi i\tau}e^{-i\pi (EE^{-1})z\cdot w}e^{2\pi ix\cdot E^\top w}e^{-\pi M[E(x-E^{-1}z)]\cdot[E(x-E^{-1}z)]}\\
				&=i^m|\det(E)|^{1/2}e^{2\pi i\tau}e^{-i\pi E^{-1}z\cdot E^\top w}e^{2\pi ix\cdot E^\top w}e^{-\pi M[E(x-E^{-1}z)]\cdot[E(x-E^{-1}z)]}\\
				&=\rho(E^{-1}z,E^\top w;\tau)\frT_E\f_M(x)=\rho(\cD_E(z,w);\tau)\frT_E\f_M(x),
			\end{align}
			where $\cD_E$ is defined as in \eqref{defDE}. Finally, we need to prove the assertion for $\frp_Q$, defined for $Q\in\Sym(d,\bR)$ as in \eqref{defpQ}. Let $V_Q$ be the lower triangular matrix defined as in \eqref{defVQ}. Then,
			\begin{equation}\label{computVQ}\begin{split}
				\rho(V_Q(z,w);\tau)\frp_Q\f_M(x)&=\rho(z,w+Qz;\tau)\frp_Q\f_M(x)\\
				&=e^{2\pi i\tau}e^{-i\pi z\cdot (w+Qz)}e^{2\pi i(w+Qz)\cdot x}e^{i\pi Q(x-z)\cdot (x-z)}e^{-\pi M(x-z)\cdot(x-z)}\\
				&=e^{i\pi Qx\cdot x}e^{2\pi i\tau}e^{-i\pi z\cdot w}e^{2\pi iw\cdot x}e^{-\pi M(x-z)\cdot(x-z)}\\
				&=\frp_Q\rho(z,w;\tau)\f_M(x).
			\end{split}\end{equation}
			This concludes the proof.
			
		\end{proof}
		
		Next, we check the validity of \eqref{intertOsc} for atomic metaplectic contractions and $x,\xi\in\rd$.
		
		\begin{lemma}\label{lemma2}
			Let $\widehat \Xi$ be an atomic metaplectic contraction. For every $x,\xi\in\rd$, $\tau\in\bR$
			\begin{equation}
				\widehat \Xi\rho(x,\xi;\tau)\f_M=\rho(\Xi(x,\xi);\tau)\widehat \Xi\f_M.
			\end{equation}
		\end{lemma}
		\begin{proof}
			{\bf Claim 1.} {The assertion holds in dimension $d=1$ for $\widehat \Xi=\frp_i$ and $\widehat \Xi=\frR_\vartheta$, $\vartheta>0$}. 
			
			The very same computations in \eqref{computVQ} show that
			\begin{equation}
				\frp_i\rho(x,\xi;\tau)\f_a=\rho(x,\xi+ix;\tau)\frp_i\f_a=\rho(V_i(x,\xi);\tau)\frp_i\f_a
			\end{equation}
			when $\f_a(y)=e^{-\pi ay^2}$, $(z,w)=(x,\xi)\in\bR^2$ and $Q=i$. Let $\vartheta>0$, we show that
			\begin{align}
				\frR_\vartheta \rho(x,\xi;\tau)\f_a =  \rho(\mathcal{R}_\vartheta(x,\xi);\tau)\frR_\vartheta\f_a=\rho(\cosh\vartheta-i\sinh\vartheta,i\sinh\vartheta+\cosh\vartheta;\tau)\frR_\vartheta\f_a.
			\end{align}
			By the Gaussian integral formula of Proposition \ref{GaussianIntegrals}, it follows easily that
			\begin{equation}
				\frR_\vartheta\f_a=\left(\frac{1}{\cosh\vartheta +a\sinh\vartheta}\right)^{1/2}\f_b, \qquad b=\frac{a+\tanh\vartheta}{1+a\tanh\vartheta}.
			\end{equation}
			Therefore $\frR_\vartheta\f_a$ is a Gaussian and $T_z\frR_\vartheta \f_a(y)=\frR_\vartheta \f_a(y-z)$ for every $z\in\bC$. Then, by using the identity $\cosh\vartheta^2-\sinh\vartheta^2=1$ and its surrogates, we obtain
			\begin{align}
				\frR_\vartheta &\rho(x,\xi;\tau)\f_a(y) = (\cosh\vartheta)^{-1/2}\int_{-\infty}^\infty (\rho(x,\xi;\tau)\f_a){\text{\textasciicircum}}(\eta)e^{i\pi (2y\eta +i(y^2+\eta^2)\sinh\vartheta)/\cosh\vartheta}d\eta\\
				&= (\cosh\vartheta)^{-1/2}\int_{-\infty}^\infty \rho(\xi,-x;\tau)\widehat{\f_a}(\eta)e^{i\pi (2y\eta +i(y^2+\eta^2)\sinh\vartheta)/\cosh\vartheta}d\eta\\
				&=(\cosh\vartheta)^{-1/2}e^{2\pi i\tau}e^{i\pi x\cdot \xi}\int_{-\infty}^\infty e^{-2\pi i\eta \cdot x}\widehat{\f_a}(\eta-\xi)e^{i\pi (2y\eta +i(y^2+\eta^2)\sinh\vartheta)/\cosh\vartheta}d\eta\\
				&=(\cosh\vartheta)^{-1/2}e^{2\pi i\tau}e^{i\pi x\cdot \xi}\int_{-\infty}^\infty e^{-2\pi i(\omega+\xi) \cdot x}\widehat{\f_a}(\omega)e^{i\pi (2y(\omega+\xi)+i(y^2+(\omega+\xi)^2)\sinh\vartheta)/\cosh\vartheta}d\omega\\
				&=(\cosh\vartheta)^{-1/2}e^{2\pi i\tau}e^{-i\pi x\cdot \xi}e^{-\pi\tanh\vartheta y^2}e^{2\pi iy\xi/\cosh\vartheta}\\
				&\qquad\qquad\qquad\qquad\qquad\qquad\qquad\times\int_{-\infty}^\infty e^{-2\pi i\omega \cdot x}\widehat{\f_a}(\omega)e^{i\pi (2y\omega +i(\omega+\xi)^2\sinh\vartheta)/\cosh\vartheta}d\omega\\
				&=(\cosh\vartheta)^{-1/2}e^{2\pi i\tau}e^{-i\pi x\cdot \xi}e^{-\pi\tanh\vartheta y^2}e^{2\pi iy\xi/\cosh\vartheta}e^{-\pi\tanh\vartheta\xi^2}\times\\
				&\qquad\qquad\qquad\qquad\qquad\qquad\qquad\times\int_{-\infty}^\infty \widehat{\f_a}(\omega)e^{2\pi i\omega  (y-x\cosh\vartheta+i\xi\sinh\vartheta)/\cosh\vartheta}e^{-\pi\tanh\vartheta\omega^2}d\omega\\
				&=(\cosh\vartheta)^{-1/2}e^{2\pi i\tau}e^{-i\pi x\cdot \xi}e^{-\pi\tanh\vartheta y^2}e^{-\pi\tanh\vartheta\xi^2}e^{2\pi iy\xi/\cosh\vartheta}e^{\pi\tanh\vartheta(y-x\cosh\vartheta+i\xi\sinh\vartheta)^2}\times\\
				&\times\underbrace{\int_{-\infty}^\infty \widehat{\f_a}(\omega)e^{2\pi i\omega  (y-x\cosh\vartheta+i\xi\sinh\vartheta)/\cosh\vartheta}e^{-\pi\tanh\vartheta(\omega^2+(y-x\cosh\vartheta+i\xi\sinh\vartheta)^2)}d\omega}_{=\frR_{\vartheta}\f_a(y-x\cosh\vartheta+i\xi\sinh\vartheta)=T_{x\cosh\vartheta-i\xi\sinh\vartheta}\frR_\vartheta \f_a}\\
				&=e^{2\pi i\tau}e^{-i\pi(x\cosh\vartheta-i\xi\sinh\vartheta)(ix\sinh\vartheta+\xi\cosh\vartheta)}e^{2\pi iy(ix\sinh\vartheta+\xi\cosh\vartheta)}T_{x\cosh\vartheta-i\xi\sinh\vartheta}\frR_\vartheta \f_a(y)\\
				&=\rho(x\cosh\vartheta-i\xi\sinh\vartheta,ix\sinh\vartheta+\xi\cosh\vartheta;\tau)\frR_\vartheta \f_a(y).
			\end{align}
			This proves Claim 1.
				
			{\bf Claim 2.} The assertion holds in dimension $d\geq1$ for Gaussian states $\f_\Delta(y)=e^{-\pi\Delta y\cdot y}$, $\Delta>0$ diagonal.
			
			Observe that for $z\in\bC^{2d}$, $S_1,S_2\in\Sp_+(1,\bC)$,
			\begin{equation}
				\rho((S_1\otimes S_2)(x,y,\xi,\eta);\tau)(f\otimes g)=e^{-2\pi i\tau}\rho(S_1(x,\xi);\tau)f\otimes \rho(S_2(y,\eta);\tau)g
			\end{equation}
			for Gaussian functions $f,g$. Then, a standard inductive argument shows that
			\begin{equation}
				\rho((S_1\otimes\ldots\otimes S_d)(x,\xi);\tau)(f_1\otimes\ldots\otimes f_d)=e^{-2\pi i\tau(d-1)}\rho(S_1(x_1,\xi_1);\tau)f_1\otimes\ldots\otimes \rho(S_d(x_d,\xi_d);\tau)f_d
			\end{equation}
			for Gaussians $f_1,\ldots,f_d$, $x,\xi\in\rd$ and $\tau\in\bR$, and $S_1,\ldots,S_d\in\Sp_+(1,\bC)$. In particular, if $\Xi$ is purely positive, $ \Xi= \Xi_1\otimes\ldots\otimes \Xi_d$, with either $ \Xi_j=V_i$ or $ \Xi_j=\mathcal{R}_{\vartheta_j}$ ($\vartheta_j\geq0$), $j=1,\ldots,d$, we have
			\begin{equation}
				\rho(\Xi(x,\xi);\tau)(f_1\otimes\ldots\otimes f_d)=e^{-2\pi i\tau(d-1)}\rho(\Xi_1(x_1,\xi_1);\tau)f_1\otimes\ldots\otimes \rho(\Xi_d(x_d,\xi_d);\tau)f_d.
			\end{equation}
			Let $\widehat \Xi$ be an atomic metaplectic contraction with $\widehat \Xi=\bigotimes_{j=1}^d\widehat \Xi_j$, where either $\widehat \Xi_j=\frp_i$ or $\widehat \Xi_j=\frR_{\vartheta_j}$ for some $\vartheta_j\geq0$ ($j=1,\ldots,d$). Then, by tensorizing $\f_\Delta=\bigotimes_{j=1}^d\f_{a_j}$, $a_1,\ldots,a_d>0$,
			\begin{align}
				\widehat \Xi\rho(x,\xi;\tau)\f_\Delta(y)&=e^{-2\pi i\tau(d-1)}\Big(\bigotimes_{j=1}^d\widehat \Xi_j\bigotimes_{j=1}^d\rho(x_j,\xi_j;\tau)\Big)\f_\Delta(y)\\
				&=e^{-2\pi i\tau(d-1)}\bigotimes_{j=1}^d\widehat \Xi_j\rho(x_j,\xi_j;\tau)\f_{a_j}(y_j)\\
				&=e^{-2\pi i\tau(d-1)}\bigotimes_{j=1}^d\rho(\Xi_j(x_j,\xi_j);\tau)\widehat \Xi_j\f_{a_j}(y_j)\\
				&=e^{-2\pi i\tau(d-1)}\rho\Big(\bigotimes_{j=1}^d\Xi_j(x_j,\xi_j);\tau\Big)\bigotimes_{j=1}^d\widehat \Xi_j\f_{a_j}(y_j)\\
				&=\rho(\Xi(x,\xi);\tau)\widehat \Xi\f_\Delta(y).
			\end{align}
			This concludes the proof of Claim 2.
			
			{\bf Claim 3.} Finally, we prove the assertion for $M$ symmetric and positive-definite. 
			
			Let $M=\Sigma^\top\Delta\Sigma$, with $\Sigma$ orthogonal and $\Delta>0$ diagonal. Then, for $z\in\rdd$ and $\tau\in\bR$, by writing $\f_M=\mathfrak{T}_\Sigma\f_\Delta$ and by applying Lemma \ref{lemma1} to $\mathfrak{T}_{\Sigma}$,
			\begin{align}
				\widehat S\rho(z;\tau)\f_M&=\widehat S\rho(z;\tau)\mathfrak{T}_\Sigma \f_\Delta=\widehat S\mathfrak{T}_\Sigma\rho(\cD_{\Sigma}^{-1}z;\tau) \f_\Delta.
			\end{align}
			Then, Claim 2 applied to $\widehat S\mathfrak{T}_\Sigma$ and $\f_\Delta$ yields
			\begin{align}
				\widehat S\rho(z;\tau)\f_M&=\rho(S\cD_\Sigma\cD_{\Sigma}^{-1}z;\tau)\widehat S\mathfrak{T}_\Sigma \f_\Delta=\rho(Sz;\tau)\widehat S\f_M.
			\end{align}
			This concludes the proof.
			
		\end{proof}
        
		\begin{proof}[Proof of Theorem \ref{thmGG1}] The proof of the Theorem follows by combining Lemmas \ref{lemma1} and \ref{lemma2}. To be precise, let $\widehat S\in \Mp_+(d,\bC)$ and observe that both sides of \eqref{intertOsc} are well-defined in $L^2(\rd)$. By item $(ii)$ of Theorem \ref{thmOscillator}, we can write $\widehat S=\widehat U_1\widehat \Xi\widehat U_2$, where $\widehat U_1,\widehat U_2\in\Mp(d,\bR)$ and $\widehat \Xi$ is an atomic metaplectic contraction. Then, by applying \eqref{defMetap}, Lemma \ref{lemma2} and Lemma \ref{lemma1} in the order,
			\begin{align}
				\widehat S\rho(x,\xi;\tau)\f_M&=\widehat U_1\widehat \Xi\widehat U_2\rho(x,\xi;\tau)\f_M=\widehat U_1\widehat \Xi\rho(U_2(x,\xi);\tau)\widehat U_2\f_M=\widehat U_2\rho(\Xi U_2(x,\xi);\tau)\widehat \Xi\widehat U_2\f_M\\ 
				&=\rho(U_1\Xi U_2(x,\xi);\tau)\widehat U_1\widehat \Xi\widehat U_2\f_M=\rho(S(x,\xi);\tau)\widehat S\f_M,
			\end{align}
		and we are done.
		\end{proof}
        
		We conclude this section with the following corollary, extending the intertwining relation \eqref{intertOsc} to a dense subset of $L^2(\rd)$.
			
		\begin{corollary}\label{cor510}
			Let $S\in\Sp_+(d,\bC)$. 
            The intertwining relation
			\begin{equation}\label{intertOsc2}
				\widehat S\rho(z;\tau)f=\rho(Sz;\tau)\widehat Sf, \qquad z\in\rdd
			\end{equation}
			holds on a dense subset of $L^2(\rd)$, namely on functions in the form
			\begin{equation}\label{defTFshiftsGauss}
				f(t)=\sum_{j=1}^Nc_j\rho(z_j;\tau_j)\f_{M_j}(t)
			\end{equation}
			where $N\geq1$ and for every $j=1,\ldots,N$, $c_j\in\bC$, $z_j\in\rdd$, $\tau_j\in\bR$ and $M_j>0$.  
		\end{corollary}	
		\begin{proof}
			Observe that for $z=(x,\xi),z_1=(x_1,\xi_1)\in\rdd$ and $\tau,\tau_1\in\bR$, by Theorem \ref{thmGG1},
			\begin{align}
				\widehat S\rho(z;\tau)\rho(z_1;\tau_1)\f_M&=\widehat S\rho\left(z+z_1;\tau+\tau_1-\frac 1 2 \sigma(z,z_1)\right)\f_M\\
				&=\rho\left(S(z+z_1);\tau+\tau_1-\frac 1 2 \sigma(Sz,Sz_1)\right)\widehat S\f_M\\
				&=\rho(Sz;\tau)\rho(Sz_1;\tau_1)\widehat S\f_M=\rho(Sz;\tau)\widehat S\rho(z_1;\tau_1)\f_M.
			\end{align}
			Consequently, \eqref{intertOsc} holds for functions in the form \eqref{defTFshiftsGauss} with $N=1$, and the assertion follows by linearity.
            
		\end{proof}
		
		\section{Subsemigroups and Wigner analysis}\label{sec:Subsemigroups}
		In this section, we characterize three subsemigroups of $\Mp_+(d,\bC)$ in terms of the structures of their projections onto $\Sp_+(d,\bC)$. We also investigate their intertwining relation with the Wigner distribution and with complex conjugation.
			\subsection{Two subsemigroups of \texorpdfstring{$\Mp_+(d,\bC)$}{Lg}}
			Consider the Cauchy problem associated to the complex heat equation
			\begin{equation}\label{HeatEq}
				\begin{cases}
					i\partial_t u=(\alpha+i\beta)\Delta_x u, & \text{$t>0$, $x\in\rd$}\\
					u(0,x)=u_0(x),
				\end{cases}
			\end{equation}
			where $\alpha,\beta\in\bR$, $|\alpha|+|\beta|\neq0$, $\beta\geq0$ and $u_0\in\cS(\rd)$. The propagator of \eqref{HeatEq} consists of a one-parameter semigroup of Fourier multipliers
			\begin{equation}
				u(t,x)=\cF^{-1}(e^{4\pi^2 i(\alpha+i\beta)t|\cdot|^2}\widehat{u_0})(x), \qquad t\geq0.
			\end{equation}
			With the notation of the present paper,
			\begin{equation}\label{propEqCalore}
				u(t,x)=\mathfrak{m}_{P_t}u_0(x),
			\end{equation}
			where for $t\geq0$, 
			\begin{equation}\label{subsemigroup1}
				\pi^{Mp}_+(\mathfrak{m}_{P_t})=V_{-P_t}^\top=\begin{pmatrix}
					I_d & -4\pi (\alpha+i\beta)t\cdot I_d\\
					O_d & I_d
				\end{pmatrix}, \qquad P_t=4\pi (\alpha+i\beta)tI_d,
			\end{equation}
			The reader may observe that the conditions $t,\beta\geq0$ not only guarantee the well-posedness of \eqref{HeatEq} for $t\geq0$, but also the validity of the condition of Definition \ref{defPosSymp} for $V_{-P_t}^\top$. Stated differently, $V_{-P_t}^\top\in\Sp_+(d,\bC)$ and $\mathfrak{m}_{P_t}\in\Mp_+(d,\bC)$ for every $t\geq0$. 
			
			One-parameter subsemigroups of $\Sp_+(d,\bC)$ were characterized in \cite{Brunet} in terms of their infinitesimal generators. In this section, we focus on the subsemigroups of matrices $S\in \Sp_+(d,\bC)$ with blocks \eqref{blockS} satisfying one of the following conditions: $B=O_d$ and $C=O_d$. In view of \eqref{propEqCalore} and \eqref{subsemigroup1}, the projections of the propagators \eqref{propEqCalore} are of the second type. Let us focus on the condition $B=O_d$ first, and observe that it entails $D=A^{-\top}$. A simple factorization shows then that
			\begin{equation}
				S=\begin{pmatrix}
					A & O_d\\
					C & D
				\end{pmatrix}=
				\begin{pmatrix}	
					A & O_d\\
					O_d & A^{-\top}
				\end{pmatrix}
				\begin{pmatrix}	
					I_d & O_d\\
					A^\top C & I_d
				\end{pmatrix},
			\end{equation}
			and it is straightforward to check that the two factors separately belong to $\Sp_+(d,\bC)$ if $A\in\GL(d,\bR)$ and $\Im(A^\top C)\geq0$. Nevertheless, since $\Sp_+(d,\bC)$ is a semigroup, this decomposition is not sufficient to obtain the characterization
			\begin{equation}
				\text{$S\in{\Sp}_+(d,\bC)$ if and only if $A\in\GL(d,\bR)$ and $\Im(A^\top C)\geq0$},
			\end{equation}
			which is actually true.
			\begin{proposition}\label{prop51GG}
			Let $S\in\Sp(d,\bC)$ with blocks as in \eqref{blockS} satisfying $B=O_d$. Then, $S\in \Sp_+(d,\bC)$ if and only if $A\in\GL(d,\bC)$ and $\Im(A^\top C)\geq0$. 
		\end{proposition}
		\begin{proof}
			Recall the definition of $\cM_S$ in \eqref{defMS}. Let us write $A=A_1+iA_2$, $C=C_1+iC_2$ and $A^{-\top}=D_1+iD_2$. Of course, $A\in\GL(d,\bR)$ if and only if $A_2=O_d$ or, equivalently, $D_2=O_d$. A straightforward computation allows us to write
			\begin{equation}\label{MSprop51}
				\cM_S=\cK\left(\begin{array}{cc|cc}
					-C_1^\top A_2+A_1^\top C_2& -C_2^\top A_2+A_2^\top C_2 & A_1^\top D_2 & A_2^\top D_2\\
					C_2^\top A_2-A_2^\top C_2 & -C_1^\top A_2+A_1^\top C_2 & -A_2^\top D_2 & A_1^\top D_2\\
					\hline
					-D_1^\top A_2 & -D_2^\top A_2 & O_d & O_d\\
					D_2^\top A_2 & -D_1^\top A_2 & O_d & O_d
				\end{array}\right)\cK,
			\end{equation}
			where $\cK$, defined by
            \begin{equation}\label{defKpermConj}
                \cK=\begin{pmatrix}
                    I_d & O_d & O_d & O_d\\
                    O_d & O_d & I_d & O_d\\
                    O_d & I_d & O_d & O_d\\
                    O_d & O_d & O_d & I_d
                \end{pmatrix},
            \end{equation}
            permutes the two central columns/rows. 
            Clearly, $\cM_S\geq0$ if and only if $\cK\cM_S\cK\geq0$. Theorem \ref{theoremSchur} implies that $\cM_S\geq0$ if and only if both
			\begin{equation}\label{cond12conj}
				\begin{pmatrix}
					A_1^\top D_2 & A_2^\top D_2\\
					-A_2^\top D_2 & A_1^\top D_2
				\end{pmatrix}=
				\begin{pmatrix}
					A_1^\top & A_2^\top\\
					-A_2^\top & A_1^\top
				\end{pmatrix}
				\begin{pmatrix}
					D_2  & O_d \\
					O_d  & D_2
				\end{pmatrix}=O_{2d}
			\end{equation}
			and 
			\begin{equation}\label{cond22conj}
				\begin{pmatrix}
					A_1^\top C_2 & O_d\\
					O_d & A_1^\top C_2
				\end{pmatrix}\geq 0
			\end{equation}
			hold. The matrix 
			\begin{equation}
			\begin{pmatrix}
					A_1^\top & A_2^\top\\
					-A_2^\top & A_1^\top
				\end{pmatrix}
			\end{equation}
			is the realification of $A^{\top}$, and it is therefore invertible, whence \eqref{cond12conj} is equivalent to $D_2=O_d$. Consequently, it must be $A\in\GL(d,\bR)$. Condition \eqref{cond22conj} is then equivalent to $A^\top C_2\geq0$, i.e., $\Im(A^\top C)\geq0$. 
This concludes the proof.
			
		\end{proof}

		The following result is the analogue for the subsemigroup of upper-block-triangular matrices, including the projections of the heat semigroup \eqref{subsemigroup1}. The proof is completely analogous to that of Proposition \ref{prop51GG}, with \eqref{MSprop51} being replaced by
		\begin{equation}\label{MSprop52}
				\cM_S=\cK\left(\begin{array}{cc|cc}
					O & O & A_1^\top D_2 & A_2^\top D_2\\
					O & O & -A_2^\top D_2 & A_1^\top D_2\\
					\hline
					-D_1^\top A_2 & -D_2^\top A_2 & -D_1^\top B_2+B_1^\top D_2 & -D_2^\top B_2+B_2^\top D_2\\
					D_2^\top A_2 & -D_1^\top A_2 & D_2^\top B_2 -B_2^\top D_2 & -D_1^\top B_2 +B_1^\top D_2
				\end{array}\right)\cK,
			\end{equation}
		here, $A=A_1+iA_2$, $B=B_1+iB_2$, $A^{-\top}=D_1+iD_2$ and $\cK$ is as in \eqref{defKpermConj}. 
		\begin{proposition}\label{prop52GG}
			Let $S\in\Sp(d,\bC)$ with blocks as in \eqref{blockS} with $C=O_d$. Then, $S\in \Sp_+(d,\bC)$ if and only if $A\in\GL(d,\bC)$ and $\Im(A^{-1} B)=\Im(D^\top B)\leq0$.
		\end{proposition}
		\begin{remark}
		Let us mention that a straightforward corollary of both Propositions \ref{prop51GG} and \ref{prop52GG} gives that matrices in the form
		\begin{equation}
			\cD_E=\begin{pmatrix}
				E^{-1} & O_d\\
				O_d & E^\top
			\end{pmatrix}
		\end{equation}
		are in $\Sp_+(d,\bC)$ if and only if $E\in\GL(d,\bR)$, i.e., if and only if $\cD_E\in\Sp(d,\bR)$.
        \end{remark}

        \subsection{Wigner analysis}\label{subsec:WignerAnalysis}

            Let $\widehat S\in\Mp_+(d,\bC)$ and consider its polar decomposition $\widehat S=\widehat U\widehat Z$, with $\widehat U\in\Mp(d,\bR)$ and $\widehat Z\in\Mp_0(d,\bC)$. If $U$ and $Z$ denote the projections of $\widehat U$ and $\widehat Z$, respectively, then for $f,g\in L^2(\rd)$ we obtain
            \begin{align}\label{WignerPolar}
                W(\widehat Sf,\widehat Sg)(z)= W(\widehat Zf,\widehat Zg)\big(U^{-1}z\big)= K_ZW(f,g)\big(U^{-1}z\big),
            \end{align}
            for a suitable Wigner operator $K_Z:\cS(\rdd)\to\cS'(\rdd)$. Unless $Z=\pm I_{2d}$, we have $\widehat Z\notin\Mp(d,\bR)$, and therefore $K_Z$ does not correspond to a symplectic change of variables.
        
		    In this subsection, we compute the Wigner operator of a general $\widehat S\in\Mp_+(d,\bC)$ and, by the above argument, this amounts to computing $K_Z$ in \eqref{WignerPolar}.
            
            As in the previous sections, if $\widehat S\in\Mp_+(d,\bC)$, we use the factorization \eqref{HFactoriz} and write $\widehat S=\widehat U_1\widehat \Xi\widehat U_2$, with $\widehat U_1, \widehat U_2\in\Mp(d,\bR)$ and $\widehat \Xi$ atomic metaplectic contraction. For the rest of this section, we omit the phase factors defining metaplectic operators. We also recall that
		\begin{equation}
			\cF_2f(x,\xi)=({\id}_{L^2}\otimes\cF)f(x,\xi)=i^{-d/2}\int_{\rd}f(x,y)e^{-2\pi iy\xi}dy, \qquad f\in\cS(\rdd),
		\end{equation}
		the {\em partial Fourier transform with respect to the last $d$ variables}, belongs to $\Mp(2d,\bR)$ and has projection
		\begin{equation}\label{defAFT2}
			\cA_{FT2}=I_{2d}\otimes J=\begin{pmatrix}
				I_d & O_d & O_d & O_d \\
				O_d & O_d & O_d & I_d\\
				O_d & O_d & I_d & O_d\\
				O_d & -I_d & O_d & O_d 
			\end{pmatrix}.
		\end{equation} 
        
		The main result in this direction is the following. 
		\begin{theorem}\label{thmS1S2S3W}
			Under the notation above, $\widehat S=\widehat U_1\widehat\Xi\widehat U_2$,
			\begin{equation}
				W(\widehat Sf,\widehat Sg)=KW(f,g), \qquad f,g\in\cS(\rd),
			\end{equation}
			where the Wigner operator $K$ is in $\Mp_+(2d,\bC)$ with
			\begin{equation}\label{WignerOpMetappiu}
				K=(\cF_2\mathfrak{T}_{2^{-1/2}U_1 })^{-1}(\widehat \Xi\otimes\widehat \Xi)(\cF_2\mathfrak{T}_{2^{-1/2} U_2^{-1}}).
			\end{equation}
            Consequently, if $\widehat S=\widehat U\widehat Z$ is the polar decomposition of $\widehat S$, formula \eqref{WignerPolar} holds with $U=U_1U_2$ and
            \begin{equation}\label{WignerOpMetappiu2}
                K_Z=(\cF_2\mathfrak{T}_{2^{-1/2}U_2 })^{-1}(\widehat \Xi\otimes\widehat \Xi)(\cF_2\mathfrak{T}_{2^{-1/2} U_2^{-1}})
            \end{equation}
            and $K=\mathfrak{T}_{U^{-1}} K_Z$ is the polar decomposition of $K$.
		\end{theorem} 
		
        The idea is to factorize $\widehat\Xi$ as in Remark \ref{remAMC}, $\widehat\Xi=\mathfrak{p}_{i\Delta}\frR_\Theta$.
		Moreover, note that if $K_\Xi$ is the Wigner operator of $\widehat \Xi$, then
		\begin{equation}\label{GGWigOp1}
			W(\widehat Sf,\widehat Sg)=W(\widehat U_1\widehat \Xi\widehat U_2f,\widehat U_1\widehat \Xi\widehat U_2g)=\mathfrak{T}_{U_1^{-1}}K_\Xi\mathfrak{T}_{U_2^{-1}}W(f,g).
		\end{equation}
		We need to compute $K_\Xi$. 
		The following lemma is proven in a recent, unpublished, manuscript by two of the authors \footnote{\label{fn1} See E. Cordero, G. Giacchi, and L. Rodino. Wigner and Gabor phase-space analysis of propagators for evolution equations. {\em arXiv:2511.19400}.}.
		\begin{lemma}\label{lemmatensor}
		Let $\widehat S_1,\widehat S_2\in\Mp_+(d,\bC)$. There exists a unique operator $\widehat S\in\Mp_+(2d,\bC)$ such that
		\begin{equation}
			\widehat S(f\otimes g)=\widehat S_1f\otimes \widehat S_2g, \qquad f,g\in L^2(\rd).
		\end{equation}
		Moreover, if $S_j=\pi^{Mp}_+(\widehat S_j)$, $j=1,2$, then $\pi^{Mp}_+(\widehat S)=S_1\otimes S_2$.
	\end{lemma}
		\begin{lemma}\label{lemma1gg1}
			Under the notation above, 
			\begin{equation}\label{merge1}
				W(\mathfrak{p}_{i\Delta}f,\mathfrak{p}_{i\Delta}g)=(\mathfrak{p}_{2i\Delta}\otimes \mathfrak{m}_{i\Delta/2})W(f,g), \qquad f,g\in L^2(\rd),
			\end{equation}
			where $\mathfrak{m}_{i\Delta/2}$ is defined as in \eqref{defmP}.
		\end{lemma}
		\begin{proof}
			We argue at the level of projections. We have
			\begin{equation}\label{Wigpp}
				W(\mathfrak{p}_{i\Delta}f,\mathfrak{p}_{i\Delta}g)=\widehat{\cA_{1/2}}(\mathfrak{p}_{i\Delta}\otimes \mathfrak{p}_{i\Delta})(f\otimes\overline g),
			\end{equation}
			where we used that $\overline{\mathfrak{p}_{i\Delta}g}=\mathfrak{p}_{i\Delta}\overline g$, and
			\begin{equation}
				\widehat{\cA_{1/2}}F(x,\xi)=\int_{\rd}F(x+y/2,x-y/2)e^{-2\pi i\xi y}dy, \qquad F\in\cS(\rdd)
			\end{equation}
			is a metaplectic operator in $\Mp(2d,\bR)$ with projection
			\begin{equation}\label{projWig}
				\cA_{1/2}=\begin{pmatrix}
					I_d/2 & I_d/2 & O_d & O_d\\
					O_d & O_d & I_d/2 & -I_d/2\\
					O_d & O_d & I_d & I_d\\
					-I_d & I_d & O_d & O_d
				\end{pmatrix}.
			\end{equation}
			At the level of projections, \eqref{Wigpp} is
			\begin{equation}
				\cA_{1/2}(V_{i\Delta}\otimes V_{i\Delta})=\cA_{1/2}(V_{i\Delta}\otimes V_{i\Delta})\cA_{1/2}^{-1}\cA_{1/2},
			\end{equation}
			whence
			\begin{equation}
				W(\mathfrak{p}_{i\Delta}f,\mathfrak{p}_{i\Delta}g)=\widehat\cB W(f,g),
			\end{equation}
			with
			\begin{equation}
				\cB=\cA_{1/2}(V_{i\Delta}\otimes V_{i\Delta})\cA_{1/2}^{-1}.
			\end{equation}
			A direct computation shows that
			\begin{align}
				\cB&=
				\begin{pmatrix}
					I_d/2 & I_d/2 & O_d & O_d\\
					O_d & O_d & I_d/2 & -I_d/2\\
					O_d & O_d & I_d & I_d\\
					-I_d & I_d & O_d & O_d
				\end{pmatrix}
				\begin{pmatrix}
					I_d & O_d & O_d & O_d\\
					O_d & I_d & O_d & O_d\\
					i\Delta & O_d & I_d & O_d\\
					O_d & i\Delta & O_d & I_d
				\end{pmatrix}
				\begin{pmatrix}
					I_d & O_d & O_d & -I_d/2\\
					I_d & O_d & O_d & I_d/2\\
					O_d & I_d & I_d/2 & O_d\\
					O_d & -I_d & I_d/2 & O_d
				\end{pmatrix}\\
				&=
				\begin{pmatrix}
					I_d & O_d & O_d & O_d\\
					O_d & I_d & O_d & -i\Delta/2\\
					2i\Delta & O_d & I_d & O_d\\
					O_d & O_d & O_d & I_d
				\end{pmatrix}=V_{2i\Delta}\otimes V_{-i\Delta/2}^\top.
			\end{align}
			The assertion follows by Lemma \ref{lemmatensor} and Remark \ref{remConvolution}.
		\end{proof}
		
		The equivalent of Lemma \ref{lemma1gg1} was proven in the unpublished work in Footnote \ref{fn1}.
		
		\begin{proposition}\label{propIntertWig2}
			Under the notation above,
			\begin{align}
				\label{merge2}
				W(\mathfrak{R}_\Theta f,\mathfrak{R}_\Theta g)&=
				\mathfrak{T}_{2^{1/2} I_{2d}}(\mathfrak{R}_{\Theta}\otimes \mathfrak{R}_{\Theta})\mathfrak{T}_{2^{-1/2} I_{2d}}W(f,g), \qquad f,g\in L^2(\rd)
			\end{align}
			up to a constant.
		\end{proposition}
		
		Equations \eqref{merge1} and \eqref{merge2} together give the operator $K_\Xi$ in \eqref{GGWigOp1}.
		
		\begin{proposition}
			Let $\widehat \Xi$ be an atomic metaplectic contraction. Then,
			\begin{equation}
				W(\widehat \Xi f,\widehat \Xi g)=K_\Xi W(f,g), \qquad f,g\in L^2(\rd),
			\end{equation}
			with
			\begin{equation}\label{GGWigOp2}
				K_\Xi=(\cF_2\mathfrak{T}_{2^{-1/2}I_d})^{-1}(\widehat \Xi\otimes \widehat \Xi)(\cF_2\mathfrak{T}_{2^{-1/2}I_d}).
			\end{equation}
		\end{proposition} 
		\begin{proof}
			Let $f,g\in L^2(\rd)$. By \eqref{S2Vp}, we can write
			\begin{align}
				W(\widehat \Xi f,\widehat \Xi g)&=W(\mathfrak{p}_{i\Delta}\mathfrak{R}_{\Theta}f,\mathfrak{p}_{i\Delta}\mathfrak{R}_{\Theta}g).
			\end{align}
			By applying \eqref{merge1} and \eqref{merge2} in the order, we obtain
			\begin{align}
				W(\widehat \Xi f,\widehat \Xi g)&=(\mathfrak{p}_{2i\Delta}\otimes \mathfrak{m}_{i\Delta/2})W(\mathfrak{R}_{\Theta}f,\mathfrak{R}_{\Theta}g)\\
				\label{exp2gg12}
				&=(\mathfrak{p}_{2i\Delta}\otimes \mathfrak{m}_{i\Delta/2})\mathfrak{T}_{2^{1/2} I_{2d}}(\mathfrak{R}_{\Theta}\otimes \mathfrak{R}_{\Theta})\mathfrak{T}_{2^{-1/2} I_{2d}}W(f,g).
			\end{align}
			Observe that $\mathfrak{T}_{2^{1/2} I_{2d}}=\mathfrak{T}_{2^{1/2} I_{d}}\otimes \mathfrak{T}_{2^{1/2} I_{d}}$ and $\mathfrak{m}_{i\Delta/2}=\cF^{-1}\mathfrak{p}_{i\Delta/2}\cF$, whence
			\begin{align}
				W(\widehat\Xi f,\widehat \Xi g)&=\cF_2^{-1}(\mathfrak{p}_{2i\Delta}\otimes \mathfrak{p}_{i\Delta/2})\cF_2(\mathfrak{T}_{2^{1/2} I_{d}}\otimes \mathfrak{T}_{2^{1/2} I_{d}})(\mathfrak{R}_{\Theta}\otimes \mathfrak{R}_{\Theta})\mathfrak{T}_{2^{-1/2} I_{2d}}W(f,g)\\
				&=\cF_2^{-1}(\mathfrak{p}_{2i\Delta}\otimes \mathfrak{p}_{i\Delta/2})({\id}_{L^2}\otimes\cF)(\mathfrak{T}_{2^{1/2} I_{d}}\otimes \mathfrak{T}_{2^{1/2} I_{d}})(\mathfrak{R}_{\Theta}\otimes \mathfrak{R}_{\Theta})\mathfrak{T}_{2^{-1/2} I_{2d}}W(f,g)\\
				&=\cF_2^{-1}(\mathfrak{p}_{2i\Delta}\otimes \mathfrak{p}_{i\Delta/2})(\mathfrak{T}_{2^{1/2} I_{d}}\otimes \mathfrak{T}_{2^{-1/2} I_{d}})({\id}_{L^2}\otimes\cF)(\mathfrak{R}_{\Theta}\otimes \mathfrak{R}_{\Theta})\mathfrak{T}_{2^{-1/2} I_{2d}}W(f,g).
				\end{align}
				It is straightforward to check that $V_{2i\Delta}\cD_{2^{1/2}I_d}=\cD_{2^{1/2}I_d}V_{i\Delta}$ and similarly $V_{i\Delta/2}\cD_{2^{-1/2}I_d}=\cD_{2^{-1/2}I_d}V_{i\Delta}$, whence $\mathfrak{p}_{2i\Delta}\mathfrak{T}_{2^{1/2}I_d}=\mathfrak{T}_{2^{1/2}I_d}\mathfrak{p}_{i\Delta}$ and $\mathfrak{p}_{-i\Delta/2}\mathfrak{T}_{2^{-1/2}I_d}=\mathfrak{T}_{2^{-1/2}I_d}\mathfrak{p}_{-i\Delta/2}$. Consequently,
				\begin{align}
				W(\widehat \Xi f,\widehat \Xi g)&=\cF_2^{-1}(\mathfrak{T}_{2^{1/2} I_{d}}\otimes \mathfrak{T}_{2^{-1/2} I_{d}})
				(\mathfrak{p}_{i\Delta}\otimes \mathfrak{p}_{i\Delta})({\id}_{L^2}\otimes\cF)(\mathfrak{R}_{\Theta}\otimes \mathfrak{R}_{\Theta})\mathfrak{T}_{2^{-1/2} I_{2d}}W(f,g)\\
				&=(\mathfrak{T}_{2^{1/2} I_{d}}\otimes \mathfrak{T}_{2^{1/2} I_{d}})\cF_2^{-1}
				(\mathfrak{p}_{i\Delta}\otimes \mathfrak{p}_{i\Delta})({\id}_{L^2}\otimes\cF)(\mathfrak{R}_{\Theta}\otimes \mathfrak{R}_{\Theta})\mathfrak{T}_{2^{-1/2} I_{2d}}W(f,g)\\
				&=\mathfrak{T}_{2^{1/2} I_{2d}}\cF_2^{-1}
				(\mathfrak{p}_{i\Delta}\otimes \mathfrak{p}_{i\Delta})({\id}_{L^2}\otimes\cF)(\mathfrak{R}_{\Theta}\otimes \mathfrak{R}_{\Theta})\mathfrak{T}_{2^{-1/2} I_{2d}}W(f,g).
				\end{align}
				
				 It is also straightforward to verify that $J\mathcal{R}_\Theta=\mathcal{R}_\Theta J$, and consequently $\cF\mathfrak{R}_\Theta=\mathfrak{R}_\Theta \cF$, which yields
				
				\begin{align}
				W(\widehat \Xi f,\widehat \Xi g)&=\mathfrak{T}_{2^{1/2} I_{2d}}\cF_2^{-1}
				(\mathfrak{p}_{i\Delta}\otimes \mathfrak{p}_{i\Delta})(\mathfrak{R}_{\Theta}\otimes \mathfrak{R}_{\Theta})({\id}_{L^2}\otimes\cF)\mathfrak{T}_{2^{-1/2} I_{2d}}W(f,g)\\
				&=(\cF_2\mathfrak{T}_{2^{-1/2} I_{2d}})^{-1}(\widehat \Xi\otimes\widehat \Xi)(\cF_2\mathfrak{T}_{2^{-1/2} I_{2d}})W(f,g),
				\end{align}
				thereby concluding the proof.
				
		\end{proof}
		\begin{proof}[Proof of Theorem \ref{thmS1S2S3W}]
		The expression of $K$ in \eqref{WignerOpMetappiu} follows now by \eqref{GGWigOp1} and \eqref{GGWigOp2}. Formula \eqref{WignerOpMetappiu2} follows then directly from \eqref{WignerOpMetappiu}:

        \begin{align}
            K&=(\cF_2\mathfrak{T}_{2^{-1/2}U_1 })^{-1}(\widehat \Xi\otimes\widehat \Xi)(\cF_2\mathfrak{T}_{2^{-1/2} U_2^{-1}})\\
            &=(\mathfrak{T}_{U_1^{-1}}\mathfrak{T}_{U_2^{-1}})(\mathfrak{T}_{U_2}\mathfrak{T}_{2^{-1/2}}\cF_2^{-1})(\widehat \Xi\otimes\widehat \Xi)(\cF_2\mathfrak{T}_{2^{-1/2} U_2^{-1}}).
        \end{align}

        Finally, observe that $K_Z$ is in $\Mp_0(2d,\bC)$ by item $(ii)$ of Corollary \ref{CorPolar}, whereas $\mathfrak{T}_{U^{-1}}\in \Mp(2d,\bR)$, so that \eqref{WignerOpMetappiu} is the polar decomposition of $K$. The proof of Theorem \ref{thmS1S2S3W} is therefore concluded. 
        \end{proof}

		\subsection{Metaplectic operators and complex conjugation}\label{subsec:complexConjugation}
			In this subsection, we delve into the details of the intertwining relation between metaplectic operators and complex conjugation, thereby rephrasing \cite[Appendix A]{CG02} in our context and proving Theorem \ref{thmComplexConj}. The characterization of metaplectic operators which commute with complex conjugation is then provided in terms of the structure of their projections onto $\Sp_+(d,\bC)$.
			\begin{proof}[Proof of Theorem \ref{thmComplexConj}]
				As always, let us write $S=U_1\Xi U_2$ with $U_1,U_2\in\Sp(d,\bR)$ and $\Xi$ metaplectic atomic contraction. Since Gaussian products trivially commute with complex conjugation and also the atoms $\mathfrak{R}_\Theta$ commute with complex conjugation, as it can be verified directly from \eqref{defRtau}, we have by \eqref{S2Vp} that $\widehat \Xi \bar f=\overline{\widehat \Xi f}$. By applying Proposition \ref{PropCG02}, we get
				\begin{equation}
					Tf=\overline{\widehat{U}_1\widehat \Xi\widehat{U}_3\bar f}={\widehat{\widetilde U}_1\widehat \Xi\widehat{\widetilde U}_2 f}.
				\end{equation}
				Since $\widehat{\widetilde U}_{1,2}\in\Mp(d,\bR)$ and $\widehat \Xi$ is an atomic metaplectic contraction, $T=\widehat{\widetilde U}_1\widehat \Xi\widehat{\widetilde U}_2\in\Mp_+(d,\bC)$. Its projection is $\widetilde S=\widetilde{U}_1\Xi\widetilde {U}_2$. By writing explicitly $\Xi$ as in \eqref{S2Vp}, we see that $\Re(B_2)=\Re(C_2)=O_d$ and $\Im(A_2)=\Im(D_2)=O_d$, therefore if
				\begin{equation}
					U_1=\begin{pmatrix}
						A_1 & B_1\\
						C_1 & D_1
					\end{pmatrix}, \qquad
                    \Xi=\begin{pmatrix}
						A_2 & B_2\\
						C_2 & D_2
					\end{pmatrix},\qquad
                    U_2=\begin{pmatrix}
						A_3 & B_3\\
						C_3 & D_3
					\end{pmatrix},
				\end{equation}
				we can write
				\begin{equation}\label{expressiontildeS}
					\widetilde S=\begin{pmatrix}
						A_1 & -B_1\\
						-C_1 & D_1
					\end{pmatrix}
					\begin{pmatrix}
						\Re(A_2) & i\Im(B_2)\\
						i\Im(C_2) & \Re(D_2)
					\end{pmatrix}
					\begin{pmatrix}
						A_3 & -B_3\\
						-C_3 & D_3
					\end{pmatrix}.
				\end{equation}
				By comparing \eqref{expressiontildeS} with its analogue for $S$, obtained by replacing $-B_j,-C_j$ with $B_j$ and $C_j$ ($j=1,3$) in \eqref{expressiontildeS}, we get
				\begin{align}
					\Re(\widetilde S)&=\begin{pmatrix}
						A_1\Re(A_2)A_3+B_1\Re(D_2)C_3 & -A_1\Re(A_2)B_3-B_1\Re(D_2)D_3\\
						-C_1\Re(A_2)A_3-D_1\Re(D_2)C_3 & C_1\Re(A_2)B_3 +D_1\Re(D_2)D_3
					\end{pmatrix}\\
					&=\begin{pmatrix}
						\Re(A) & -\Re(B)\\
						-\Re(C) & \Re(D)
					\end{pmatrix}.
				\end{align}
				Similarly, we obtain
				\begin{align}
					\Im(\widetilde S)&=-\begin{pmatrix}
						B_1\Im(C_2)A_3 +A_1\Im(B_2)C_3 & -B_1\Im(C_2)B_3-A_1\Im(B_2)D_3\\
						-D_2\Im(C_2)A_3-C_1\Im(B_2)C_3 & D_2\Im(C_2)B_3 +C_1\Im(B_2)D_3
					\end{pmatrix}\\
					&=-\begin{pmatrix}
						\Im(A) & -\Im(B)\\
						-\Im(C) & \Im(D)
					\end{pmatrix}.
				\end{align}
				Therefore,
				\begin{align}
					\widetilde S&=\Re(\widetilde S)+i\Im(\widetilde S)=\begin{pmatrix}
						\Re(A) & -\Re(B)\\
						-\Re(C) & \Re(D)
					\end{pmatrix}-i\begin{pmatrix}
						\Im(A) & -\Im(B)\\
						-\Im(C) & \Im(D)
					\end{pmatrix}\\
					&=\begin{pmatrix}
						\Re(A)-i\Im(A) & -\Re(B)+i\Im(B)\\
						-\Re(C)+i\Im(C) & \Re(D)-i\Im(D)
					\end{pmatrix},
				\end{align}
				which agrees with \eqref{defStildecomplex}. This concludes the proof.
			\end{proof}

            It follows trivially by \eqref{defStildecomplex} that a matrix $S\in\Sp(d,\bC)$ satisfies $S=\widetilde S$ if and only if
            \begin{equation}\label{blockSconj}
                S=\begin{pmatrix}
                    A & iB\\
                    iC & D
                \end{pmatrix}, \qquad \text{with} \qquad S'=\begin{pmatrix}
                    A & -B\\
                    C & D
                \end{pmatrix}\in\Sp(d,\bR).
            \end{equation}
            Moreover, by multiplying two matrices in the form \eqref{blockSconj} it is straightforward to verify that 
            \begin{equation}
                \left\{ S\in\Sp(d,\bC) : S=\widetilde S \right\}
            \end{equation}
            is a subgroup of $\Sp(d,\bC)$. Therefore,
            \begin{equation}
                \left\{ S\in{\Sp}_+(d,\bC) : S=\widetilde S \right\}
            \end{equation}
            is a subsemigroup of $\Sp_+(d,\bC)$.
            We conclude this section by characterizing the operators in this subsemigroup and the corresponding projections. We need a preliminary structural lemma.
            \begin{lemma}\label{structuralLemma}
                Let $U\in\Sp(d,\bR)$ have blocks \eqref{blockS}. The following statements are equivalent.
                \begin{enumerate}[(i)]
                    \item $\mathrm{range}(C^\top B)\subseteq \mathrm{range}(C^\top A)$.
                    \item $A\in\GL(d,\bR)$.
                \end{enumerate}
            \end{lemma}
            \begin{proof}
                We first prove that $(i)\Leftrightarrow(ii')$ and then $(ii')\Leftrightarrow(ii)$, where\\
                $(ii')$ $\mathrm{range}(B)\subseteq \mathrm{range}(A)$.\\
                The implication $(ii')\Rightarrow(i)$ is trivial, and thus we prove the converse. Assume that $\mathrm{range}(C^\top B)\subseteq \mathrm{range}(C^\top A)$. For $\zeta=Bx\in \mathrm{range}(B)$, $x\in\rd$, we must show that there exists $z\in\rd$ such that $\zeta=Az$. By the assumption, there exists $y\in\rd$ such that $C^\top Bx=C^\top Ay$. Then, there exists $v\in \ker(C^\top)=\mathrm{range}(C)^\perp$ such that $ Bx-Ay=v$. By Proposition \ref{propMO} $(i)$, there exists a unique $w\in \ker(C)$ such that $Aw=v$. It follows that $Bx=A(y+w)\in \mathrm{range}(A)$. We now turn to the equivalence $(ii')\Leftrightarrow(ii)$. Observe that if $A$ is invertible, then $(ii')$ is trivial. We need to check that $(ii')\Rightarrow(ii)$. If $\mathrm{range}(B)\subseteq \mathrm{range}(A)$, then $\mathrm{range}(A)^\perp\subseteq \mathrm{range}(B)^\perp=\ker(B^\top)$. By item $(ii)$ of Proposition \ref{propMO}, we obtain $\ker(A)=B^\top(\mathrm{range}(A)^\perp)=\{0\}$. This concludes the proof.
            \end{proof}
        
            We are ready to prove the main result of this subsection.
            
            \begin{theorem}\label{StructuralTheorem}
                Let $S\in\Sp(d,\bC)$ be such that $\widetilde S=S$. The following statements are equivalent.
                \begin{enumerate}[(i)]
                    \item $S\in\Sp_+(d,\bC)$.
                    \item If $S$ has blocks as in \eqref{blockSconj}, then
                    \begin{equation}
                       A\in\GL(d,\bR), \qquad A^\top C\geq0, \qquad \text{and} \qquad AB^\top\leq 0.
                    \end{equation}
                    \item If $S$ has blocks as in \eqref{blockS}, then $\Re(C)=\Re(B)=O_d$ and
                    \begin{equation}
                       A\in\GL(d,\bR),\qquad \Im(A^\top C)\geq0 \qquad \text{and} \qquad \Im(AB^\top)\leq0.
                    \end{equation}
                \end{enumerate}
            \end{theorem}
            \begin{proof}
                Let $S$ have blocks \eqref{blockSconj} and $\cM_S$ be the corresponding matrix in \eqref{defMS}. By Proposition \ref{SmathcalM}, $S\in\Sp_+(d,\bC)$ if and only if
                \begin{equation}\label{EqGG1}
                    \cK\cM_S\cK=\left(\begin{array}{cc|cc}
                        A^\top C & O_d & O_d &-C^\top B \\
                        O_d &  A^\top C & C^\top B & O_d\\
                        \hline
                        O_d & B^\top C & -D^\top B & O_d\\
                        -B^\top C & O_d & O_d & -D^\top B
                    \end{array}\right)\geq0,
                \end{equation}
            where $\cK$ is defined as in \eqref{defKpermConj}.
			By Theorem \ref{theoremSchur}, condition \eqref{EqGG1} is equivalent to the simultaneous validity of the following:
            \begin{enumerate}[(1)]
                \item $A^\top C\geq0$.
                \item $\mathrm{range}(C^\top B)\subseteq \mathrm{range}(C^\top A)$.
                \item $-D^\top B-B^\top C(A^\top C)^+C^\top B\geq0$.
            \end{enumerate}
            Condition $\mathrm{range}(C^\top B)\subseteq \mathrm{range}(C^\top A)$ is equivalent either to $\mathrm{range}(B)\subseteq \mathrm{range}(A)$ or to $A\in\GL(d,\bR)$ by Lemma \ref{structuralLemma} applied to $S'$ in \eqref{blockSconj}. On the other hand, item (3) is equivalent to 
            \begin{equation}\label{EqGG2}
                -D^\top Bx\cdot x-(A^\top C)^+C^\top Bx\cdot C^\top Bx.
            \end{equation}
            Set $y=A^{-1}Bx$. Then, equation \eqref{EqGG2} reads as
            \begin{align}
                 0&\leq-D^\top Ay\cdot x-(A^\top C)^+C^\top Ay\cdot C^\top Ay=-D^\top Ay\cdot x-C^\top Ay\cdot y\\
                 &=-(I-B^\top C)y\cdot x-C^\top Ay\cdot y=-y\cdot x+B^\top Cy\cdot x-C^\top Ay\cdot y\\
                 &=-y\cdot x+Cy\cdot Ay-C^\top Ay\cdot y=-y\cdot x,
            \end{align}
            where we used that $D^\top A+B^\top C=I_d$ by \eqref{blockSconj}. Since $A$ is invertible we infer that (3) is equivalent to
            \begin{align}
                 0&\leq-y\cdot x=-A^{-1}Ay\cdot x=-A^{-1}Bx\cdot x.
            \end{align}
            We have proven that $\cM_S\geq0$ if and only if
            \begin{equation}
                A^\top C\geq0,\qquad A\in\GL(d,\bR), \qquad \text{and}\qquad 
                A^{-1}B\leq 0.
            \end{equation}
            Observe that $A^{-1}B\leq0$ if and only if $A^{-1}(BA^\top)A^{-\top}\leq0$ or, equivalently, $AB^\top\leq0$. Finally, $(iii)$ is a restatement of $(ii)$ in terms of the block decomposition \eqref{blockS}. This concludes the proof.

            \end{proof}

            \begin{corollary}\label{corSemigroup}
                Let $\widehat S\in\Mp_+(d,\bC)$ have projection \eqref{blockS}, then $\widehat S$ commutes with complex conjugation if and only if
                \begin{equation}\label{conjForm}
                    \widehat Sf(x)=c|\det(A)|^{-1/2}e^{i\pi CA^{-1}x\cdot x}\int_{\rd}\widehat f(\xi)e^{-i\pi A^{-1}B\xi\cdot\xi}e^{2\pi i\xi\cdot A^{-1}x}d\xi, \qquad f\in L^1(\rd),
                \end{equation}
                for some $c\in\bC$ with $|c|\leq1$.
            \end{corollary}
            \begin{proof}
                It follows by Theorem \ref{StructuralTheorem} and by the factorization
                \begin{equation}
                    S=\cD_{A^{-1}}V_{A^\top C}V^\top_{A^{-1}B},
                \end{equation}
                holding in $\Sp_+(d,\bC)$ if and only if the conditions on $S$ of Theorem \ref{StructuralTheorem} are met.
            \end{proof}
            
			The results in this subsection will be applied in Section \ref{sec.apptimefreq} below to the theory of time-frequency representations.
					
		\section{Applications to time-frequency analysis}\label{sec.apptimefreq}
		The notion of time-frequency representation lies at the foundation of time-frequency analysis. Across essentially every line of inquiry in the field, one is naturally led to metaplectic operators. 
        In the present section, we define the analogue of metaplectic Wigner distributions in the context of the metaplectic semigroup $\Mp_+(2d,\bC)$, and we characterize those that satisfy certain fundamental properties arising in operator theory, time-frequency analysis and quantum mechanics. 
			
		\begin{definition}
			Let $\widehat\cA\in \Mp_+(2d,\bC)$. The associated {\em metaplectic Wigner distribution}, or $\widehat\cA$-{\em Wigner distribution}, is the operator $W_\cA:L^2(\rd)\times L^2(\rd)\to L^2(\rdd)$ given by
			\begin{equation}
				W_\cA(f,g)=\widehat\cA(f\otimes\overline g), \qquad f,g\in L^2(\rd).
			\end{equation}
			If $f=g$, we write $W_\cA f=W_\cA(f,f)$.
		\end{definition}
		
		The reader may observe that this definition resembles \cite[Definition 4.1]{Cordero_Part_I}, with the difference that we now allow $\widehat\cA\in\Mp_+(2d,\bC)$. The phraseologies {\em metaplectic Wigner distribution} and {\em $\cA$-Wigner distribution} are used in the aforementioned work, and in the subsequent ones by the same authors, in the case $\widehat\cA\in\Mp(2d,\bR)$. This shall not cause confusion; therefore, we adopt the same terminology in our more general setting to avoid unnecessary burden on notation and terminology.
		
		\begin{example}
			Let us define, for $M\in\Sigma_{\geq0}(2d)$, the time-frequency representation
			\begin{equation}
				W_\cA(f,g)=\Phi_M\ast W(f,g), \qquad f,g\in L^2(\rd),
			\end{equation}
			where $\Phi_M(z)=e^{i\pi Mz\cdot z}$. If $\Im(M)>0$, this metaplectic Wigner distribution does not fall into the framework of \cite{Cordero_Part_I}. This is one of the reason that motivates us to consider the case of Wigner distributions $W_\cA$, with $\widehat\cA \in \Mp_+(2d,\mathbb{C})$.
		\end{example}
		
       We need the following lemma, whose proof is reported for the sake of completeness.

        \begin{lemma}\label{lemmaGaussians}
            Let $Q\in\Sym(d,\bR)\setminus\{O_d\}$, $Q\geq0$. Then, $\norm{\mathfrak{p}_{iQ}f}_2<\norm{f}_2$ for every Gaussian $f(x)=e^{-\pi Mx\cdot x}$, $M>0$. Consequently, if $\widehat Z\in\Mp_0(d,\bC)$ is not a multiple of the identity, then $\norm{\widehat Zf}_2<\norm{f}_2$ for every Gaussian $f$ as above.
        \end{lemma}
        \begin{proof}
            If $Q\geq0$, the Gaussian $e^{-\pi Qx\cdot x}=1$ if and only if $x\in \ker(Q)$. 
            Let $\Gamma=\{y:\mathrm{dist}(y,\ker(Q))<1\}$. Then, if $f(x)=e^{-\pi Mx\cdot x}$, $M>0$,
            \begin{align}
                \norm{\mathfrak{p}_{iQ}f}_2^2&=\int_{\rd}e^{-2\pi Q y\cdot y}|f(y)|^2dy\\
                &=\int_{\Gamma}e^{-2\pi Qy\cdot y}|f(y)|^2dy+\int_{\rd\setminus\Gamma} e^{-2\pi Q y\cdot y}|f(y)|^2dy<\norm{f}_2^2.
            \end{align}
            Finally, we can always write $\widehat Z=\widehat V^{-1}\Xi\widehat V$, for $\widehat\Xi$ atomic metaplectic contraction and $\widehat V\in\Mp(d,\bR)$. Then,
            \begin{align}
                \norm{\widehat Zf}_2=\norm{\widehat\Xi\widehat Vf}_2.
            \end{align}
            Since $\widehat\Xi$ is an atomic metaplectic contraction, we can always write $\widehat\Xi$ as $\widehat\Xi  = \mathfrak{p}_{i\Delta}\mathfrak{R}_\Theta$,
            see Remark \ref{remAMC} for the details.
            Since $\mathfrak{R}_\Theta=\bigotimes_{j=1}^d\mathfrak{R}_{\vartheta_j}$ and for each $j=1,\ldots,d$, it holds \begin{equation}
                \mathfrak{R}_{\vartheta_j}=\mathfrak{p}_{i\tanh\vartheta_j}\mathfrak{T}_{1/\cosh\vartheta_j}\cF^{-1}\mathfrak{p}_{i\tanh\vartheta_j}\cF,
            \end{equation} 
            we obtain that
            \begin{equation}
                \mathfrak{R}_\Theta=\mathfrak{p}_{iP}\mathfrak{T}_{E}\cF^{-1}\mathfrak{p}_{iP}\cF,
            \end{equation}
            where $P=\diag(\tanh\vartheta_1,\ldots,\tanh\vartheta_d)\geq0$ and $E=\diag(1/\cosh\vartheta_1,\ldots,1/\cosh\vartheta_d)\in\GL(d,\bR)$. Therefore,
            \begin{equation}
                \widehat\Xi =\mathfrak{p}_{i(Q+P)}\mathfrak{T}_{E}\cF^{-1}\mathfrak{p}_{iP}\cF.
            \end{equation}
           Then, 
            \begin{equation}\label{Gauss2}
                \norm{\widehat\Xi\widehat Vf}_2=\norm{\mathfrak{p}_{i(Q+P)}\mathfrak{T}_{E}\cF^{-1}\mathfrak{p}_{iP}\cF\widehat Vf}_2.
            \end{equation}
            By construction, $Q+P=\diag(\mu_1,\ldots,\mu_d)$ where, if $\widehat\Xi=\bigotimes_{j=1}^d\widehat \Xi_j$, being $\widehat \Xi_j$ the identity of in one of the forms $\mathfrak{p}_{i\alpha_j}$,  or $\mathfrak{R}_{\vartheta_j}$, then
            \begin{equation}
                \mu_j=\begin{cases}
                    \alpha_j & \text{if $\widehat \Xi_j=\mathfrak{p}_{i\alpha_j}$},\\
                    \vartheta_j & \text{if $\widehat \Xi_j=\mathfrak{R}_{\vartheta_j}$},\\
                    0 & \text{otherwise}.
                \end{cases}
            \end{equation}
            Since $\widehat\Xi$ is not a multiple of the identity and $P+Q=O_d$ if and only if $\mu_j=0$ for every $j=1,\ldots,d$, we see that $\mathfrak{p}_{i(Q+P)}$ satisfies the assumptions of the first part of the Lemma \ref{lemmaGaussians}. Consequently, 
            
            \begin{equation}
                \norm{\widehat\Xi \widehat Vf}_2<\norm{\mathfrak{T}_{E}\cF^{-1}\mathfrak{p}_{iP}\cF\widehat Vf}_2\leq\norm{f}_2,
            \end{equation}
            and we are done.
        \end{proof}

        As detailed in the following sections, the transition from metaplectic Wigner distributions $W_\cU$ with $\widehat\cU \in \mathrm{Mp}(2d,\mathbb{R})$ to those $W_\cA$ with $\widehat\cA \in \mathrm{Mp}_+(2d,\mathbb{C})$ provides distributions that satisfy several properties which, in the real case, would only hold for a very restricted, and basically trivial, subclass of metaplectic Wigner distributions. 
        Continuity properties of metaplectic Wigner distributions follow by those of the operators in $\Mp_+(d,\bC)$. However, this comes at the cost of losing the isometry property on $L^2$ and Moyal's identity, as the next two results clarify.
        
		\begin{proposition}\label{prop64}
			Let $W_\cA$ be a a metaplectic Wigner distribution with $\cA\in\Sp_+(2d,\bC)$.
			\begin{enumerate}[(i)]
				\item $W_\cA:\cS(\rd)\times\cS(\rd)\to\cS(\rdd)$ is continuous.
				\item $W_\cA:L^2(\rd)\times L^2(\rd)\to L^2(\rdd)$ is bounded with
				\begin{equation}\label{L2WA}
					\norm{W_\cA(f,g)}_2\leq \norm{f}_2\norm{g}_2, \qquad f,g\in L^2(\rd),
				\end{equation}
				and equality holds for every $f,g\in L^2(\rd)$ if 
                and only if 
                $\widehat\cA\in\Mp(2d,\bR)$.
				\item $W_\cA:\cS'(\rd)\times\cS'(\rd)\to\cS'(\rdd)$ is continuous.
			\end{enumerate}
		\end{proposition}
        \begin{proof} Items $(i)$ and $(iii)$, along with the $L^2$ boundedness in $(ii)$, follow simply by the continuity properties of operators in $\Mp_+(2d,\bC)$. The only item that remains to verify is the ``only if" part of $(ii)$. 
            To prove it, we check that if $W_\cA$ has $\widehat\cA\in\Mp_+(d,\bC)\setminus\Mp(d,\bR)$, then there exist functions $f,g\in L^2(\rd)$ so that \eqref{L2WA} holds with a strict inequality. 
            By Theorem \ref{polar}, we may decompose $\widehat\cA=\widehat\cU\widehat\cZ$ with $\widehat\cU\in\Mp(2d,\bR)$ and $\widehat\cZ\in\Mp_0(2d,\bC)$, and $\widehat\cZ$ is not a multiple of the identity. The left-hand side of \eqref{L2WA} reads as
            \begin{align}
                \norm{W_\cA(f,g)}_2&=\norm{\widehat\cU\widehat\cZ(f\otimes \overline g)}_2=\norm{\widehat\cZ(f\otimes\overline g)}_2.
            \end{align}
            It follows by Lemma \ref{lemmaGaussians} that $\norm{\widehat\cZ(f\otimes\overline g)}_2<\norm{f\otimes\overline g}_2=\norm{f}_2\norm{g}_2$ if $f$ and $g$ are Gaussians. This concludes the proof.
            
        \end{proof}


        \begin{corollary}
            Let $W_\cA$ be a metaplectic Wigner distribution. Moyal's identity \eqref{Moyal} holds if and only if $\widehat\cA\in\Mp(2d,\bR)$.
        \end{corollary}
        \begin{proof}
            The ``if" part is trivial. The ``only if" part follows then by Proposition \ref{prop64}, since Moyal's identity implies the equality in \eqref{L2WA}, which holds if and only if $\cA\in\Sp(2d,\bR)$.
            
        \end{proof}

		\subsection{Covariance and Cohen class}
		
		Let us start by recalling the definition of covariant time-frequency representation, stated hereafter in the framework of metaplectic Wigner distribution.
		
		\begin{definition}
			We say that a metaplectic Wigner distribution $W_\cA$ is {\em covariant} if for every $z\in\rdd$ and every $f,g\in L^2(\rd)$,
			\begin{equation}\label{covariance2}
				W_\cA(\pi(z)f,\pi(z)g)(w)=W_\cA(f,g)(w-z), \qquad w\in\rdd.
			\end{equation}
		\end{definition}

		Recall that the {\em Cohen class} consists of those time-frequency representations $Q(f,g)$ that can be written as
		\begin{equation}\label{CohenClass}
			Q(f,g)=k\ast W(f,g), \qquad f,g\in L^2(\rd),
		\end{equation}
		for a suitable kernel $k\in\cS'(\rdd)$. Because of the convolution nature of the Cohen class, time-frequency representations satisfying \eqref{CohenClass} are always covariant, see e.g., \cite[Proposition 1.3.19 (i)]{CorderoBook}. On the top of that, it turns out that in the case of metaplectic Wigner distributions the reverse relation also holds:
		\begin{equation}
			\text{$W_\cA$ covariant if and only if $W_\cA$ is in the Cohen class},
		\end{equation}
		that is the content of Theorem \ref{theo.maintimefreq}, that we now prove.
		
		\begin{theorem}\label{theo.maintimefreq}
			Let $W_\cA$ be a metaplectic Wigner distribution with projection 
			\begin{equation}\label{blockA}
				\cA=\pi^{Mp}_+(\widehat\cA)=\begin{pmatrix}
					A_{11} & A_{12} & A_{13} & A_{14}\\
					A_{21} & A_{22} & A_{23} & A_{24}\\
					A_{31} & A_{32} & A_{33} & A_{34}\\
					A_{41} & A_{42} & A_{43} & A_{44}
				\end{pmatrix}, \qquad A_{ij}\in\bC^{d\times d}, \; i,j=1,\ldots,4,
			\end{equation}
			in $\Sp_+(2d,\bC)$. Let
			\begin{equation}\label{defBA}
				B_\cA=\begin{pmatrix}
					A_{13} & \dfrac{1}{2}I_d-A_{11}\\
					\dfrac{1}{2}I_d-A_{11}^\top & -A_{21}
				\end{pmatrix}.
			\end{equation}
			The following statements are equivalent:
			\begin{enumerate}[(i)]
			\item $W_\cA$ is covariant.
			\item It holds
			\begin{equation}\label{charAcov}
				\cA=\begin{pmatrix}
					A_{11} & I_d-A_{11} & A_{13} & A_{13}\\
					A_{21} & -A_{21} & I_d-A_{11}^\top & -A_{11}^\top\\
					O_d & O_d & I_d & I_d\\
					-I_d & I_d & O_d & O_d
				\end{pmatrix}, \qquad A_{11},A_{13},A_{21}\in\bC^{d\times d},
			\end{equation}
			$A_{13},A_{21}\in\Sym(d,\bC)$ and $\Im(B_\cA)\leq0$.
			\item $W_\cA$ is the Cohen class, and
			\begin{equation}\label{usefulspectro1}
				W_\cA(f,g)=k_\cA\ast W(f,g)
			\end{equation}
			with $k_\cA=\cF^{-1}\Phi_{-B_\cA}$.
			\end{enumerate}
		\end{theorem} 
		\begin{proof}
		To prove that $(ii)\Leftrightarrow(iii)$, observe that the matrix in \eqref{charAcov} satisfies
			\begin{equation}
				\cA=V_{B_\cA}^\top\cA_{1/2},
			\end{equation}
			where $V_{B_\cA}$ is defined in \eqref{defVQ} and $\cA_{1/2}$ is defined as in \eqref{projWig}, 
			and it is the projection of the Wigner distribution, i.e. $W_{\cA_{1/2}}(f,g)=W(f,g)$. Hence, by Remark \ref{remConvolution},
			\begin{equation}
				W_\cA(f,g)=\widehat\cA(f\otimes\overline g)=\cF^{-1}(e^{-i\pi B_\cA z\cdot z})\ast \widehat{\cA_{1/2}}(f\otimes\overline g)=\cF^{-1}(e^{-i\pi B_\cA z\cdot z})\ast W(f,g),
			\end{equation}
			up to a constant. The implication $(iii)\Rightarrow(i)$ follows by the fact that every quadratic time-frequency representation in the Cohen class is covariant. It remains to prove that $(i)\Rightarrow(ii)$. Assume that $(i)$ holds. By Theorem \ref{thmGG1} applied to $f(y)=e^{-\pi|y|^2}$ and $(x,\xi)\in\rdd$, we have
			\begin{equation}
				W_\cA(\pi(x,\xi)f)=W_\cA(\rho(x,\xi;0)f)=\widehat\cA(\rho(x,x,\xi,-\xi;0)(f\otimes\overline f))=\rho(\cA(x,x,\xi,-\xi);0)W_\cA f.
			\end{equation}
			Since $W_\cA$ is covariant, $W_\cA(\pi(x,\xi)f)=T_{(x,\xi)}W_\cA f$, and therefore it must be
			\begin{equation}\label{transl}
				\cA(x,x,\xi,-\xi)=(x,\xi,0,0), \qquad (x,\xi)\in\rdd.
			\end{equation}
			By writing the blocks \eqref{blockA} of $\cA$ explicitly and using \eqref{symp-rel1}--\eqref{symp-rel3}, equation \eqref{transl} is equivalent to $\cA$ having blocks \eqref{charAcov} with $A_{13},A_{21}\in\Sym(d,\bC)$. It remains to impose the condition for $\cA\in\Sp_+(2d,\bC)$. By writing $\cA=\Re(\cA)+i\Im(\cA)=\cA_R+i\cA_I$, it is straightforward to see that $\cA_I^\top J\cA_I=O_{4d}$, so the positive semi-definiteness condition \eqref{defMS} reads for $\cA$ as
			\begin{equation}
				\mathcal{M}_\cA=\begin{pmatrix}
					\cA_R^\top J\cA_I & O_{4d}\\
					O_{4d} & \cA_R^{\top}J\cA_I 
				\end{pmatrix}\geq0,
			\end{equation}
			i.e., $\cA_R^\top J\cA_I\geq0$. By writing $\cA_R$ and $\cA_I$ in terms of the blocks of $\cA$, we can express
			\begin{equation}
				\cA_R^\top J\cA_I=\Im\left(\begin{array}{cc|cc}
				A_{21} & -A_{21} & -A_{11}^\top & -A_{11}^\top\\
				-A_{21} & A_{21} & A_{11}^\top & A_{11}^\top\\
				\hline
				-A_{11} & A_{11} & -A_{13} & -A_{13}\\
				-A_{11} & A_{11} & -A_{13} & -A_{13}
				\end{array}\right),
			\end{equation}
			If
			\begin{equation}\label{SigmaPerRedRemark}
				\Sigma=\begin{pmatrix}
					I_d & I_d & O_d & O_d\\
					O_d & O_d & I_d & I_d\\
					-I_d & I_d & O_d & O_d\\
					O_d & O_d & -I_d & I_d
				\end{pmatrix},
			\end{equation}
			then,
			\begin{equation}
				\cA_R^\top J\cA_I=\Sigma^\top \Im\begin{pmatrix}
					O_d & O_d & O_d & O_d\\
					O_d & -4A_{13} & 4A_{11} & O_d\\
					O_d & 4A_{11}^\top & 4A_{21} & O_d\\
					O_d & O_d & O_d & O_d
				\end{pmatrix}\Sigma=\Sigma^\top 
				\left(\begin{array}{c|c|c}
					O_d & O_{d\times 2d} & O_d\\
					\hline
					O_{2d\times d} & -4\Im (B_\cA) & O_{2d\times d}\\
					\hline
					O_d & O_{d\times2d}& O_d
				\end{array}\right)\Sigma.
			\end{equation}
			Consequently, 
			\begin{equation}
				\begin{array}{ccccc}
			\cA\in\Sp_+(2d,\bC) & \Longleftrightarrow & \cA_R^{\top}J\cA_I\geq0 &\Longleftrightarrow&\Im (B_\cA)\leq0.
			\end{array}
			\end{equation}
			Thus, $(i)\Rightarrow(ii)$, and the theorem is proven.
			
		\end{proof}
        
        
		\subsection{Generalized spectrograms}
		Our next goal is to characterize metaplectic distributions that are {\em generalized spectrograms}, in terms of the block decomposition \eqref{blockA} of $\cA$. Recall that a generalized spectrogram is a time-frequency representation in the form
		\begin{equation}
			\mathrm{Spec}_{\phi,\psi}(f,g)=V_\phi f\overline{V_\psi g}, \qquad f,g\in\cS'(\rd),
		\end{equation}
		where $\phi,\psi\in\cS'(\rd)$ are fixed distributional windows. If $f=g$ and $\phi=\psi$, we write $\mathrm{Spec}_\phi f=\mathrm{Spec}_{\phi,\phi} (f,f)$, the {\em spectrogram of $f$ with window $\phi$}. Observe that spectrograms are always non-negative real-valued time-frequency representations, since
		\begin{equation}
			\mathrm{Spec}_{\phi}f=|V_\phi f|^2.
		\end{equation}
		An important instance of a spectrogram is the so-called {\em Husimi distribution} \cite{Husimi}, given by 
		\begin{equation}
			Hf=\mathrm{Spec}_\phi f=|V_\phi f|^2,
		\end{equation} 
		with $\phi(x)=c\cdot 2^{d/4}e^{-\pi|x|^2}$, $|c|=1$. Moreover, we recall that generalized spectrograms are always in the Cohen class, as it is evident by the formula
		\begin{equation}\label{usefulspectro2}
		\mathrm{Spec}_{\phi,\psi}(f,g)=W(\cI \phi,\cI \psi) \ast W(f,g), 
		\end{equation}
		where $\cI F(x)=F(-x)$ is the flip operator, see \cite{BDO1}. In particular, generalized spectrograms are covariant. 
		
		The following result adapts \cite[Corollary 4.7]{CG03} to the complex case. By allowing complex matrices, we are finally able to include time-frequency representations, such as the Husimi distribution, in the framework of metaplectic Wigner distributions.
        
		\begin{theorem}\label{theo.spectrogram}
		Let $W_\cA$ be a metaplectic Wigner distribution, with $\cA \in \Sp_+(2d,\mathbb{C})$ having block as in \eqref{blockA} and assume $A_{13}\in\GL(d,\bC)$. Then the following statements are equivalent. 
		\begin{enumerate}[(i)]
		\item There exist $\psi,\phi \in \mathcal{S}'(\R^d)$ such that
		\begin{equation}\label{spectrogramcA}
		W_\cA(f,g)=\mathrm{Spec}_{\phi,\psi}(f,g), \quad f,g \in \mathcal{S}'(\R^d),
		\end{equation} 
		\item The matrix $\cA$ has blocks as in \eqref{charAcov}, with $A_{13},A_{21}\in\Sym(d,\bC)$, $\Im(A_{13})\leq 0$. Additionally
		\begin{align}\label{DTeq1}
		& A_{21}+A_{11}^\top A_{13}^{-1}(A_{11}-I_d)=O_d \\
		\label{DTeq2}
		&\Im(A_{11}^\top A_{13}^{-1}) \geq 0 \\
		\label{DTeq3}
		&\Im(A_{13}^{-1}(A_{11}-I_d))\leq 0, 
		\end{align}
		and 
\begin{equation}\label{finestregenspect}
\begin{split}
\phi (x)&=c_1 e^{-i\pi A_{13}^{-1}(A_{11}-I_d)x\cdot x},\\
\psi (y)&=c_2 e^{-i \pi \overline{A_{11}^\top A_{13}^{-1}} y \cdot y},
\end{split}
\end{equation}
	for constants $c_1,c_2$ satisfying $c_1c_2=i^{-d/2}\det(A_{13})^{-1/2}$.
	\end{enumerate}
\end{theorem}		
\begin{proof}
Since generalized spectrograms are covariant, we may assume a priori that $W_\cA$ is covariant, with $\cA$ as in \eqref{charAcov}. By comparing \eqref{usefulspectro1} and \eqref{usefulspectro2}, we have that \eqref{spectrogramcA} holds if and only if
\begin{equation*}
W(\cI \phi, \cI \psi)(\zeta)=\mathcal{F}^{-1}(e^{-i\pi B_\cA z \cdot z})(\zeta), \quad \zeta \in \R^{2d},
\end{equation*} 
which in turn holds if and only if (with $\cI_2F(x,y)=F(x,-y)$)
\begin{equation*}
 \cI\phi(x) \otimes \overline{\cI\psi(y)}=\mathfrak{T}_W^{-1}\cI_2\mathcal{F}_1^{-1}(e^{-i\pi B_\cA z \cdot z})(x,y).
\end{equation*}
By writing explicitly the right-hand side, we can write 
\begin{equation*}
\begin{split}
\phi(-x)\overline{\psi(-y)}&=\int e^{2\pi i u ((x+y)/2)}e^{-i\pi B_\cA (u, y-x) \cdot (u, y-x)}du \\
&=e^{i\pi A_{21}(y-x)\cdot (y-x)}\int_{\R^d}e^{-2\pi i u\bigl((\frac{1}{2}I_d-A_{11})(y-x)-\frac{x+y}{2}\bigr)}e^{-i\pi A_{13} u \cdot u} du \\
&=e^{i\pi A_{21}(y-x)(y-x)} \int_{\R^d}e^{2\pi i \bigl((A_{11}-I_d)x-A_{11}y\bigr)u}e^{-i\pi A_{13} u \cdot u} du. 
\end{split}
\end{equation*}
Hence, under the hypotheses $A_{13}\in\GL(d,\bC)$ and $\Im(A_{13})\leq 0$, by Proposition \ref{GaussianIntegrals}, we have that \eqref{spectrogramcA} is equivalent to 
\begin{equation}\label{intStepGG1}
\begin{split}
\phi(-x)\overline{\psi(-y)}&= (\mathrm{det}(iA_{13}))^{-1/2} e^{i \pi A_{21}y \cdot y}e^{i \pi A_{21} x \cdot x}e^{-2\pi i A_{21} y \cdot x}e^{i\pi (A_{11}^\top-I_d)A_{13}^{-1}(A_{11}-I_d)x \cdot x } \times \\
& \times e^{i \pi A_{11}^\top A_{13}^{-1}A_{11}y \cdot y}e^{-2\pi i A_{11}^\top A_{13}^{-1}(A_{11}-I_d)x\cdot y}.
\end{split}
\end{equation}
The right-hand side of the previous equation is a tensor product if and only if \eqref{DTeq1} holds,
which implies that $A_{11}^\top A_{13}^{-1}$ is symmetric. Together with \eqref{DTeq2} and \eqref{DTeq3}, which hold if and only if     
\begin{equation}
\begin{split}
\phi (x)&=c_1 \cdot e^{-i\pi A_{13}^{-1}(A_{11}-I_d)x\cdot x},\\
\psi (y)&=c_2 \cdot e^{-i \pi \overline{A_{13}^{-1}A_{11}} y \cdot y},
\end{split}
\end{equation}
belong to $\mathcal{S}'(\R^d)$, for constants $c_1$, $c_2$ satisfying $c_1c_2=\det(iA_{13})^{-1/2}=i^{d/2}\det(A_{13})^{-1/2}$, this concludes the proof. 
\end{proof}

	As far as spectrograms are concerned, we can state the following result, which follows by \eqref{intStepGG1}, by choosing $\phi=\psi$.
	
	\begin{corollary}\label{corHusimi}
		Let $W_\cA$ be a metaplectic Wigner distribution, with $\cA \in \Sp_+(2d,\mathbb{C})$ having blocks as in \eqref{blockA} and assume $A_{13}\in\GL(d,\bC)$. The following statements are equivalent.
		\begin{enumerate}[(i)]
			\item $W_\cA$ is a spectrogram, in the sense that
			\begin{equation}\label{nonPolSpect}
				W_\cA f=|V_\phi f|^2, \qquad f\in\cS'(\rd),
			\end{equation}
			for some $\phi\in\cS'(\rd)$.
			\item The matrix $\cA$ has blocks as in \eqref{charAcov}, with $A_{13},A_{21}\in\Sym(d,\bC)$, $\Im(A_{13})<0$ satisfying
			\begin{align}
				\label{GGeq1-1}
				& \Re(A_{11})=\frac{1}{2}I_d,\\
				\label{GGeq1-2}
				& \Re(A_{13})=O_d,\\
				\label{GGeq1-3}
				& A_{21}=\frac{1}{4}A_{13}^{-1} +\Im(A_{11})^\top A_{13}^{-1}\Im(A_{11}).
			\end{align}
		\end{enumerate}
		In this case,
		\begin{equation}\label{finestraSpect}
			\phi(x)=\det(iA_{13})^{-1/4}e^{-i\pi A_{13}^{-1}(A_{11}-I_d)x\cdot x}.
		\end{equation}
	\end{corollary}
	The reader may observe that conditions \eqref{GGeq1-2} and \eqref{GGeq1-3} together give that also $\Re(A_{21})=O_d$ for metaplectic Wigner distribution that are spectrograms.
	\begin{proof}[Proof of Corollary \ref{corHusimi}]
		Again, we may restrict to covariant metaplectic Wigner distributions, for generalized spectrograms are covariant. Let us consider the polarized version of \eqref{nonPolSpect}:
		\begin{equation}\label{startCoroGG2}
			W(\cI\phi) \ast W(f,g)=\cF^{-1}\Phi_{-B_\cA}\ast W(f,g).
		\end{equation}
		Since the span of $\{W(f,g):f,g\in\cS'(\rd)\}$ is dense in $\cS'(\rdd)$, the correspondence \eqref{startCoroGG2} holds if and only if $W(\cI\phi)=\cF^{-1}\Phi_{-B_\cA}$. So, by repeating the computations in the proof of Theorem \ref{theo.spectrogram}, we obtain
		\begin{equation}
			\begin{split}
\phi(-x)\overline{\phi(-y)}&= (\mathrm{det}(iA_{13}))^{-1/2} e^{i \pi A_{21}y \cdot y}e^{i \pi A_{21} x \cdot x}e^{-2\pi i A_{21} y \cdot x}e^{i\pi (A_{11}^\top -I_d)A_{13}^{-1}(A_{11}-I_d)x \cdot x } \times \\
& \times e^{i \pi A_{11}^\top A_{13}^{-1}A_{11}y \cdot y}e^{-2\pi i A_{11}^\top A_{13}^{-1}(A_{11}-I_d)x\cdot y}.
\end{split}
		\end{equation}
		We thereby retrieve the same conditions \eqref{DTeq1}, \eqref{DTeq2} and \eqref{DTeq3}. By imposing $\phi=\psi$, we obtain
		\begin{equation}
			A_{13}^{-1}(A_{11}-I_d)=\overline{A_{13}^{-1}A_{11}},
		\end{equation}
		giving
		\begin{equation}\label{GGintStep3}
			A_{13}^{-1}=A_{13}^{-1}A_{11}-\overline{A_{13}^{-1}A_{11}}=2i\Im(A_{13}^{-1}A_{11}).
		\end{equation}
		In particular, it must by $A_{13}=iP$ for some $P<0$, in view of the conditions $A_{13}\in\GL(d,\bC)$ and $\Im(A_{13})\leq 0$, entailing condition \eqref{GGeq1-2}. By plugging this expression in \eqref{GGintStep3} and by writing $A_{11}=E+iF$, $E,F\in\bR^{d\times d}$, we have
		\begin{equation}
			-iP^{-1}=2i\Im(-iP^{-1}(E+iF))=2i\Im(P^{-1}F-iP^{-1}E)=2i(-iP^{-1}E)=-2iP^{-1}E,
		\end{equation}
		which gives $2E=I_d$, whence \eqref{GGeq1-1}. Finally, by \eqref{DTeq1}, we have
		\begin{align}
			A_{21}&=-A_{11}^\top A_{13}^{-1}(A_{11}-I_d)\\
			&=-\left(\frac{1}{2}I_d +iF^\top\right)(-iP^{-1})\left(-\frac{1}{2}I_d +iF\right)\\
			&=\frac 1 4 (-iP^{-1})+\frac{P^{-1}F-F^\top P^{-1}}{2}+F^\top(-iP^{-1})F\\
			&=\frac 1 4 A_{13}^{-1}+\Im(A_{11})^\top A_{13}^{-1}\Im(A_{11}),
		\end{align}
		where we used that $P^{-1}F-F^\top P^{-1}=O_d$, since $A_{11}^\top A_{13}^{-1}=F^\top P^{-1}-iP^{-1}/2$ is symmetric. This proves \eqref{GGeq1-3}. The form \eqref{finestraSpect} of the window $\phi$ is then a restatement of \eqref{finestregenspect} under the further condition $\phi=\psi$, and we are done.
	\end{proof}
	
	\begin{example}
	Let us express the Husimi distribution as a metaplectic Wigner distribution. Let $\cA \in \Sp_+(2d,\mathbb{C})$ satisfying the conditions of Corollary \ref{corHusimi}, with 
		\begin{equation}
		A_{11}=\frac{1}{2}I_d, \qquad A_{13}=-\frac{i}{2} I_d,
		\end{equation}
		and $A_{21}$ defined through \eqref{GGeq1-3}. In this case $B_\cA=-i I_{2d}/2$ and then, as for the $\cA$-Wigner distribution associated, we have
		\begin{equation*}
		W_\cA f= \cF^{-1}(e^{-\pi z^2/2})\ast W f=|V_\phi f|^2, \quad f \in \mathcal{S}'(\R^d)
		\end{equation*}
		with $\phi(x)=c\cdot e^{-\pi |x|^2}$. By choosing $c=(2i)^{d/4}$, we retrieve the Husimi distribution.
	\end{example}
		
		\subsection{Metaplectic Wigner distributions and complex conjugation}\label{subsec:63}
			Alongside covariance, a fundamental property of the classical Wigner distribution, having implications in the theory of Wigner microlocal analysis, is its behavior with respect to complex conjugation. Indeed, for $f,g\in L^2(\rd)$
			\begin{equation}\label{Wignerconj}
				W(f,g)=\overline{W(g,f)}.
			\end{equation}
			Among other consequences, \eqref{Wignerconj} implies that $Wf$ is a real-valued continuous function for every $f\in L^2(\rd)$. Here, we use the results of Section \ref{subsec:complexConjugation} to characterize those metaplectic Wigner distributions $W_\cA$ with the property that
			\begin{equation}\label{AWignerconj}
				W_\cA(f,g)=\overline{W_\cA(g,f)}, \qquad f,g\in L^2(\rd).
			\end{equation}
			The real case was studied in an unpublished contribution \footnote{See G. Giacchi, G., and L. T. Minh. The symplectic Wigner distribution associated to quadratic Fourier transforms: Heisenberg's uncertainty inequality and applications. {\em 10.36227/techrxiv.176472679.93143839/v1}, 2025.}, where it was proven that \eqref{AWignerconj} holds only for the classical Wigner distributions, up to rescalings. This suggests that the classical Wigner distribution is, among metaplectic Wigner distributions with $\widehat\cU\in\Mp(2d,\bR)$, the only quantization tailored to study interaction between time-evolving signals, as it is shown in a recent, unpublished manuscript by two of the authors \footnote{\label{footoneboh} See G. Giacchi and D. Tramontana. On the microlocal phase-space concentration of Wigner distributions associated with Schr\"odinger evolutions. {\em arXiv preprint arXiv:2511.19733v2}, 2025.}. In the complex framework, the situation is more involved. Actually, this complication paves the way to many new quantization rules satisfying covariance and \eqref{AWignerconj} within the metaplectic framework.
			
			\begin{theorem}\label{thmImportant}
				Let $W_\cA$ be a metaplectic Wigner distribution, let $\cA$ be the projection of $\widehat\cA$ with blocks as in \eqref{blockA}. The following statements are equivalent.
				\begin{enumerate}[(i)]
					\item $W_\cA$ satisfies \eqref{AWignerconj}.
					\item $\cA\in\Sp_+(2d,\bC)$ and
					\begin{equation}\label{charAconj}
						\cA=\begin{pmatrix}
							A_{11} & \overline{A_{11}} & A_{13} & -\overline{A_{13}}\\
							A_{21} & \overline{A_{21}} & A_{23} & -\overline{A_{23}}\\
							A_{31} & -\overline{A_{31}} & A_{33} & \overline{A_{33}}\\
							A_{41} & -\overline{A_{41}} & A_{43} & \overline{A_{43}}
						\end{pmatrix}.
					\end{equation}
					\item $W_\cA(f,g)=\widehat\cB W(f,g)$, where
					\begin{equation}
						\cB=\begin{pmatrix}
							2\Re(E_\cA) & i\Im(E_\cA)J\\
							2i\Im(F_\cA) & \Re(F_\cA)J
						\end{pmatrix},
					\end{equation}
					where $J$ is the standard symplectic matrix, whereas
					\begin{equation}
						E_\cA=\begin{pmatrix}
							A_{11} & A_{13}\\
							A_{21} & A_{23}
						\end{pmatrix} \qquad \text{and} \qquad 
						F_\cA=\begin{pmatrix}
							A_{31} & A_{33}\\
							A_{41} & A_{43}
						\end{pmatrix},
					\end{equation}
                    with $\Re(E_\cA)\in\GL(2d,\bR)$, $\Re(E_\cA)^\top\Im(F_\cA)\geq0$ and $\Re(E_\cA)J\Im(E_\cA)^\top\geq0$.
				\end{enumerate}
                Consequently, if any of the previous conditions holds, then $W_\cA f$ is a real-valued function for every $f\in L^2(\rd)$.
			\end{theorem}
			\begin{proof}
                Let $W_\cA$ be a metaplectic Wigner distribution and $\cA$ have blocks \eqref{blockA}. We have
                \begin{align}
                   \widehat\cA(f\otimes\overline g)=W_\cA(f,g)=\overline{W_\cA(g,f)}=\overline{\widehat\cA(g\otimes\overline f)}=\widehat{\widetilde\cA}(\overline g\otimes f)=\widehat{\widetilde\cA}\mathfrak{T}_{L}(f\otimes\overline g),
                \end{align}
                where
                \begin{equation}
                    L=\begin{pmatrix}
                        O_d & I_d\\
                        I_d & O_d
                    \end{pmatrix}.
                \end{equation}
                By linearity and the density of $\mbox{span}\{f\otimes\overline g:f,g\in L^2(\rd)\}$, it follows that \eqref{AWignerconj} holds if and only if $\widehat\cA=\widehat{\widetilde\cA}\mathfrak{T}_L$. From the perspective of symplectic projections, this holds if and only if $\cA=\widetilde\cA\cD_L$. By writing the blocks of $\cA$, $\widetilde\cA$ and $\cD_L$ explicitly, the equivalence $(i)\Leftrightarrow(ii)$ follows. Let $\widehat\cB:=\widehat\cA\widehat{\cA}_{1/2}^{-1}$, where $\widehat\cA_{1/2}$ is the metaplectic operator defining the Wigner distribution, i.e., $W(f,g)=\widehat\cA_{1/2}(f\otimes\overline g)$, whose projection $\cA_{1/2}$ is given by \eqref{projWig}. The equivalence $(ii)\Longleftrightarrow(iii)$ is then a simple restatement of $\widehat\cA=\widehat\cB\widehat\cA_{1/2}$. Since $\cA_{1/2}\in\Sp(2d,\bR)$, $\cA\in\Sp_+(2d,\bC)$ if and only if $\cB\in\Sp_+(2d,\bC)$. Moreover, $\cB$ is a matrix in the form \eqref{blockSconj}. It follows by Theorem \ref{StructuralTheorem} that $\cB\in\Sp_+(2d,\bC)$ if and only if $\Re(E_\cA)\in\GL(2d,\bR)$, $\Re(E_\cA)^\top\Im(F_\cA)\geq0$ and $\Re(E_\cA)(\Im(E_\cA)J)^\top\leq0$. Since $J^\top=-J$, the latter condition is equivalent to $\Re(E_\cA)J\Im(E_\cA)^\top\geq0$, and we are done.
                
			\end{proof}
            
            \begin{remark}
                It follows by Theorem \ref{thmImportant} that $W_\cA(f,g)$ satisfies \eqref{AWignerconj} if and only if $W_\cA(f,g)=\widehat\cB W(f,g)$ for some $\widehat\cB\in\Mp_+(2d,\bC)$ satisfying $\overline{\widehat\cB f}=\widehat\cB \bar f$ for every $f\in L^2(\rdd)$. The ``if" part is trivial, whereas the ``only if" part is the content of Theorem \ref{thmImportant}. Corollary \ref{corSemigroup} is then a recipe to construct metaplectic Wigner distributions satisfying \eqref{AWignerconj}. 
            \end{remark}

            A direct comparison between \eqref{charAcov} and \eqref{charAconj} yields the characterization of metaplectic Wigner distributions satisfying both \eqref{covariance2} and \eqref{AWignerconj}. 

            \begin{corollary}
                Let $W_\cA$ be a metaplectic Wigner distribution with $\cA$. Then, the following statements are equivalent.
                \begin{enumerate}[(i)]
                    \item $W_\cA$ is covariant and satisfies \eqref{AWignerconj}.
                    \item The projection $\cA$ of $\widehat\cA$ has blocks
                    \begin{equation}\label{matrixA}
                       \cA=\begin{pmatrix}
                           I_d/2 & I_d/2 & A_{13} & A_{13}\\
                           A_{21} & -A_{21} & I_d/2 & -I_d/2\\
                           O_d & O_d & I_d & I_d\\
                           -I_d & I_d & O_d & O_d
                       \end{pmatrix},
                    \end{equation}
                    with $\Re(A_{13}),\Re(A_{21})=O_d$, $\Im(A_{13})\leq0$ and $\Im(A_{21})\geq0$.
                    \item Up to a constant, $W_\cA(f,g)=\cF^{-1}(\Phi_{-B_\cA})\ast W(f,g)$. Here, if $\cA$ is as in \eqref{matrixA}, then 
                    $B_\cA=\Im\big(\diag(A_{13},-A_{21})\big)\leq0$, $\Re(B_\cA)=O_{2d}$.
                \end{enumerate}
            \end{corollary}

            \section{Applications to evolution equations with complex quadratic Hamiltonians}\label{sec.evo}
		H\"ormander's theory has been used to investigate
the propagation of concentration for the Wigner distribution associated with solutions of evolution equations of the form
		\begin{equation}\label{Schreq}
		\begin{cases}
		\frac{1}{2\pi}\partial_t u+\mathrm{Op}^\mathrm{w}(a)u=0 \quad \text{in} \ \ \R^+ \times \R^d \\
		u(0,\cdot)=u_0 \in L^2(\R^d),
		\end{cases}
		\end{equation}
		with $a(z)=z \cdot \cQ z$, $z \in \R^{2d}$  \textit{complex} quadratic form associated with a matrix $\cQ \in \mathbb{C}^{2d\times 2d}$ satisfying $\Re(\cQ) \geq 0$.
        Let us point out that the imaginary unit $i$ does not appear in the formulation \eqref{Schreq}, emphasizing that the equation is of mixed heat/Schr\"odinger type rather than a purely Schr\"odinger equation.

        We shall fix attention on the metaplectic operators in $\Mp_+(d,\bC)$ representing the propagator $e^{-2\pi t\Op^\w(a)}$ of \eqref{Schreq}, emphasizing applications to some relevant equations.
        
        \subsection{The Weyl perspective and modulation spaces}
            Consider $\widehat Z\in\Mp_0(d,\bC)$, and write $\widehat Z=\widehat V\widehat\Xi\widehat V^{-1}$, where $\widehat V\in\Mp(d,\bR)$ and $\widehat{\Xi}$ is an atomic metaplectic contraction.

            \begin{theorem}\label{propSymbS000}
                Under the notation above, with $\widehat\Xi=\bigotimes_{j=1}^d\widehat\Xi_j$ as in Remark \ref{remAMC}, 
                let $\Sigma$ be the $2d\times2d$ diagonal matrix with diagonal entries
                \begin{equation}
                    \Sigma_{j,j}=\begin{cases}
                        \alpha_j & \text{if $\widehat\Xi_j=\frp_{i\alpha_j}$,}\\
                        0 & \text{if $\widehat\Xi_j=\id_{L^2}$,}\\
                        2\tanh(\vartheta_j/2) & \text{if $\widehat\Xi_j=\frR_{\vartheta_j}$}.
                    \end{cases}, \qquad \Sigma_{j+d,j+d}=\begin{cases}
                        0 & \text{if $\widehat\Xi_j=\frp_{i\alpha_j}$,}\\
                        0 & \text{if $\widehat\Xi_j=\id_{L^2}$,}\\
                        2\tanh(\vartheta_j/2) & \text{if $\widehat\Xi_j=\frR_{\vartheta_j}$},
                    \end{cases}
                \end{equation}
                $j=1,\ldots,d$.
                Then, $\widehat Z=\Op^\w(a)$, where \begin{equation}\label{symbMp0}
                    a(z)=\prod_{j=1}^d \cosh(\vartheta_j/2)^{-1}e^{-\pi V^{-\top} \Sigma V^{-1}z\cdot z}, \qquad z\in\rdd,
                \end{equation} 
                where it is implied $\vartheta_j=0$ if $\widehat\Xi_j$ is not in the form $\mathfrak{R}_{\vartheta_j}$ for $\vartheta_j>0$. Consequently, if $\widehat Z\in \Mp_0(d,\bC)$, then $\widehat Z:M^{p,q}_{v_s}(\rd)\to M^{p,q}_{v_s}(\rd)$ is bounded for every $1\leq p,q\leq\infty$ and $s\geq0$, with
                \begin{equation}\label{Norm1}
                    \norm{\widehat Z}_{op}\lesssim c(s)\prod_{j=1}^d\cosh(\vartheta_j/2)^{-1}(1+\sigma_{\mathrm{max}}(\Sigma^{1/2}V^{-1})^2)^{s/2},
                \end{equation}
                where $\sigma_{\mathrm{max}}(\Sigma^{1/2}V^{-1})$ is the largest singular value of $\Sigma^{1/2}V^{-1}$ and $c(s)=\int_{\rdd}e^{-\pi|x|^2}v_s(x)dx$.
            \end{theorem}
            \begin{proof}
                Let us write $\widehat{Z}=\widehat{V}\widehat\Xi\widehat{V}^{-1}$, as in Corollary \ref{CorPolar}, with 
                $\widehat\Xi=\bigotimes_{j=1}^d\widehat\Xi_j$, as in Remark \ref{remAMC}. Let $a\in\cS'(\rdd)$ be the Weyl symbol of $\widehat Z$, then by \eqref{intertOpwMetap}, 
                \begin{equation}\label{symbTens1}
                    a(z)=\left(\bigotimes_{j=1}^da_j\right)(V^{-1}z),
                \end{equation}
                where, for $j=1,\ldots,d$, $a_j\in\cS'(\bR^2)$ is the Weyl symbol of $\widehat{\Xi}_j$.
                If $\widehat{\Xi}_j=\mathrm{id}_{L^2}$, then its Weyl symbol is $a_{0,j}(z)=1$. Moreover, the Weyl symbol of $\frp_{i\alpha_j}$ is the Gaussian $a_{1,j}(z)=e^{-\pi \alpha_j z^2}$. Assume $\vartheta_j>0$ and $f\in\cS(\bR)$, then by writing the Fourier transform of $f$ in \eqref{defRtau}, we obtain
                \begin{align}
                    \mathfrak{R}_{\vartheta_j}f(x)=\sinh\vartheta_j^{-1/2}\int_{-\infty}^\infty f(y)e^{-\pi\tanh\vartheta_j^{-1}(x^2+y^2)}e^{2\pi\sinh\vartheta_j^{-1}xy}dy,
                \end{align}
                and thus the kernel of $\mathfrak{R}_{\vartheta_j}$ is
                \begin{equation}
                    k(x,y)=\sinh\vartheta_j^{-1/2}e^{-\pi\tanh\vartheta_j^{-1}(x^2+y^2)}e^{2\pi\sinh\vartheta_j^{-1}xy}.
                \end{equation}
                By \eqref{relak}, we can compute the symbol $a_{2,j}(x,\xi)$ of $\mathfrak{R}_{\vartheta_j}$
                \begin{align}
                    a_{2,j}(x,\xi)&=\sinh\vartheta_j^{-1/2}\int_{-\infty}^\infty k(x+y/2,x-y/2)e^{-2\pi i\xi y}dy\\
                    &=\sinh\vartheta_j^{-1/2}\int_{-\infty}^\infty e^{-2\pi (\tanh\vartheta_j^{-1}-\sinh\vartheta_j^{-1})x^2}e^{-\pi(\tanh\vartheta_j^{-1}+\sinh\vartheta_j^{-1})y^2/2}e^{-2\pi i\xi y}dy.
                \end{align}
                By Proposition \ref{GaussianIntegrals},
                 \begin{align}
                    a_{2,j}(x,\xi)&= c_j\cdot e^{-2\pi (\tanh\vartheta_j^{-1}-\sinh\vartheta_j^{-1})x^2}e^{-2\pi (\tanh\vartheta_j^{-1}+\sinh\vartheta_j^{-1})^{-1}\xi^2},
                \end{align}
                with 
                \begin{align}
                    c_j&=2^{1/2}\sinh\vartheta_j^{-1/2}(\tanh\vartheta_j^{-1}+\sinh\vartheta_j^{-1})^{-1/2}
                    =2^{1/2}(1+\cosh\vartheta_j)^{-1/2}=\cosh(\vartheta_j/2)^{-1}.
                \end{align}
                Since
                \begin{equation}
                    \tanh\vartheta_j^{-1}-\sinh\vartheta_j^{-1}=(\tanh\vartheta_j^{-1}+\sinh\vartheta_j^{-1})^{-1}=\tanh(\vartheta_j/2),
                \end{equation}
                we obtain
                \begin{equation}\label{weylsymbolrtheta}
                    a_{2,j}(x,\xi)=\cosh(\vartheta_j/2)^{-1}e^{-2\pi\tanh(\vartheta_j/2)(x^2+\xi^2)}.
                \end{equation}
                Then, \eqref{symbMp0} follows by choosing $a_j$ as $a_{0,j}$, $a_{1,j}$ and $a_{2,j}$ into \eqref{symbTens1}, according to the form of the corresponding atom $\widehat{\Xi}_j$.

                Next, we compute the operator norm of $\widehat Z\in\Mp_0(d,\bC)$. By Theorem \ref{thm20},
                \begin{equation}
                    \norm{\widehat Z}_{op}\lesssim\norm{a}_{M^{\infty,1}_{1\otimes v_s}},
                \end{equation}
                so we just compute the Sj\"ostrand norm of $a$. For the purpose, we use the Gaussian window $\Phi(z)=e^{-\pi |z|^2}$. 
                \begin{align}
                    \norm{a}_{M^{\infty,1}_{1\otimes v_s}}=\norm{V_\Phi a}_{L^{\infty,1}_{1\otimes v_s}}=\norm{\zeta\mapsto\norm{V_\Phi a(\cdot,\zeta)}_{\infty}v_s(\zeta)}_{1}.
                \end{align}
                A straightforward computation using the Gaussian integral formula of Proposition \ref{GaussianIntegrals} shows that
                \begin{align}
                    \norm{V_\Phi a(\cdot,\zeta)}_{\infty}=\prod_{j=1}^d\cosh(\vartheta_j/2)^{-1}\det(I+V^{-\top}\Sigma V^{-1})^{-1/2}e^{-\pi(I+V^{-\top}\Sigma V^{-1})^{-1}\zeta\cdot\zeta}.
                \end{align}
                It follows that
                \begin{align}
                    \norm{a}_{M^{\infty,1}_{1\otimes v_s}}&=\prod_{j=1}^d\cosh(\vartheta_j/2)^{-1}\det(I+V^{-\top}\Sigma V^{-1})^{-1/2}\int_{\rdd}e^{-\pi (I+V^{-\top}\Sigma V^{-1})^{-1}\zeta\cdot\zeta}v_s(\zeta)d\zeta\\
                    &=\prod_{j=1}^d\cosh(\vartheta_j/2)^{-1}\int_{\rdd}e^{-\pi|\zeta'|^2}v_s((I+V^{-\top}\Sigma V^{-1})^{1/2}\zeta')d\zeta'.
                \end{align}
                Let us write $M=(I+V^{-\top}\Sigma V^{-1})^{1/2}$, and $\sigma_{\mathrm{max}}(M)$ be the largest eigenvalue of $M$. Then,
                \begin{align}
                    v_s(M\zeta')=(1+|M\zeta'|^2)^{s/2}\leq(1+\lambda^2|\zeta'|^2)^{s/2}\leq\sigma_{\mathrm{max}}(M)^sv_s(\zeta'),
                \end{align}
                whence
                 \begin{align}
                    \norm{a}_{M^{\infty,1}_{1\otimes v_s}}&\leq c\prod_{j=1}^d\cosh(\vartheta_j/2)^{-1}\sigma_{\mathrm{max}}(M)^s,
                \end{align}
                where
                \begin{equation}
                    c=\int_{\rdd}e^{-\pi|\zeta'|^2}v_s(\zeta')d\zeta'.
                \end{equation}
                By writing the singular value decomposition of $\Sigma^{1/2}V^{-1}=W\delta W^\top$, we obtain
                \begin{align}
                    (I+V^{-\top}\Sigma V^{-1})^{1/2}&=(I+W\delta^2 W^\top)^{1/2}=W(I+\delta^2)^{1/2}W^\top,
                \end{align}
                whose largest singular value is $(1+\sigma_{\mathrm{max}}(\Sigma^{1/2}V^{-1})^2)^{1/2}$, being $\sigma_{\mathrm{max}}(\Sigma^{1/2}V^{-1})$ the largest diagonal entry of $\delta$, that is the largest singular value of $\Sigma^{1/2}V^{-1}$. 
                This concludes the proof.
            \end{proof}
            
            \begin{corollary}\label{CorollaryS000}
                Every $\widehat{Z}\in\Mp_0(d,\bC)$ is a Weyl-quantized pseudodifferential operator with symbol in $S^0_{0,0}(\rdd)$. Consequently, every $\widehat S\in \Mp_+(d,\bC)$ can be uniquely written as
                \begin{equation}\label{OpwMp}
                    \widehat S=\Op^\w(a)\widehat U=\widehat U\Op^\w(a\circ U),
                \end{equation}
                for $a\in S^0_{0,0}(\rdd)$ as in Theorem \ref{propSymbS000}.
            \end{corollary}

            \begin{remark}
            Corollary \ref{CorollaryS000} allows to frame the theory of $\Mp_+(d,\bC)$ in the context of {\em generalized metaplectic operators}, that are operators in the form $T=\Op^\w(a)\widehat U$, with $\widehat U\in\Mp(d,\bR)$ and $a\in S^0_{0,0}(\rdd)$. The interest in these operators is justified by the fact that the propagator of the evolution problem
            \begin{equation}
                \begin{cases}
                    i\frac{1}{2\pi}\partial_tu(t,x)=\Op^\w(b)u(t,x)+\Op^\w(c)u(t,x), & \text{$t\in\bR$, $x\in\rd$},\\
                u(0,x)=u_0(x),
                \end{cases}
            \end{equation}
            with $u_0\in\cS(\rd)$, $\Op^\w(b)$ Weyl quantization of a real quadratic form $b$, and $c\in S^0_{0,0}(\rdd)$, is a one-parameter subgroup of generalized metaplectic operators.
            The benefits of this perspectives cover different areas of harmonic analyis, with contributions also in applied and computational sciences, due to the characterization of generalized metaplectic operators in terms of their Gabor matrix decay, the infinite dimensional equivalent of sparsity. This approach lies beyond the scope of the present manuscript and will be addressed in future work. The reader may refer to \cite{CGNRgen,Trapasso} as main references.
            \end{remark}
            
            As a straightforward example of the insights this approach could provide, we mention a first boundedness result within modulation spaces for operators in $\Mp_+(d,\bR)$, which follows by writing $\widehat S=\Op^\w(a)\widehat U$ as in \eqref{OpwMp}, and by applying Theorems \ref{thm20}
                and \ref{thmFS} in the order. Item $(i)$ is instead classical, and it has been proven in \cite{CGNRgen} in the framework of generalized metaplectic operators.
            \begin{corollary}\label{CorMpq}
                Let $\widehat S\in\Mp_+(d,\bC)$ have polar decomposition $\widehat S=\widehat U\widehat Z$, $\widehat U\in \Mp(d,\bR)$ and $\widehat Z\in \Mp_0(d,\bC)$. Then:
                \begin{enumerate}[(i)]
                    \item $\widehat S:M^{p}_{v_s}(\rd)\to M^{p}_{v_s}(\rd)$ is bounded for every $1\leq p,q\leq\infty$ and $s\geq0$.
                    \item If the projection $U\in\Sp(d,\bR)$ of $\widehat U$ is upper-block triangular, i.e., $C=O_d$ in \eqref{blockS}, then $\widehat S:M^{p,q}_{v_s}(\rd)\to M^{p,q}_{v_s}(\rd)$ is bounded for every $1\leq p,q\leq\infty$ and $s\geq0$.
                \end{enumerate}
                In every case, under the notations of Theorems \ref{thmFS} and \ref{propSymbS000},
                \begin{equation}\label{combEstimates}
                    \norm{\widehat S f}_{M^{p,q}_{v_s}}\lesssim_{p,q,g,s} c\cdot \norm{f}_{M^{p,q}_{v_s}},
                \end{equation}
                where $g$ is the window used to measure the $M^{p,q}_{v_s}$-norm and, understood $|\det(A)|^{1/p-1/q}=1$ for $p=q$,
                \begin{equation}
                    c=|\det(A)|^{1/p-1/q}{\sigma_{\mathrm{max}}(U)^{2s}}\prod_{j=1}^d\Big(\frac{\sigma_j}{1+\sigma_j^2}\Big)^{-1/2}\cosh(\vartheta_j/2)^{-1}(1+\sigma_{\mathrm{max}}(\Sigma^{1/2}V^{-1}))^{s/2}.
                \end{equation}
            \end{corollary} 
            \begin{proof}
                It follows directly by combining \eqref{normUB} and \eqref{Norm1}.
            \end{proof}
            
            \begin{remark}
            An estimate similar to \eqref{combEstimates} was obtained in \cite{Trapasso} for the one-parameter semigroup of propagators of heat-type equations \eqref{intro.eq1} with complex quadratic Hamiltonians, under additional non-degeneracy assumptions on the geometry of the Hamiltonians. Trapasso obtains the remarkable exponential decay in time
               \begin{equation}\label{TrapassoEq}
                   \norm{\widehat{S_t}}_{op}\lesssim e^{-\varepsilon t},
               \end{equation}
               for a suitable $\varepsilon>0$. 
               Our computations go beyond this limitation and extend to the whole metaplectic semigroup $\Mp_+(d,\bC)$. The reader may observe that \eqref{TrapassoEq} reflects the products of the hyperbolic cosines in \eqref{Norm1}.
               Nevertheless, the relation between \eqref{intro.eq1} and the corresponding propagator, expressed as a one-parameter subsemigroup of $\Mp_+(d,\bC)$, is rather implicit, and further analysis is required to obtain explicit results in this direction. This topic will be treated in future work.
            \end{remark}

            \begin{example}
                Consider the complex Hermite equation
                \begin{equation}\label{ComplexHermite}
                    \begin{cases}
                        \partial_tu(t,x)=(\vartheta+i\mu)\left(\frac 1 {4\pi}\Delta_x-\pi|x|^2\right)u(t,x), & x\in\rd,\, t>0,\\
                        u(0,x)=u_0(x),
                    \end{cases}
                \end{equation}
                where we may assume $u_0\in\cS(\rd)$. The propagator of \eqref{ComplexHermite} has been computed in the unpublished work at Footnote \ref{fn1}, but it is also computable by means of Mehler fomulas. It has polar decomposition
                \begin{equation}\label{ex1}
                    \widehat{S_t}=\widehat{U}(t)\widehat{Z}(t),
                \end{equation}
                where $\widehat U(t)$ is the $\mu t$-fractional power of the Fourier transform. When $\mu t\neq k\pi$ ($k\in \bZ$), it is given by
                \begin{equation}
                    \widehat U(t)f(\xi)=(1-i\cot(\mu t))^{d/2}\int_{\rd}f(x)e^{i\pi \cot(\mu t)(|x|^2+|\xi|^2)}e^{-2\pi ix\xi/\sin(\mu t)}dx,
                \end{equation}
                whereas $\widehat Z(t)$ is
                \begin{equation}
                    \widehat Z(t)=\cosh(\vartheta t)^{-d/2}\int_{\rd}\widehat f(\eta)e^{-\pi\tanh(\vartheta t)(|x|^2+|\eta|^2)}e^{2\pi ix\eta/\cosh(\vartheta t)}d\eta.
                \end{equation}
                If $\mu=0$, only $\widehat Z(t)$ appears in \eqref{ex1}, whence we retrieve the $M^{p,q}_{v_s}$-boundedness for every $1\leq p,q<\infty$ and $s\geq0$, with
                \begin{equation}
                    \norm{\widehat Su_0}_{M^{p,q}_{v_s}}\leq c(s) (1+2\tanh(\vartheta t/2))^{s/2} \cosh(\vartheta t/2)^{-d}\norm{u_0}_{M^{p,q}_{v_s}},
                \end{equation}
                by \eqref{Norm1}, where $V=I_{2d}$. If $\mu\neq0$, for $\mu t\neq k\pi$ ($k\in\bZ$), the projection of $\widehat U(t)$ is not upper-block triangular. By \eqref{normUB} with $p=q$, we obtain the exponential decay of $\norm{\widehat{S_t}}_{op}$:
                \begin{align}
                    \norm{\widehat{S_t}u_0}_{M^p_{v_s}}&=\norm{\widehat U(t)\widehat Z(t)u_0}_{M^p_{v_s}}\lesssim \norm{\widehat Z(t)u_0}_{M^p_{v_s}}\lesssim \cosh(\vartheta t/2)^{-d}(1+2\tanh(\vartheta t))^{s/2}\norm{u_0}_{M^p_{v_s}}\\
                    &\lesssim e^{-d\vartheta t/2}\norm{u_0}_{M^p_{v_s}}.
                \end{align}
        
            \end{example}

            \begin{example}
                Consider the complex heat equation \eqref{HeatEq}, whose solution is expressed in terms of metaplectic opertors as in \eqref{propEqCalore} and the projection is as in \eqref{subsemigroup1}. The polar decomposition of the propagator $\widehat{S_t}=\mathfrak{m}_{P_t}$ is
                \begin{equation}
                   \widehat{S_t}= \widehat U(t)\widehat Z(t)=\mathfrak{m}_{4\pi\alpha tI_d}\mathfrak{m}_{4\pi i\beta  tI_d},
                \end{equation}
                with upper-block triangular projections
                \begin{equation}
                    U(t)=\begin{pmatrix}
                        I_d & -4\pi\alpha t I_d\\
                        O_d & I_d
                    \end{pmatrix}, \qquad
                    Z(t)=\begin{pmatrix}
                        I_d & -4\pi i\beta t I_d\\
                        O_d & I_d
                    \end{pmatrix}=J^{-1}\begin{pmatrix}
                        I_d & O_d\\
                        4\pi i\beta tI_d & O_d
                    \end{pmatrix}J.
                \end{equation}
                The singular values of $U(t)$ are all $\geq1$ and coincide with the largest of them 
                \begin{equation}
                    \sigma_{\mathrm{max}}(U(t))=\sqrt{1+2\pi^2\alpha^2t^2+\sqrt{(1+2\pi^2\alpha^2 t^2)^2-1}}=\sqrt{1+\pi^2\alpha^2t^2}+\pi|\alpha t|.
                \end{equation}
                With the notation of Theorem \ref{propSymbS000},
                \begin{equation}
                    \Sigma^{1/2}V^{-1}={2\pi \beta t}I_{2d}\cdot J,
                \end{equation}
                whose singular values are all equal to $\mu=2\pi \beta t$. Therefore,
                \begin{align}
                    \norm{\widehat{S_t}u_0}_{M^{p,q}_{v_s}}&
                    \lesssim  \sigma_{\mathrm{max}}(U(t))^{2s}
                    \frac{(1+ \sigma_{\mathrm{max}}(U(t))^2)^{d/2}}{\sigma_{\mathrm{max}}(U(t))^{d/2}}
                    \norm{\widehat{Z(t)}u_0}_{M^{p,q}_{v_s}}\\
                    &\lesssim \sigma_{\mathrm{max}}(U(t))^{2s-d/2}(1+ \sigma_{\mathrm{max}}(U(t))^2)^{d/2} (1+4\pi^2\beta^2 t^2)^{s/2} \norm{u_0}_{M^{p,q}_{v_s}}\\
                    &\lesssim (1+|t|)^{3s+d/2}\norm{u_0}_{M^{p,q}_{v_s}}.
                \end{align}
            \end{example}

            \begin{remark}
                Recall that by Remark \ref{Shubinrem}, the modulation space $M^2_{v_s}(\rd)$ coincides with the Shubin-Sobolev space of order $s$, $Q_s(\rd)$. Applied to this particular setting, Corollary \ref{CorMpq} and the following examples, provide therefore time-dependent estimates for the $Q_s$-norm of evolution operators.
            \end{remark}
    
        \subsection{Wigner microlocal analysis}
        
		The solution of the problem \eqref{Schreq} has the form (see \cite{Hormander3}) 
		\begin{equation}\label{eq.solution}
		u(t,\cdot)=e^{-2\pi t\mathrm{Op}^{\mathrm{w}}(a)}u_0=\widehat{S}_t u_0,
		\end{equation}
		where $S_t=e^{-itF} \in \Sp_+(d,\mathbb{C})$ and $F$ is the Hamilton map of $a$, i.e. $F=JQ$. 
        
        Our aim is to study the propagation of concentration in terms of Wigner distributions from a microlocal point of view, by using the $\cA$-wigner wave front set, whose definition is given as follows.
        \begin{definition}\label{def.WFS}
            Let $\cA \in \Mp_+(2d,\mathbb{C})$ with projection $\cA \in \Sp_+(2d,\mathbb{C})$ and let $f \in L^2(\R^d)$. We say that $0 \neq z_0 \notin {\cW\cF}_\cA(f) \subseteq \R^d \times (\R^d \setminus \lbrace 0 \rbrace)$ if there exists a conic neighborhood $\Gamma_{z_0}$ of $z_0$ such that 
            \begin{equation*}
                \int_{\Gamma_{z_0}}v_{s}(z)^2|W_\cA f(z)|^2dz <+\infty, \quad \forall s \geq 0.
            \end{equation*}
        \end{definition}

 \subsubsection{A prototypical example}
 We first investigate the propagation of microlocal concentration of the $(\cA$-)Wigner distribution associated with the evolution 
\begin{equation}\label{eq.Har}
		\begin{cases}
		\frac{1}{2\pi}\partial_t u+Hu=0 \quad \text{in} \ \ \R^+ \times \R^d \\
		u(0,\cdot)=u_0 \in \mathcal{S}'(\R^d),
		\end{cases}
		\end{equation}
        where 
        \begin{equation}\label{HamiltonianHar}
            H=H_R+iH_I=-\frac{1}{8\pi^2}\Delta_{x'}+\frac 1 2 |x'|^2+i\left(-\frac{1}{8\pi^2}\Delta_{x''}+ \frac 1 2 |x''|^2\right),
        \end{equation}
with $x',\xi'\in \mathbb{R}^{d_1}$,  $x'',\xi''\in \mathbb{R}^{d_2}$ and $d_1,d_2 \in \mathbb{N}$ satisfy $d_1+d_2\leq d$. For simplicity, let us assume $d_1+d_2=d$. 

We shall now compute $\widehat S_t^H$ from the Hamiltonian flow $S_t^H=\exp(-itF)$. By \eqref{HamiltonianHar}, we infer
	\begin{equation}
		F=\left(\begin{array}{cc|cc}
			O_{d_1} & O_{d_1\times d_2} & I_{d_1} & O_{d_1\times d_2}\\
			O_{d_2\times d_1} & O_{d_2} & O_{d_2\times d_1} & iI_{d_2}\\
			\hline
			-I_{d_1} & O_{d_1\times d_2} & O_{d_1} & O_{d_1\times d_2}\\
			O_{d_2\times d_1} & -iI_{d_2} & O_{d_2\times d_1} & O_{d_2}
		\end{array}\right).
	\end{equation}
	To facilitate reading, we omit all the dimension-related subscripts in the following. Then,
	\begin{equation}
		S_t^H=\sum_{k=0}^\infty \frac{(-i tF)^k}{k!}=\sum_{k=0}^\infty \frac{(-i t)^k}{k!}F^k. 
	\end{equation}
	By denoting
	\begin{equation}
			L=\left(\begin{array}{cc|cc}
				-I & O & O & O\\
				O & I & O & O\\
				\hline
				O & O & -I & O\\
				O & O & O & I
			\end{array}\right),
		\end{equation}
		and observing that $F^{2k}=L^k$ and $F^{2k+1}=L^kF$, we may split the sum into
		\begin{equation}
			S_t^H=\underbrace{\underset{k\equiv_40}{\sum_{k\geq0}}\frac{(-i t)^k}{k!}}_{=:f_0(t)}I
			+\underbrace{\underset{k\equiv_41}{\sum_{k\geq0}}\frac{(-i t)^k}{k!}}_{=:f_1(t)}F
			+\underbrace{\underset{k\equiv_42}{\sum_{k\geq0}}\frac{(-i t)^k}{k!}}_{=:f_2(t)}L
			+\underbrace{\underset{k\equiv_43}{\sum_{k\geq0}}\frac{(-i t)^k}{k!}}_{=:f_3(t)}LF.
		\end{equation}
		Using that
		\begin{align}
			&f_0(t)+f_2(t)=\cos(t),\\
			&f_0(t)-f_2(t)=\cosh(t),\\
			&f_1(t)+f_3(t)=-i\sin(t),\\
			&f_1(t)-f_3(t)=-i\sinh(t),
		\end{align}
		we obtain the explicit expressions of $f_0,f_1,f_2,f_3$ and conclude
		\begin{equation}\label{StH}
			S_t^H=
			\left(\begin{array}{cc|cc}
				\cosh(t) I & O & -i\sinh(t)I & O\\
				O & \cos(t)I & O &\sin(t) I\\
				\hline
				i\sinh(t) I & O & \cosh(t) I & O\\
				O & -\sin(t)I & O & \cos(t)I
			\end{array}\right),
		\end{equation}
    from which we infer
    \[
    S_t^H=S_{I,t}S_{R,t},
    \]
    with
    \begin{align}
    S_{I,t}=\left(\begin{array}{cc|cc}
				I & O & O & O\\
				O & \cos(t)I & O &\sin(t) I\\
				\hline
				O  & O &  I & O\\
				O & -\sin(t)I & O & \cos(t)I
			\end{array}\right)
    \quad \quad 
     S_{R,t}=\left(\begin{array}{cc|cc}
				\cosh(t)I & O & -i\sinh(t)I & O\\
				O & I & O & O\\
				\hline
				i\sinh(t)I  & O &  \cosh(t)I & O\\
				O & O & O & I
			\end{array}\right).
    \end{align}
    Hence,
    \[
    \widehat S_t^H=\widehat S_{I,t} \widehat S_{R,t}=\widehat{S}_{I,t}\mathrm{Op}^{\w}(a_R(t)), 
    \]
    where, by \eqref{weylsymbolrtheta} (cf. \cite{Hormander3} p. 425),
    \[
    a_R(t)=e^{-2\pi(|x'|^2+|\xi'|^2)\tanh(t/2)}/(\cosh(t/2))^{d_1} \in S_{0,0}^0(\R^{2d}),
    \]

    Our aim is to give the propagation result for \emph{covariant} Wigner distributions. To this end, recall that (see \cite{Cordero_Part_II}) if $\cA \in \Sp(2d,\bR)$ is covariant there exists $k_\cA \in \mathcal{S}'(\R^{2d})$ such that 
\[
W_\cA(f,g)=k_\cA\ast W(f,g)=:Q_{k_\cA}(f,g), \quad f,g \in \mathcal{S}(\mathbb{R}^d).
\]
Moreover, if we define $k_\cA(t)(z):=k_\cA(S_t z)$, we denote by $W_{\cA_t}(f,g)=Q_{k_\cA(t)}(f,g)=k_{\cA_t}\ast W(f,g)$ the corresponding Wigner distribution (cf. \cite{Cordero_Part_II}). Under these notations we state the next theorem.
\begin{theorem}\label{theo.H}
    Let $\widehat \cA \in \Mp(2d,\R)$ be covariant with projection $\cA \in \Sp(2d,\R)$ and let $u$ be the solution of \eqref{eq.Har}.
    Then, for any fixed $t>0$, we have
    \begin{equation}
        \cW\cF_\cA(u(t,\cdot))\subseteq S_{I,t}\cW\cF_{\cA_t}(u_0).
    \end{equation}
\end{theorem}
\begin{proof}
    Let $t>0$ be fixed.  We have, see Footnote \ref{footoneboh},
    \begin{equation}\label{eq.awignerH}
        W_\cA(u(t,\cdot))(z)=W_\cA(\widehat S_{I,t} \mathrm{Op}^{\mathrm{w}}(a_R(t))u_0)(z)=W_{\cA_t}(\mathrm{Op}^{\mathrm{w}}(a_R(t))u_0)(S_{I,t}^{-1}z), \ z \in \R^{2d}.
    \end{equation}
    Hence, since $a_R(t) \in S^{0}_{0,0}(\R^{2d})$, we obtain 
    \begin{equation*}
        W_\cA(u(t,\cdot)(z)=\mathrm{Op}^{\mathrm{w}}(c(t))W_{\cA_t}u_0(S_{I,t}^{-1}z), \quad z \in \R^{2d},
    \end{equation*}
    where $c(t) \in S_{0,0}^0(\R^{4d})$. Therefore, the proof follows as in  \cite[Theorem 1.6]{Cordero_Part_I} (cf. \cite[Theorem 7.7]{Cordero_Part_II}).
    \end{proof}

    \begin{remark}
        As a future direction, we stress that Definition \ref{def.WFS} can be extended to encode the microlocal interactions between two states $f$ and $g$. This was done for metaplectic Wigner distributions $W_\cA$ with $\widehat\cA\in\Mp(2d,\bR)$ in the recent unpublished contribution at Footnote \ref{footoneboh}. In particular, by defining the (cross-)Wigner wave front sets $\cW\cF_{\mathcal A}(f,g)$ one can investigate results analogous to those obtained in the aforementioned manusctipt.
    \end{remark}

    As far as the boundedness on modulation spaces is concerned, a straightforward application of Corollary \ref{CorMpq} to the propagator of \eqref{eq.Har} yields the following consequence.

    \begin{corollary}
        Under the notation of this Section and the assumption $d_1+d_2= d$, assume $u_0\in M^{p,q}_{v_s}(\rd)$ for $1\leq p,q\leq\infty$ and $s\geq0$.
        \begin{enumerate}[(i)]
            \item If $p=q$, then $u(t,\cdot)\in M^{p}_{v_s}(\rd)$ for every $t>0$.
            \item If $p\neq q$ and $d_2=0$, then $u(t,\cdot)\in M^{p,q}_{v_s}(\rd)$.
        \end{enumerate}
    \end{corollary}
    \begin{proof}
        We just need to justify $(ii)$, which follows simply by observing that under the assumption $d_2=0$, the propagator $e^{-2\pi tH}$ of \eqref{Schreq} has polar decomposition $e^{-2\pi tH}=\widehat S_{I,t}\Op^\w(a_R(t))$ with symbol $a_R(t)\in\cS(\rdd)$ and $S_{I,t}=I$, which is upper-block triangular. The assertion follows then by Corollary \ref{CorMpq}.
    \end{proof}
    
    \subsubsection{The general case}
    We now come back to the general case, that is we consider the propagation of $(\cA)$-Wigner distributions associated with evolutions governed by \eqref{Schreq}. 
    
    Corollary \ref{CorollaryS000} makes now straightforward to prove Theorem \ref{theo.evo} concerning the propagation of concentration for covariant (real) Wigner distributions associated with evolution equations of the form \eqref{Schreq}.
    
    \begin{theorem}\label{theo.evo}
    Let $W_\cA$ be covariant and $\cA \in \Sp(2d,\R)$ be the corresponding projection. Let $u(t,\cdot)=e^{-2\pi t\mathrm{Op}^{\mathrm{w}}(a)}u_0$ be the solution of \eqref{Schreq}. Let $e^{-2\pi t\mathrm{Op}^{\mathrm{w}}(a)}=\widehat U(t)\Op^\w(b(t))$ be its polar decomposition given in Theorem \ref{intro.thm2}. Then, for any fixed $t>0$, we have
    \begin{equation}\label{eq.finres}
        \cW\cF_\cA(u(t,\cdot))\subseteq U(t) \cW\cF_{\cA_t}(u_0),
    \end{equation}
        where $W_{\cA_t}$ is a suitable covariant metaplectic Wigner distribution.
    \end{theorem}
    
    \begin{proof}
    In the first place, by Corollary \ref{CorollaryS000} we may write (recall \eqref{eq.solution})
    \begin{equation*}
        u(t,\cdot)=e^{-2\pi t\mathrm{Op}^{\mathrm{w}}(a)}u_0=\widehat U(t)\mathrm{Op}^{\mathrm{w}}(b(t)) u_0,
    \end{equation*}
    for $\widehat U(t) \in \Mp(d,\R)$ and $b(t) \in S_{0,0}^0(\R^{2d})$. 
    Hence, by arguing as in \eqref{eq.awignerH}, we have 
    \[
    W_\cA(u(t,\cdot))(z)=W_\cA(\widehat U(t)\mathrm{Op}^{\mathrm{w}}(b(t))u_0)(z)=W_{\cA_t}(\mathrm{Op}^{\mathrm{w}}(b(t)u_0))(U(t)^{-1}z), \ z \in \R^{2d},
    \]
    which yields \eqref{eq.finres}, and once again the proof follows by a microlocalization argument as in \cite[Theorem 1.6]{Cordero_Part_I}.
        
    \end{proof}

\section*{Acknowledgments}
The  three authors have been supported by the Gruppo Nazionale per l’Analisi Matematica, la Probabilità e le loro Applicazioni (GNAMPA) of the Istituto Nazionale di Alta Matematica (INdAM). Gianluca Giacchi has been supported by the SNSF starting grant ``Multiresolution methods for unstructured data” (TMSGI2 211684).

\section*{Conflict of Interest}
The authors declare that they have no conflict of interest.

\section*{Data Availability}
This manuscript does not report on or involve the use of any datasets. 
Accordingly, no data are available.

\bibliographystyle{abbrv}
\bibliography{bibliography2.bib}

\end{document}